\let\ifAllProof\iffalse\let\ifFullProof\iftrue\let\ifPersonalNote\iffalse

\unless\ifdefined\IsMainDocument

\documentclass[12pt]{amsart}
\usepackage{amsmath,amssymb,iftex,cancel,xy,xcolor,verbatim}
\usepackage{theoremref}
\usepackage[normalem]{ulem}
\usepackage[margin=1in]{geometry}

\xyoption{all}

\newtheorem{thm}{Theorem}[section]
\newtheorem{lm}[thm]{Lemma}
\newtheorem{prop}[thm]{Proposition}
\newtheorem{cor}[thm]{Corollary}
\theoremstyle{definition}
\newtheorem{defn}[thm]{Definition}
\newtheorem{remark}[thm]{Remark}
\newtheorem{example}[thm]{Example}

\numberwithin{equation}{section}

\title{Hermitian Maass lift for General Level}
\date{}
\author{Lawrence (An Hoa) Vu}

\ifPersonalNote
	\ifPDFTeX
		\usepackage{mdframed}
		\newmdtheoremenv[
		    linewidth=2pt,
		    leftmargin=0pt,
		    innerleftmargin=0.4em,
		    rightmargin=0pt,
		    innerrightmargin=0.4em,
		    innertopmargin=-5pt,
		    innerbottommargin=3pt,
		    splittopskip=\topskip,
		    splitbottomskip=0.3\topskip,
		    skipabove=0.6\topsep
		]{pnote}{Personal note}
	\else
		\definecolor{OliveGreen}{rgb}{0,0.6,0}
		\newcounter{personaln}
		\newenvironment{pnote}{\color{OliveGreen}
		\stepcounter{personaln}
		\par\bigskip\noindent{\bfseries Personal note \arabic{personaln}:}
		\par\medskip}{\par\medskip}
	\fi
\else	
	\let\pnote\comment
	\let\endpnote\endcomment
\fi

\allowdisplaybreaks

\ifPDFTeX
\else
	\usepackage{yfonts}
	\renewcommand{\mathfrak}{\textfrak}
\fi

\newcommand{\LHS}{\text{LHS}}

\newcommand{\UHP}{\mathfrak{H}}
\newcommand{\C}{\mathbb{C}}
\newcommand{\R}{\mathbb{R}}
\newcommand{\Q}{\mathbb{Q}}
\newcommand{\Z}{\mathbb{Z}}

\newcommand{\term}[1]{\emph{#1}}

\newcommand{\Maass}{Maa\ss\ }

\newcommand{\Mat}[1]{M_{#1}}
\newcommand{\GGm}{\mathsf{G_m}}
\newcommand{\GL}[1]{\mathsf{GL}_{#1}}
\newcommand{\SL}[1]{\mathsf{SL}_{#1}}
\newcommand{\GU}[1]{\mathsf{GU}(#1,#1)}
\newcommand{\U}[1]{\mathsf{U}(#1,#1)}
\newcommand{\SU}[1]{\mathsf{SU}(#1,#1)}
\newcommand{\HerMat}[1]{\mathbb{S}_{#1}} \newcommand{\Res}{\text{Res}} 

\newcommand{\GS}{\mathbf{G}} \newcommand{\aDNu}{a_u} \newcommand{\aDNw}{a_w} 

\newcommand{\MF}{\mathfrak{M}}
\newcommand{\CF}{\mathfrak{S}}
\newcommand{\JF}{\mathfrak{J}}
\newcommand{\MFDp}{\MF_{k-1}^+(D, \chi)}
\newcommand{\MFDNp}{\MF_{k-1}^+(DN, \chi)}
\newcommand{\MFDN}{\MF_{k-1}(DN, \chi)}

\newcommand{\Matx}[1]{\begin{pmatrix} #1 \end{pmatrix}}

\newcommand{\ConjTran}[1]{#1^*} \newcommand{\ConjTranInv}[1]{\widehat{#1}} 

\newcommand{\epx}[1]{\mathsf{e}\left[#1\right]}
\newcommand{\sign}[1]{\mathsf{sign}\left[#1\right]}
\newcommand{\cc}[1]{\overline{#1}}

\newcommand{\OK}{\mathfrak{o}_K} \newcommand{\DK}{\mathfrak{d}_K} \newcommand{\Mabcd}{\Matx{a & b\\c & d}}
\newcommand{\Mxyzt}{\Matx{x & y\\z & t}}
\newcommand{\val}{\mathsf{val}}

\newcommand{\parens}[1]{\left( #1 \right)}
\newcommand{\braces}[1]{\left\{ #1 \right\}}
\newcommand{\suchthat}{\;\vline\;}
\newcommand{\tif}{\text{ if }}
\newcommand{\tand}{\text{ and }}
\newcommand{\twhere}{\text{ where }}
\newcommand{\totherwise}{\text{ otherwise }}
\newcommand{\tsince}{\text{ since }}
\newcommand{\teither}{\text{ either }}
\newcommand{\tforsome}{\text{ for some }}
\newcommand{\tor}{\text{ or }}

\newcommand{\Gm}{\Gamma}
\renewcommand{\a}{\alpha}
\renewcommand{\b}{\beta}
\newcommand{\gm}{\gamma}
\newcommand{\ld}{\lambda}
\newcommand{\eps}{\varepsilon}
\renewcommand{\theta}{\vartheta}
\renewcommand{\phi}{\varphi}

\newcommand{\dv}{\;|\;}
\newcommand{\ndv}{\;\nmid\;}
\newcommand{\smD}{\sqrt{-D}}
\newcommand{\sD}{\sqrt{D}}
\renewcommand{\gcd}{}
\newcommand{\jcb}[2]{(#1 \;|\; #2)}
\newcommand{\Tr}{\mathsf{Tr}}
\newcommand{\Nm}[1]{|#1|^2}  \newcommand{\Id}{\text{Id}}
\newcommand{\Gal}{\text{Gal}}

\renewcommand{\labelenumi}{(\roman{enumi})}

\begin{document}

\fi 
\let\IsMainDocument\relax

\begin{abstract}
For an imaginary quadratic field $K$ of discriminant $-D$, let $\chi = \chi_K$ be the associated quadratic character. We will show that the space of special hermitian Jacobi forms of level $N$ is isomorphic to the space of plus forms of level $DN$ and nebentypus $\chi$ (the hermitian analogue of Kohnen's plus space) for any integer $N$ prime to $D$. This generalizes the results of Krieg from $N = 1$ to arbitrary level. Combining this isomorphism with the recent work of Berger and Klosin and a modification of Ikeda's construction we prove the existence of a lift from the space of elliptic modular forms to the space of hermitian modular forms of level $N$ which can be viewed as a generalization of the classical hermitian \Maass lift to arbitrary level.
\end{abstract}

\maketitle

\section{Introduction}

\let\ifAllProof\iffalse\let\ifFullProof\iftrue\let\ifPersonalNote\iffalse

\unless\ifdefined\IsMainDocument

\documentclass[12pt]{amsart}
\usepackage{amsmath,amssymb,iftex,cancel,xy,xcolor,verbatim}
\usepackage{theoremref}
\usepackage[normalem]{ulem}
\usepackage[margin=1in]{geometry}

\xyoption{all}

\newtheorem{thm}{Theorem}[section]
\newtheorem{lm}[thm]{Lemma}
\newtheorem{prop}[thm]{Proposition}
\newtheorem{cor}[thm]{Corollary}
\theoremstyle{definition}
\newtheorem{defn}[thm]{Definition}
\newtheorem{remark}[thm]{Remark}
\newtheorem{example}[thm]{Example}

\numberwithin{equation}{section}

\title{Hermitian Maass lift for General Level}
\date{}
\author{Lawrence (An Hoa) Vu}

\ifPersonalNote
	\ifPDFTeX
		\usepackage{mdframed}
		\newmdtheoremenv[
		    linewidth=2pt,
		    leftmargin=0pt,
		    innerleftmargin=0.4em,
		    rightmargin=0pt,
		    innerrightmargin=0.4em,
		    innertopmargin=-5pt,
		    innerbottommargin=3pt,
		    splittopskip=\topskip,
		    splitbottomskip=0.3\topskip,
		    skipabove=0.6\topsep
		]{pnote}{Personal note}
	\else
		\definecolor{OliveGreen}{rgb}{0,0.6,0}
		\newcounter{personaln}
		\newenvironment{pnote}{\color{OliveGreen}
		\stepcounter{personaln}
		\par\bigskip\noindent{\bfseries Personal note \arabic{personaln}:}
		\par\medskip}{\par\medskip}
	\fi
\else	
	\let\pnote\comment
	\let\endpnote\endcomment
\fi

\allowdisplaybreaks

\ifPDFTeX
\else
	\usepackage{yfonts}
	\renewcommand{\mathfrak}{\textfrak}
\fi

\newcommand{\LHS}{\text{LHS}}

\newcommand{\UHP}{\mathfrak{H}}
\newcommand{\C}{\mathbb{C}}
\newcommand{\R}{\mathbb{R}}
\newcommand{\Q}{\mathbb{Q}}
\newcommand{\Z}{\mathbb{Z}}

\newcommand{\term}[1]{\emph{#1}}

\newcommand{\Maass}{Maa\ss\ }

\newcommand{\Mat}[1]{M_{#1}}
\newcommand{\GGm}{\mathsf{G_m}}
\newcommand{\GL}[1]{\mathsf{GL}_{#1}}
\newcommand{\SL}[1]{\mathsf{SL}_{#1}}
\newcommand{\GU}[1]{\mathsf{GU}(#1,#1)}
\newcommand{\U}[1]{\mathsf{U}(#1,#1)}
\newcommand{\SU}[1]{\mathsf{SU}(#1,#1)}
\newcommand{\HerMat}[1]{\mathbb{S}_{#1}} \newcommand{\Res}{\text{Res}} 

\newcommand{\GS}{\mathbf{G}} \newcommand{\aDNu}{a_u} \newcommand{\aDNw}{a_w} 

\newcommand{\MF}{\mathfrak{M}}
\newcommand{\CF}{\mathfrak{S}}
\newcommand{\JF}{\mathfrak{J}}
\newcommand{\MFDp}{\MF_{k-1}^+(D, \chi)}
\newcommand{\MFDNp}{\MF_{k-1}^+(DN, \chi)}
\newcommand{\MFDN}{\MF_{k-1}(DN, \chi)}

\newcommand{\Matx}[1]{\begin{pmatrix} #1 \end{pmatrix}}

\newcommand{\ConjTran}[1]{#1^*} \newcommand{\ConjTranInv}[1]{\widehat{#1}} 

\newcommand{\epx}[1]{\mathsf{e}\left[#1\right]}
\newcommand{\sign}[1]{\mathsf{sign}\left[#1\right]}
\newcommand{\cc}[1]{\overline{#1}}

\newcommand{\OK}{\mathfrak{o}_K} \newcommand{\DK}{\mathfrak{d}_K} \newcommand{\Mabcd}{\Matx{a & b\\c & d}}
\newcommand{\Mxyzt}{\Matx{x & y\\z & t}}
\newcommand{\val}{\mathsf{val}}

\newcommand{\parens}[1]{\left( #1 \right)}
\newcommand{\braces}[1]{\left\{ #1 \right\}}
\newcommand{\suchthat}{\;\vline\;}
\newcommand{\tif}{\text{ if }}
\newcommand{\tand}{\text{ and }}
\newcommand{\twhere}{\text{ where }}
\newcommand{\totherwise}{\text{ otherwise }}
\newcommand{\tsince}{\text{ since }}
\newcommand{\teither}{\text{ either }}
\newcommand{\tforsome}{\text{ for some }}
\newcommand{\tor}{\text{ or }}

\newcommand{\Gm}{\Gamma}
\renewcommand{\a}{\alpha}
\renewcommand{\b}{\beta}
\newcommand{\gm}{\gamma}
\newcommand{\ld}{\lambda}
\newcommand{\eps}{\varepsilon}
\renewcommand{\theta}{\vartheta}
\renewcommand{\phi}{\varphi}

\newcommand{\dv}{\;|\;}
\newcommand{\ndv}{\;\nmid\;}
\newcommand{\smD}{\sqrt{-D}}
\newcommand{\sD}{\sqrt{D}}
\renewcommand{\gcd}{}
\newcommand{\jcb}[2]{(#1 \;|\; #2)}
\newcommand{\Tr}{\mathsf{Tr}}
\newcommand{\Nm}[1]{|#1|^2}  \newcommand{\Id}{\text{Id}}
\newcommand{\Gal}{\text{Gal}}

\renewcommand{\labelenumi}{(\roman{enumi})}

\begin{document}

\fi 
The Saito-Kurokawa lift, established by the work of many authors (classically by Maass \cite{Maass1979, Maass1979II, Maass1979III}, Andrianov \cite{Andrianov1979}, Eichler-Zagier \cite{EichlerZagier1985} for the full level and reinterpreted in representation theoretic language by Piatetski-Shapiro \cite{PiatetskiShapiro1983}), has been of interest and importance in number theory, for instance, in proving part of the Bloch-Kato conjecture by Skinner-Urban \cite{SkinnerUrban2006}, providing evidence for the Bloch-Kato conjecture by Brown \cite{Brown2007} and Agarwal-Brown \cite{AgarwalBrown2014}. For these applications, one needs a generalization of the Saito-Kurokawa lift to higher level, which was established by the work of Manickam-Ramakrishnan-Vasudevan \cite{ManickamRamakrishnan2000, ManickamRamakrishnan2002, ManickamRamakrishnanVasudevan1993}, Ibukiyama \cite{Ibukiyama2012}, Agarwal-Brown \cite{AgarwalBrown2015} and Schmidt \cite{Schmidt2007} for square-free level. (Saito-Kurokawa lift is known to exist for all level.)

The Hermitian analogue of Saito-Kurokawa lift is a lifting of elliptic modular forms to hermitian modular forms of degree 2, typically referred to as the \Maass lift in the literature, has been studied since 1980 with the paper of Kojima \cite{Kojima1982}. Such a lift is of arithmetic interest, for example, in providing evidence for the Bloch-Kato conjecture by Klosin \cite{Klosin2009}, \cite{Klosin2015} and for the $p$-adic theory. For these arithmetic applications, there is a need for the generalization of the Hermitian \Maass lift. 

As with the classical Saito-Kurokawa lift, the first construction of the Hermitian \Maass lift by Kojima was a composition of the three lifts: Let $K = \Q(i)$ and assume that $k$ is divisible by 4. Then we have
\begin{equation}
\xymatrix{
\CF_{k-1}(4, \chi) \ar[r] & \CF_{k-1}^+(4, \chi) \ar[r] & \JF_{k,1}^*(1) \ar[r] & \MF_{k,2}(1)
}
\end{equation}
where
\begin{enumerate}
\item the lift from $\CF_{k-1}(4, \chi)$ to $\CF_{k-1}^+(4, \chi)$ is an analogue of Shimura--Shitani isomorphism (but it is not an isomorphism here though and unlike Shitani's map, this map is not obtained by certain cycle integrals) and $\CF_{k-1}^+(4, \chi)$ is an analogue of Kohnen's plus space \cite{Kohnen1980};
\item the lift from $\CF_{k-1}^+(4, \chi)$ to the space of special Hermitian Jacobi forms $\JF_{k,1}^*(1)$; and finally,
\item the Hermitian analogue of the original \Maass lift from $\JF_{k,1}^*(1)$ to $\MF_{k,2}(1)$.
\end{enumerate}
See section \ref{sec:hermitian_modular_forms} for our definition and notation of the spaces and groups mentioned.

In \cite{Krieg1991}, Krieg generalized Kojima's result to general imaginary quadratic fields $K$. He established the maps
\begin{equation}
\xymatrix{
\MF_{k-1}^+(D, \chi) \ar[r] & \JF_{k,1}^*(1) \ar[r] & \MF_{k,2}(1)
}
\end{equation}
Moreover, Krieg defined the Hermitian \Maass space $\MF_{k,2}^*(1) \subset \MF_{k,2}(1)$ and he showed that the image of the second map is $\MF_{k,2}^*(1)$. Also, $\JF_{k,1}^*(1) \rightarrow \MFDp$ is an isomorphism. As for the map $\CF_{k-1}(D, \chi)$ to $\CF_{k-1}^+(D, \chi)$, Krieg remarked that if the discriminant $D$ is prime, the space of cusp forms $\CF_{k-1}(D, \chi)$ has a basis consisting of primitive form $f_1,...,f_a,f_{a+1},...,f_{a+b}$ with $f_i^\rho = f_i$ for all $i = a+1,...,a+b$ and the forms $f_i - f_i^{\rho}$ is a basis for cusp forms in $\MFDp$. Here, $f^\rho(\tau) = \cc{f(-\cc{\tau})}$ is the form obtained by complex conjugating all Fourier coefficients of $f$. Ikeda \cite{Ikeda2008} gives a generalization of this construction to arbitrary discriminant $D$, which we modify to accommodate forms of higher level (c.f. section \ref{sec:to_plus_forms}). We remark that for level $N = 1$, there are alternative approaches to the Hermitian Saito-Kurokawa lift via the theory of compatible Eisenstein series by Ikeda \cite{Ikeda2008}, via Imai's Converse Theorem \cite{Imai1980} recently by Matthes \cite{Matthes2017} (but only for $K = \Q(i)$), and theta lifting by Atobe \cite{Atobe2013}.

In this paper, we generalize this result to higher level: Let $K$ be an imaginary quadratic field of discriminant $-D$, we describe the (first two) maps
\begin{equation}
\xymatrix{
\MF_{k-1}(DN, \chi) \ar[r] & \MFDNp \ar[r] & \JF_{k,1}^*(N) \ar[r] & \MF_{k,2}(N)
}
\label{diagram:hermitian_sk_lift_composition}
\end{equation}
For general level $N$ such that $\gcd(D, N) = 1$, Berger and Klosin, in their recent work \cite{Klosin2018}, generalized the definition of Krieg's \Maass space $\MF_{k,2}^*(1)$ to that of higher level $\MF_{k,2}^*(N)$. Using a modification of Kojima and Krieg's result and Ibukiyama's construction of the classical \Maass lift for general level, they then constructed the Hermitian analogue of the \Maass lift $\JF_{k,1}^*(N) \rightarrow \MF_{k,2}(N)$ i.e. the last map in the above diagram \eqref{diagram:hermitian_sk_lift_composition} and showed that its image is the \Maass space $\MF_{k,2}^*(N)$. When $D$ is prime, they defined $\MFDNp$ and constructed a map $\JF_{k,1}^*(N) \rightarrow \MFDNp$ which they showed is injective and that it is surjective onto the space of $p$-old cusps forms when $N = p$ is a split prime in $K$.

In this article, we generalize Krieg's method to higher level to provide the surjectivity for the map in the second main result of Berger and Klosin \cite{Klosin2018} i.e. to show that in fact, $\JF_{k,1}^*(N) \rightarrow \MFDNp$ is always an isomorphism, under appropriate modification of the space $\MFDNp$ defined in \cite{Klosin2018}.

The main technical challenge to overcome is the lack of a simple set of generators for the congruence subgroup $\Gamma_0(N)$ (even though there are known generators given in \cite{Chuman1973}), typically used to reduce the verification of the functional equation satisfied by a Jacobi form to a finite set of functional equations that can be verified by analytical mean. A minor challenge is to avoid the computation of the inverse of a certain matrix $M_{u,v}(\sigma)$ which Krieg \cite{Krieg1991} and Berger-Klosin \cite{Klosin2018} used.
(In Krieg's case, the matrix involved $M_{u,v}(\sigma)$ is not too complicated since he only needs the one for $\sigma = \Matx{0 & -1 \\ 1 & 0}$, the reflection matrix. In Berger-Klosin's case, they need to exploit a relationship between the matrix $M_{\pi u, \pi v}(\sigma)$ and $M_{u, v}(\sigma)$; c.f. Section 3.2 of \cite{Klosin2018} for details.)

For the outline of the article, we recall the definition of Hermitian modular forms, the Hermitian \Maass space and Hermitian Jacobi forms in section \ref{sec:hermitian_modular_forms} as well as review the relevant results of Berger and Klosin \cite{Klosin2018} and define our lift.
Then we proceed to generalize Krieg's work, namely giving a characterization of the space of plus forms in section \ref{sec:characterization_plus_forms} and obtain an arithmetic criterion for the lift we define to be a Hermitian Jacobi form in section \ref{sec:arithmetic_condition}.
In section \ref{sec:proof_hermitian_sk_lift}, we verify that the arithmetic criterion is true; thus, establishing the main result of the article (\thref{thm:jacobi_forms_plus_form_isomorphism}). We also explain our modification of Ikeda's construction to get a surjective map $\MF_{k-1}(DN, \chi) \rightarrow \MFDNp$, thus completing the Hermitian Saito-Kurokawa lift for general level depicted in \eqref{diagram:hermitian_sk_lift_composition}. In the last section \ref{sec:hecke_equivariant}, we show that the \Maass space is stable under Hecke operator $T_p$ for inert prime $p$ in $K$.

\section{Hermitian modular forms, \Maass space and Jacobi forms}
\label{sec:hermitian_modular_forms}

\let\ifAllProof\iffalse\let\ifFullProof\iftrue\let\ifPersonalNote\iffalse

\unless\ifdefined\IsMainDocument

\documentclass[12pt]{amsart}
\usepackage{amsmath,amssymb,iftex,cancel,xy,xcolor,verbatim}
\usepackage{theoremref}
\usepackage[normalem]{ulem}
\usepackage[margin=1in]{geometry}

\xyoption{all}

\newtheorem{thm}{Theorem}[section]
\newtheorem{lm}[thm]{Lemma}
\newtheorem{prop}[thm]{Proposition}
\newtheorem{cor}[thm]{Corollary}
\theoremstyle{definition}
\newtheorem{defn}[thm]{Definition}
\newtheorem{remark}[thm]{Remark}
\newtheorem{example}[thm]{Example}

\numberwithin{equation}{section}

\title{Hermitian Maass lift for General Level}
\date{}
\author{Lawrence (An Hoa) Vu}

\ifPersonalNote
	\ifPDFTeX
		\usepackage{mdframed}
		\newmdtheoremenv[
		    linewidth=2pt,
		    leftmargin=0pt,
		    innerleftmargin=0.4em,
		    rightmargin=0pt,
		    innerrightmargin=0.4em,
		    innertopmargin=-5pt,
		    innerbottommargin=3pt,
		    splittopskip=\topskip,
		    splitbottomskip=0.3\topskip,
		    skipabove=0.6\topsep
		]{pnote}{Personal note}
	\else
		\definecolor{OliveGreen}{rgb}{0,0.6,0}
		\newcounter{personaln}
		\newenvironment{pnote}{\color{OliveGreen}
		\stepcounter{personaln}
		\par\bigskip\noindent{\bfseries Personal note \arabic{personaln}:}
		\par\medskip}{\par\medskip}
	\fi
\else	
	\let\pnote\comment
	\let\endpnote\endcomment
\fi

\allowdisplaybreaks

\ifPDFTeX
\else
	\usepackage{yfonts}
	\renewcommand{\mathfrak}{\textfrak}
\fi

\newcommand{\LHS}{\text{LHS}}

\newcommand{\UHP}{\mathfrak{H}}
\newcommand{\C}{\mathbb{C}}
\newcommand{\R}{\mathbb{R}}
\newcommand{\Q}{\mathbb{Q}}
\newcommand{\Z}{\mathbb{Z}}

\newcommand{\term}[1]{\emph{#1}}

\newcommand{\Maass}{Maa\ss\ }

\newcommand{\Mat}[1]{M_{#1}}
\newcommand{\GGm}{\mathsf{G_m}}
\newcommand{\GL}[1]{\mathsf{GL}_{#1}}
\newcommand{\SL}[1]{\mathsf{SL}_{#1}}
\newcommand{\GU}[1]{\mathsf{GU}(#1,#1)}
\newcommand{\U}[1]{\mathsf{U}(#1,#1)}
\newcommand{\SU}[1]{\mathsf{SU}(#1,#1)}
\newcommand{\HerMat}[1]{\mathbb{S}_{#1}} \newcommand{\Res}{\text{Res}} 

\newcommand{\GS}{\mathbf{G}} \newcommand{\aDNu}{a_u} \newcommand{\aDNw}{a_w} 

\newcommand{\MF}{\mathfrak{M}}
\newcommand{\CF}{\mathfrak{S}}
\newcommand{\JF}{\mathfrak{J}}
\newcommand{\MFDp}{\MF_{k-1}^+(D, \chi)}
\newcommand{\MFDNp}{\MF_{k-1}^+(DN, \chi)}
\newcommand{\MFDN}{\MF_{k-1}(DN, \chi)}

\newcommand{\Matx}[1]{\begin{pmatrix} #1 \end{pmatrix}}

\newcommand{\ConjTran}[1]{#1^*} \newcommand{\ConjTranInv}[1]{\widehat{#1}} 

\newcommand{\epx}[1]{\mathsf{e}\left[#1\right]}
\newcommand{\sign}[1]{\mathsf{sign}\left[#1\right]}
\newcommand{\cc}[1]{\overline{#1}}

\newcommand{\OK}{\mathfrak{o}_K} \newcommand{\DK}{\mathfrak{d}_K} \newcommand{\Mabcd}{\Matx{a & b\\c & d}}
\newcommand{\Mxyzt}{\Matx{x & y\\z & t}}
\newcommand{\val}{\mathsf{val}}

\newcommand{\parens}[1]{\left( #1 \right)}
\newcommand{\braces}[1]{\left\{ #1 \right\}}
\newcommand{\suchthat}{\;\vline\;}
\newcommand{\tif}{\text{ if }}
\newcommand{\tand}{\text{ and }}
\newcommand{\twhere}{\text{ where }}
\newcommand{\totherwise}{\text{ otherwise }}
\newcommand{\tsince}{\text{ since }}
\newcommand{\teither}{\text{ either }}
\newcommand{\tforsome}{\text{ for some }}
\newcommand{\tor}{\text{ or }}

\newcommand{\Gm}{\Gamma}
\renewcommand{\a}{\alpha}
\renewcommand{\b}{\beta}
\newcommand{\gm}{\gamma}
\newcommand{\ld}{\lambda}
\newcommand{\eps}{\varepsilon}
\renewcommand{\theta}{\vartheta}
\renewcommand{\phi}{\varphi}

\newcommand{\dv}{\;|\;}
\newcommand{\ndv}{\;\nmid\;}
\newcommand{\smD}{\sqrt{-D}}
\newcommand{\sD}{\sqrt{D}}
\renewcommand{\gcd}{}
\newcommand{\jcb}[2]{(#1 \;|\; #2)}
\newcommand{\Tr}{\mathsf{Tr}}
\newcommand{\Nm}[1]{|#1|^2}  \newcommand{\Id}{\text{Id}}
\newcommand{\Gal}{\text{Gal}}

\renewcommand{\labelenumi}{(\roman{enumi})}

\begin{document}

\fi 
Recall that we fix an imaginary quadratic field $K = \Q(\sqrt{-D})$ of discriminant $-D = -D_K$. We fix an embedding $K \rightarrow \C$ to treat $K$ as a subfield of $\C$ and identify the non-trivial automorphism of $\Gal(K/\Q)$ as the restriction of complex conjugation of $\C$ to $K$. Thus, if $x \in K$ then its norm $\Nm{x} = x \cc{x} \in \Q$ makes sense. We denote by $\OK$ the ring of integers of $K$ and by $\DK := \frac{\OK}{\sqrt{-D}}$ the inverse different of $K$.

Throughout the article, we shall assume the weight $k$ is divisible by $|\OK^\times|$, the number of units in $K$. In particular, $k$ is even. (When $k$ is not divisible by $|\OK^\times|$, Berger and Klosin showed that the \Maass space is zero.) In this section, we recall the notion of Hermitian modular form, Hermitian Jacobi form and the Hermitian \Maass space as well as the main results of \cite{Klosin2018} relating to these forms.

\subsection{Hermitian modular forms}

Let $R$ be a ring and $S$ an $R$-algebra equipped with an involution $\iota : S \rightarrow S$. If $A$ is any $R$-algebra then $S \otimes_R A$ has an obvious induced involution $\iota \otimes \Id$.
For each matrix $g \in M_n(S \otimes_R A)$, we denote $g^{\iota}$ the matrix $(\iota \otimes \Id) (g_{ij})$ obtained by applying the induced involution to every entry of $g$. We also denote $g^t$ the transpose of $g$; and when $\iota$ could be inferred without ambiguity, $\ConjTran{g}$ is used to denote $(g^{\iota})^t$.

If the triple $[R, S, \iota]$ is reasonable (namely, so that the subsequently mentioned Weil restriction of scalars $\Res_{S/R} \GL{2n/S}$ exists, for instance, $S$ is a finitely generated projective $R$-module), we have the unitary similitude \emph{algebraic} groups defined over $R$ and its subgroups:
\begin{align*}
\GU{n}_{[R,S,\iota]} &= \{ g \in \Res_{S/R} \; \GL{2n/S} \suchthat \ConjTran{g} J_{2n} g = \mu(g) J_{2n}, \mu(g) \in \GGm_{/R} \}\\
\U{n}_{[R,S,\iota]} &= \{ g \in \GU{n} \suchthat \mu(g) = 1 \}\\
\SU{n}_{[R,S,\iota]} &= \{ g \in \U{n} \suchthat \det g = 1 \}
\end{align*}
where $J_{2n} = \Matx{0 & -I_n \\ I_n & 0}$, $I_n$ is the $n \times n$ identity matrix and $\GL{2n/S}, \GGm_{/R}$ are the general linear group (over $S$) and multiplicative group (over $R$) respectively. Here, $\Res_{S/R}$ denotes Weil restriction of scalars.

The case we are interested in is $R = \Z, A = \OK$ and $\iota$ being complex conjugation i.e. the restriction of the non-trivial element of $\Gal(K/\Q)$ to $\OK$. Thus, we omit the subscript $[R,S,\iota]$ if $R = \Z, S = \OK$ and $\iota$ is complex conjugation. Also, due to frequent usage, we also denote $G_n = \GU{n}$, $U_n = \U{n}$.

\newcommand{\bi}{\mathbf{i}}
\newcommand{\HH}{\mathbb{H}}

For $n > 1$, let $\bi_n := i I_n$ and define the hermitian upper half-plane of degree $n$ as the complex manifold
$$\HH_n := \{Z \in M_n(\C) \suchthat -\bi_n (Z - \cc{Z}^t) > 0\}$$
where for a matrix $A \in M_n(\C)$, the notation $A > 0$ means that $A \in M_n(\R)$ and $A$ is positive-definite. The subgroup
$$G_n^+(\R) := \{g \in G_n(\R) \suchthat \mu(g) > 0\}$$
acts transitively on $\HH_n$ via the familiar formula
$$g \; Z := (A Z + B) (C Z + D)^{-1}$$
if $g = \Matx{A & B \\ C & D}$ where $A, B, C, D \in M_n(\C)$.

A subgroup $\Gamma \subset G_n^+(\R)$ is called a \term{congruence subgroup} if
\begin{enumerate}
\item $\Gamma$ is \term{commensurable} with $U_n(\Z)$, and
\item $\Gamma \supseteq \Gamma_n(N) := \{g \in \U{n}(\Z) | g \equiv I_{2n} \bmod N\}$ for some positive $N \in \Z$.
\end{enumerate}

For $g \in G_n^+(\R)$ and $Z \in \HH_n$, we define the factor of automorphy
$$j(g; Z) := \det(C Z + D)$$
and for integer $k \geq 0$ and function $F : \HH_n \rightarrow \C$, we define the action $|_{k} \; g$ by
$$F |_{k} \; g \; (Z) := j(g;Z)^{-k} \; F(g Z).$$

\begin{defn}
A function $F : \HH_n \rightarrow \C$ is called a \term{Hermitian semi-modular form} of weight $k$ and level $\Gamma$ if $$F|_k \; g = F \qquad \text{ for all }g \in \Gamma.$$
If $F$ is further holomorphic then it is called a \term{Hermitian modular form}. We denote
$\MF_k'(\Gamma)$ for the space of all such Hermitian semi-modular forms and by
$\MF_k(\Gamma) \subset \MF_k'(\Gamma)$ for the subspace of Hermitian modular forms.
\end{defn}

Subsequently, we are only interested in the case where congruence subgroup $\Gamma = \Gamma_{0,n}(N)$ with
$$\Gamma_{0,n}(N) := \left\{\Matx{A & B\\C & D} \in U(n,n)(\Z) \suchthat C \equiv 0 \bmod N \OK \right\}$$
so we denote
$$\MF_{k,n}(N) := \MF_k(\Gamma_{0,n}(N))$$ 

The standard notations $\MF_k(N, \psi)$ and $\CF_k(N, \psi)$ are used to denote the space of elliptic modular (resp. cusp) forms on the standard Hecke congruence subgroup
$$\Gamma_0(N) := \braces{ \Mabcd \in \SL{2}(\Z) \suchthat c \equiv 0 \bmod N }$$
and nybentypus $\psi$ where $\psi$ is a Dirichlet character modulo $N$.

Any Hermitian modular form of level $N$ possesses Fourier expansion
$$F(Z) = \sum_{T \in S_n(\Q)} c_F(T) \; \epx{\Tr(T Z)}$$
where $c_F(T) \in \C$,
$$S_n = L_n^\vee := \{x \in \HerMat{n}(\Q) \suchthat \Tr(x L) \subset \Z\}$$
is the dual lattice of the lattice $L_n = \HerMat{n}(\Z) \cap M_n(\OK)$ of integral Hermitian matrices (with respect to the trace pairing $\Tr$) and
$$\HerMat{n} := \{h \in \Res_{\OK/\Z} \, \Mat{n/\OK} \suchthat h^* = h\}$$
is the $\Z$-group scheme of Hermitian matrices. Also, we adopt the notation
$\epx{x} := e^{2 \pi i x}.$
Explicitly, one has, for example
$$S_2 = \braces{\Matx{\ell & t \\ \cc{t} & m} \in \Mat{2}(K) \suchthat \ell, m \in \Z; t \in \DK}.$$

\begin{defn}
The \term{\Maass space} $\MF_k^*(N) \subseteq \MF_{k,2}(N)$ consists of forms $F \in \MF_{k,2}(N)$ such that there exists a function $\alpha_F : \Z_{\geq 0} \rightarrow \C$ so that the Fourier coefficients of $F$ can be described using $\alpha_F$ in the following way
$$c_F(T) = \sum_{\underset{\gcd(d,N) = 1}{d \in \Z_+, \; d \dv \epsilon(T)}} d^{k-1} \alpha_F\left(\frac{D_K \det T}{d^2}\right)$$
for all $T \in S_2(\Z), T \geq 0, T \not= 0$ where $c_F(T)$ is the $T$-th Fourier coefficient of $F$ and
$$\epsilon(T) := \max\left\{q \in \Z_+ \;|\; \frac{1}{q} T \in S_2(\Z)\right\}.$$ \thlabel{defn:Maass_space}
\end{defn}

\subsection{Hermitian Jacobi forms}

We recall the notion of Jacobi forms from \cite{Klosin2018}. First, any $M \in \U{1}(\Z)$ can be written as $\epsilon A$ for some $A \in \SL{2}(\Z)$ and $\epsilon \in \OK^\times$. Let $\UHP$ denote the complex upper half plane. We define the action of the Jacobi group $\U{1}(\Z) \ltimes \OK^2$ on the functions $\phi : \UHP \times \C^2 \rightarrow \C$ by
\begin{align*}
\phi|_{k,m}[\epsilon A] &:= (c\tau + d)^{-k} \epx{- m \frac{czw}{c\tau + d} } \phi\left( \frac{a \tau + b}{c \tau + d}, \frac{\epsilon z}{c \tau + d}, \frac{\cc{\epsilon} w}{c \tau + d} \right)\\
\phi|_{k,m}[\lambda, \mu] &:= \epx{ m (\Nm{\lambda} \tau + \cc{\lambda} z + \lambda w) } \phi(\tau, z + \lambda \tau + \mu, w + \cc{\lambda} \tau + \cc{\mu})
\end{align*}
As noted in \cite{Klosin2018}, the $\epsilon$ in the expression $M = \epsilon A$ is not unique but the action described here is well-defined as long as $k$ is divisible by $|\OK^\times|$.

\begin{defn}
A \term{Hermitian Jacobi forms of weight $k$, index $m$ and level $N$} is a holomorphic function $\phi : \UHP \times \C^2 \rightarrow \C$ such that
\begin{enumerate}
\item $\phi|_{k,m}[\epsilon A] = \phi$ for all $A \in \Gamma_0(N)$;
\item $\phi|_{k,m}[\lambda, \mu] = \phi$ for all $\lambda, \mu \in \OK$; and
\item for each $M \in \SL{2}(\Z)$, the function $\phi|_{k,m}[M]$ has a Fourier expansion of the form
$$\phi|_{k,m}[M](\tau, z, w) = \sum_{\underset{\nu \Nm{t} \leq \ell m}{\ell \in \Z_{\geq 0}, \; t \in \DK^{-1}}} c_\phi^M(\ell, t) \epx{\frac{\ell}{\nu} \tau + \cc{t} z + t w}$$
where $\nu = \nu(M) \in \Z_+$ depends on $M$ and $\nu(I_2) = 1$. The $c_\phi^M(\ell, t) \in \C$ are the Fourier coefficients. We drop the superscript for $M = I_2$ i.e. $c_\phi(\ell, t) = c_\phi^{I_2}(\ell, t)$.
\end{enumerate}
We denote by $\JF_{k,m}(N)$ for the space of Hermitian Jacobi forms of weight $k$, index $m$ and level $N$. \label{defn:jacobi_forms}
\end{defn}

Writing a matrix $Z \in \HH_2$ as $Z = \Matx{\tau & z \\ w & \tau^*}$, we can think of a Hermitian modular form as a function of four complex variables $(\tau, z, w, \tau^*)$ where we can collect its Fourier expansion to express
$$F(Z) = \sum_{m = 0}^{\infty} \phi^F_m(\tau, z, w) \epx{m \tau^*}.$$
It is well--known (c.f. \cite{Krieg1991} Section 3 for level $N = 1$ and \cite{Haverkamp1995} Satz 7.1 for general level) that the functions $\phi^F_m$'s are Jacobi forms of index $m$, normally referred to as Fourier--Jacobi coefficients of $F$. Proposition 2.4 of \cite{Klosin2018} says that just like the Siegel case, a form $F$ in the Hermitian \Maass space $\MF_k^*(N)$ is completely determined by its first Jacobi form $\phi_1^F$ and one of the major results of the same paper (Theorem 2.8) shows that in fact $\MF_k^*(N) \cong \JF_{k,1}^*(N)$ where $\JF_{k,1}^*(N) \subset \JF_{k,1}(N)$ is the subspace of special Jacobi forms.

\begin{defn}
A form $\phi \in \JF_{k,1}(N)$ is called a \term{special Jacobi form} if its Fourier coefficients $c_\phi(\ell, t)$ only depend on $\ell - \Nm{t}$.
\end{defn}

Thus, we have the analogue of the classical \Maass lift for Hermitian modular forms.

\subsection{Theta decomposition}

Recall that $\DK = \frac{i}{\sD} \OK$ is the different ideal of $\OK$. Let $[\DK] := \frac{i}{\sD} \OK / \OK$. For each $u \in [\DK]$, we define the theta function
\begin{align*}
\theta_u(\tau, z, w) := \sum_{a \in u + \OK} \epx{\Nm{a} \tau + \cc{a} z + a w}.
\end{align*}
For fixed $\tau$, the theta functions are linearly independent (c.f. Haverkamp \cite{Haverkamp1995} Proposition 5.1) and they transformed in the following way:
$$\theta_u|_{1,1} \; \sigma = \sum_{v \in [\DK]} M_{u,v}(\sigma) \theta_v$$
for each $\sigma \in \SL{2}(\Z)$ where
\begin{align}
M_{u,v}\Mabcd := \begin{cases}
\frac{-i}{c\sD} \displaystyle \sum_{\gamma \in u + \OK/c\OK} \epx{\frac{a |\gamma|^2 - \gamma \cc{v} - \cc{\gamma} v + d |v|^2}{c}} & \tif c \not= 0;\\
\sign{a} \delta_{u, a v} \epx{ab|u|^2} & \totherwise
\end{cases}
\label{eq:transformation_matrix_theta_series}
\end{align}
according to \cite{Klosin2018} Lemma 3.1 or the Lemma in Section 4 of \cite{Krieg1991}. Since the transformation $|_{k,m} \sigma$ is associative, one could think of $M$ as a group homomorphism $\SL{2}(\Z) \rightarrow \GL{d}(\C)$ where $d$ is the number of classes in $[\DK]$ sending each $\sigma$ to the matrix $(M_{u,v}(\sigma))_{u,v \in [\DK]}$ whose rows and columns are indexed by $[\DK]$ should we fix an ordering of $[\DK]$.

By section 3 of \cite{Klosin2018}, any special Jacobi form $\phi \in \JF_k^*(N)$ can be expressed uniquely as a combination
\begin{align}
\phi = \sum_{u \in [\DK]} f_u \theta_u \qquad \twhere \qquad f_u(\tau) = \sum_{\underset{\ell \equiv -D\Nm{u} \bmod D}{\ell \geq 0}} \a_\phi^*(\ell) \epx{\frac{\ell\tau}{D}}.
\label{eq:theta_decomposition}
\end{align}
Here, $\a_\phi^*(D(m - \Nm{u})) = c_\phi^{I_2}(m, u)$ are the Fourier coefficients of $\phi$ as in definition \ref{defn:jacobi_forms}. We call \eqref{eq:theta_decomposition} the \emph{theta decomposition} of $\phi$.

\subsection{The injection $\JF_k^*(N) \rightarrow \MFDNp$}

Following Krieg, we decompose the character $\chi_K = \prod_{p \dv D} \chi_p$. More concretely, $\chi_p(n) = \jcb{n}{p}$ is just the Legendre symbol mod $p$ for odd prime $p \dv D$ whereas
$$\chi_2(n) = \begin{cases}
\jcb{-4}{n} = (-1)^{(n-1)/2} &\tif 4 \dv D\tand 8 \ndv D,\\
\jcb{-8}{n} = (-1)^{(n^2-4n+3)/8} &\tif 8 \dv D \tand \frac{D}{8} \equiv 1 \bmod 4,\\
\jcb{8}{n} = (-1)^{(n^2-1)/8} &\tif 8 \dv D \tand \frac{D}{8} \equiv 3 \bmod 4.
\end{cases}$$
Then we define
$$a_D(\ell) := |\{u \in [\DK] \suchthat D\Nm{u} \equiv -\ell \bmod D\}|$$
which could be computed as
$$a_D(\ell) = \prod_{p \dv D} (1 + \chi_p(-\ell))$$
according to Krieg \cite{Krieg1991}, Section 4, equation (5). Now, we define the space
$$\MFDNp := \braces{g \in \MFDN \suchthat \a_\ell(g) = 0 \tif a_D(-\ell) = 0}$$
similar to that of \cite{Krieg1991} but for arbitrary $N$ as opposed to $N = 1$ in \cite{Krieg1991}. When $D$ is prime, this is the same as the space defined in \cite{Klosin2018}. We state the following generalization of proposition 3.5 of \cite{Klosin2018}:
\begin{prop}
Suppose that $\gcd(D, N) = 1$. Then the map $\JF_k^*(N) \rightarrow \MFDNp$ given by $\phi \mapsto f := f_0|_{k-1} W_D$ is an injection where $W_D = \Matx{1 & 0 \\ N & 1}\Matx{D & y \\ 0 & 1}$ and $y = N^{-1} \bmod D$. The Fourier coefficients of $f$ satisfy
$$a_\ell(f) = i \frac{a_D(\ell)}{\sD} \chi(N) \a_\phi^*(\ell).$$
\thlabel{prop:jacobi_to_elliptic}
\end{prop}
Note that we are basically dropping the assumptions $D$ prime and $N$ odd in \cite{Klosin2018}. The proof is the same as the proof given in \cite{Klosin2018}. One only needs to check that the simple lemma used in the proof also works for non-prime $D$ and $N$ being even number:
\begin{lm}[Lemma 3.3 in \cite{Klosin2018}] Suppose $\gcd(D, N) = 1$ and $t \in \Z$ such that $\gcd(t, N) = 1$. Then
$$\sum_{\gm \in \OK/N \OK} \epx{\frac{t |\gm|^2}{N}} = \chi(N) N.$$
\end{lm}
\begin{proof}
By Chinese Remainder Theorem, the statement can be reduced to the case where $N$ is a prime power. When $N$ is an odd prime power, the left hand side is a product of two Gauss sum as in Haverkamp \cite{Haverkamp1995}, Lemma 0.4. When $N = 2^r$ is a power of 2, one can use induction on $r$.
\end{proof}

Notice that \thref{prop:jacobi_to_elliptic} already tells us what the lifting should be, if it exists. Suppose we have $g \in \MFDNp$ and $\phi_g \in \JF_k^*(N)$ is its pre-image. Write $\phi_g = \sum f_{u,g} \theta_u$ as in  \eqref{eq:theta_decomposition}. Then by \thref{prop:jacobi_to_elliptic}, $g = f_{0,g} |_{k-1} W_D$ and the Fourier coefficients of $g$ satisfies $a_\ell(g) = i \frac{a_D(\ell)}{\sD} \chi(N) \a_{\phi_g}^*(\ell)$ which leads to
\begin{align*}
\a_{\phi_g}^*(\ell) = \frac{-i \sD}{a_D(\ell) \chi(N)} a_\ell(g)
\label{eq:Fourier_coefficient_for_phi_g}
\end{align*}
and the decomposition \eqref{eq:theta_decomposition} tells us that
\begin{align*}
f_{u,g}(\tau) &= \sum_{\underset{\ell \equiv -D\Nm{u} \bmod D}{\ell \geq 0}} \a_{\phi_g}^*(\ell) \epx{\frac{\ell\tau}{D}}\\
&= \sum_{\underset{\ell \equiv -D\Nm{u} \bmod D}{\ell \geq 0}} \frac{-i \sD}{a_D(\ell) \chi(N)} a_\ell(g) \epx{\frac{\ell\tau}{D}}\\
&= \chi(N) \frac{-i \sD}{a_D(-D\Nm{u})} \sum_{\underset{\ell \equiv -D \Nm{u} \bmod D}{\ell = 0}}^{\infty} a_\ell(g) \epx{\frac{\ell \tau}{D}}.
\end{align*}
Note that this is the lift given by Krieg in Section 6 of \cite{Krieg1991} for level $N = 1$, twisted by the factor of $\chi(N)$ to account for the level.

In section \ref{sec:proof_hermitian_sk_lift}, we will show that for any $g \in \MFDNp$, the function $\phi_g = \sum f_{u,g} \theta_u$ with the $f_{u, g}$ as derived above is indeed a Hermitian Jacobi form; and so the map in \thref{prop:jacobi_to_elliptic} is an isomorphism. To do that, we generalize Krieg's characterization of plus forms and obtain the arithmetical condition for $\phi_g$ to be a Hermitian Jacobi form.

\section{Characterization of plus forms}
\label{sec:characterization_plus_forms}

\let\ifAllProof\iffalse\let\ifFullProof\iftrue\let\ifPersonalNote\iffalse

\unless\ifdefined\IsMainDocument

\documentclass[12pt]{amsart}
\usepackage{amsmath,amssymb,iftex,cancel,xy,xcolor,verbatim}
\usepackage{theoremref}
\usepackage[normalem]{ulem}
\usepackage[margin=1in]{geometry}

\xyoption{all}

\newtheorem{thm}{Theorem}[section]
\newtheorem{lm}[thm]{Lemma}
\newtheorem{prop}[thm]{Proposition}
\newtheorem{cor}[thm]{Corollary}
\theoremstyle{definition}
\newtheorem{defn}[thm]{Definition}
\newtheorem{remark}[thm]{Remark}
\newtheorem{example}[thm]{Example}

\numberwithin{equation}{section}

\title{Hermitian Maass lift for General Level}
\date{}
\author{Lawrence (An Hoa) Vu}

\ifPersonalNote
	\ifPDFTeX
		\usepackage{mdframed}
		\newmdtheoremenv[
		    linewidth=2pt,
		    leftmargin=0pt,
		    innerleftmargin=0.4em,
		    rightmargin=0pt,
		    innerrightmargin=0.4em,
		    innertopmargin=-5pt,
		    innerbottommargin=3pt,
		    splittopskip=\topskip,
		    splitbottomskip=0.3\topskip,
		    skipabove=0.6\topsep
		]{pnote}{Personal note}
	\else
		\definecolor{OliveGreen}{rgb}{0,0.6,0}
		\newcounter{personaln}
		\newenvironment{pnote}{\color{OliveGreen}
		\stepcounter{personaln}
		\par\bigskip\noindent{\bfseries Personal note \arabic{personaln}:}
		\par\medskip}{\par\medskip}
	\fi
\else	
	\let\pnote\comment
	\let\endpnote\endcomment
\fi

\allowdisplaybreaks

\ifPDFTeX
\else
	\usepackage{yfonts}
	\renewcommand{\mathfrak}{\textfrak}
\fi

\newcommand{\LHS}{\text{LHS}}

\newcommand{\UHP}{\mathfrak{H}}
\newcommand{\C}{\mathbb{C}}
\newcommand{\R}{\mathbb{R}}
\newcommand{\Q}{\mathbb{Q}}
\newcommand{\Z}{\mathbb{Z}}

\newcommand{\term}[1]{\emph{#1}}

\newcommand{\Maass}{Maa\ss\ }

\newcommand{\Mat}[1]{M_{#1}}
\newcommand{\GGm}{\mathsf{G_m}}
\newcommand{\GL}[1]{\mathsf{GL}_{#1}}
\newcommand{\SL}[1]{\mathsf{SL}_{#1}}
\newcommand{\GU}[1]{\mathsf{GU}(#1,#1)}
\newcommand{\U}[1]{\mathsf{U}(#1,#1)}
\newcommand{\SU}[1]{\mathsf{SU}(#1,#1)}
\newcommand{\HerMat}[1]{\mathbb{S}_{#1}} \newcommand{\Res}{\text{Res}} 

\newcommand{\GS}{\mathbf{G}} \newcommand{\aDNu}{a_u} \newcommand{\aDNw}{a_w} 

\newcommand{\MF}{\mathfrak{M}}
\newcommand{\CF}{\mathfrak{S}}
\newcommand{\JF}{\mathfrak{J}}
\newcommand{\MFDp}{\MF_{k-1}^+(D, \chi)}
\newcommand{\MFDNp}{\MF_{k-1}^+(DN, \chi)}
\newcommand{\MFDN}{\MF_{k-1}(DN, \chi)}

\newcommand{\Matx}[1]{\begin{pmatrix} #1 \end{pmatrix}}

\newcommand{\ConjTran}[1]{#1^*} \newcommand{\ConjTranInv}[1]{\widehat{#1}} 

\newcommand{\epx}[1]{\mathsf{e}\left[#1\right]}
\newcommand{\sign}[1]{\mathsf{sign}\left[#1\right]}
\newcommand{\cc}[1]{\overline{#1}}

\newcommand{\OK}{\mathfrak{o}_K} \newcommand{\DK}{\mathfrak{d}_K} \newcommand{\Mabcd}{\Matx{a & b\\c & d}}
\newcommand{\Mxyzt}{\Matx{x & y\\z & t}}
\newcommand{\val}{\mathsf{val}}

\newcommand{\parens}[1]{\left( #1 \right)}
\newcommand{\braces}[1]{\left\{ #1 \right\}}
\newcommand{\suchthat}{\;\vline\;}
\newcommand{\tif}{\text{ if }}
\newcommand{\tand}{\text{ and }}
\newcommand{\twhere}{\text{ where }}
\newcommand{\totherwise}{\text{ otherwise }}
\newcommand{\tsince}{\text{ since }}
\newcommand{\teither}{\text{ either }}
\newcommand{\tforsome}{\text{ for some }}
\newcommand{\tor}{\text{ or }}

\newcommand{\Gm}{\Gamma}
\renewcommand{\a}{\alpha}
\renewcommand{\b}{\beta}
\newcommand{\gm}{\gamma}
\newcommand{\ld}{\lambda}
\newcommand{\eps}{\varepsilon}
\renewcommand{\theta}{\vartheta}
\renewcommand{\phi}{\varphi}

\newcommand{\dv}{\;|\;}
\newcommand{\ndv}{\;\nmid\;}
\newcommand{\smD}{\sqrt{-D}}
\newcommand{\sD}{\sqrt{D}}
\renewcommand{\gcd}{}
\newcommand{\jcb}[2]{(#1 \;|\; #2)}
\newcommand{\Tr}{\mathsf{Tr}}
\newcommand{\Nm}[1]{|#1|^2}  \newcommand{\Id}{\text{Id}}
\newcommand{\Gal}{\text{Gal}}

\renewcommand{\labelenumi}{(\roman{enumi})}

\begin{document}

\fi 
For each $m \dv D$, set $\psi_m := \prod_{p \dv m \text{ prime}} \chi_p$ as in section 5 of \cite{Krieg1991}.
For a character $\psi$ of modulus $M$, the (twisted) Gauss sum associated to $\psi$ is defined as
\begin{align*}
\GS(\psi; b) := \sum_{a \in \Z/M} \psi(a) \; \epx{b \frac{a}{M}}.
\end{align*}
and the standard Gauss sum $\GS(\psi) := \GS(\psi; 1)$. We remark that $\chi_K$ is of conductor $D$ and that if $\gcd(m, \frac{D}{m}) = 1$ then $\psi_m$ is of conductor $m$. It is well--known from one of Gauss' proofs of quadratic reciprocity (c.f. \cite{Miyake1989}, Lemma 4.8.1) that
\begin{align*}
\GS(\psi_m) := \eps(\psi_m) \sqrt{m}
\end{align*}
where
\begin{align*}
\eps(\psi_m) := \begin{cases}
1 &\tif \psi_m(-1) = 1,\\
i &\tif \psi_m(-1) = -1.
\end{cases}
\end{align*}

Fix natural numbers $m, n$ such that $D = mn$ and $\gcd(m,n) = 1$. We consider the Hecke operator (without normalization factor)
$$g|_{k-1} U_m := \sum_{j = 0}^{m-1} g|_{k-1} \Matx{1 & j \\ 0 & m}$$
and choose a matrix $P_m \in \SL{2}(\Z)$ such that
$$P_m \equiv \begin{cases}
J \bmod m^2\\
I \bmod (nN)^2
\end{cases}$$
where $J := J_1 = \Matx{0 & -1\\1 & 0}$ and $T := \Matx{1 & 1\\0 & 1} $ are the usual generators for $\SL{2}(\Z)$. Set
$$Q_m := P_m \Matx{m & 0 \\ 0 & 1}$$
and define the operator $V_m$ by
$$g|_{k-1} V_m := g|_{k-1} U_m |_{k-1} Q_m.$$

Similar to \cite{Krieg1991} Section 5, one has
\begin{lm}
$Q_m \Gm_0(ND) Q_m^{-1} = \Gm_0(ND)$ and so $g \mapsto g|_{k-1} Q_m$ is an involution of $\MFDN$.
\thlabel{lm:partial_fricke_involution}
\end{lm}

\begin{pnote}
\begin{proof}
Suppose that $x, y, z, t \in \Z$ such that
$$\gm = \Matx{x & y \\ NDz & t} \in \Gm_0(ND)$$
and we need to check that $Q_m \gm Q_m^{-1} \in \Gm_0(ND)$. To do that, it suffices to check that
$$Q_m \gm Q_m^{-1} \equiv \Matx{* & * \\ 0 & *} \bmod M$$
for $M = m$ and $M = nN$ by Chinese Remainder Theorem. Let's do it:
\begin{align*}
Q_m \gm Q_m^{-1} &= P_m \Matx{m & 0 \\ 0 & 1} \Matx{x & y \\ NDz & t} \Matx{1/m & 0 \\ 0 & 1} P_m^{-1}\\
&= P_m \Matx{x & my \\ Nnz & t} P_m^{-1}\\
&\equiv \begin{cases}
J \Matx{x & my \\ Nnz & t} J^{-1} \bmod m\\
I \Matx{x & my \\ Nnz & t} I^{-1} \bmod nN\\
\end{cases} \text{ by definition of } P_m\\
&\equiv \begin{cases}
\Matx{0 & -1\\1 & 0} \Matx{x & my \\ Nnz & t} \Matx{0 & 1\\-1 & 0} \bmod m\\
\Matx{x & my \\ 0 & t} \bmod nN\\
\end{cases}\\
&\equiv \begin{cases}
\Matx{t & -Nnz \\-my & x} \bmod m\\
\Matx{x & my \\ 0 & t} \bmod nN\\
\end{cases}\\
&\equiv \Matx{* & * \\0 & *} \bmod m \tand \bmod nN
\end{align*}
To check the part about involution: for any $\gm \in \Gm_0(ND)$, one has $Q_m \gm Q_m^{-1} = \gm_1$ (equivalently $Q_m \gm = \gm_1 Q_m$) for some $\gm_1 \in \Gm_0(ND)$ and so
$$(g|_{k-1} Q_m)|_{k-1}\gm = g|_{k-1} Q_m \gm = g|_{k-1} \gm_1 Q_m = \chi(\gm_1) g|_{k-1} Q_m.$$

It remains to check that $\chi(\gm_1) = \chi(\gm)$. And this follows from the computation above. Note that it is not true for general character $\chi$ since the lower right entry of $\gm_1$ is not congruent to $t \mod D$: It is congruent to $t \mod Nn$ and to $x \mod m$. However, since our character $\chi$ is real i.e. take values $\pm 1$, $\chi(x) = \frac{1}{\chi(t)} = \chi(t)$ as $\frac{1}{\pm 1} = \pm 1$. Decompose $\chi = \psi_m \psi_n$ and we find that $\chi(\gm_1) = \psi_m(x) \psi_n(t) = \psi_m(t) \psi_n(t) = \chi(t) = \chi(\gm)$.
\end{proof}
\end{pnote}

\begin{pnote}
We record the following lemma that was implicitly used by Krieg multiple times.
\begin{lm}
Suppose that the matrix
$$\gm = \Mxyzt$$
satisfies $\gcd(x, z) = 1$ and $ND \dv z$. In other words, there exists a matrix $\Matx{x & \a \\ y & \b}$ in $\Gm_0(ND)$ whose first column is the same as that of $\gm$. Then
$$g|_{k-1} \gm = \chi(x)^{-1} g|_{k-1} \Matx{1 & \b y - \a t \\0 & \det \gm}$$
for any $g \in \MFDN$. If we further have $\gcd(x, \det \gm) = 1$ (which is true when we have $\det \gm \dv DN$ as in the many applications below) then
$$g|_{k-1} \gm = \chi(x)^{-1} g|_{k-1} \Matx{1 & y/x \bmod \det \gm\\0 & \det \gm}.$$
\thlabel{lm:action_of_completable_matrix}
\end{lm}
\begin{proof}
Using our assumption, let $\a, \b$ be such that
$$\Matx{x & \a \\ z & \b} \in \Gm_0(ND)$$
and then we have
\begin{align*}
\Matx{x & \a \\ z & \b}^{-1} \gm &= \Matx{\b & -\a \\ -z & x} \Mxyzt\\
&= \Matx{\b x - \a z & \b y - \a t \\0 & xt - zy}\\
&= \Matx{1 & \b y - \a t \\0 & d}
\end{align*}
where $d = \det \gm$.

It follows that
\begin{align*}
g|_{k-1} \gm &= g|_{k-1} \Matx{x & \a \\ z & \b} \Matx{1 & \b y - \a t \\0 & d}\\
&= \chi(\b) g|_{k-1} \Matx{1 & \b y - \a t \\0 & d} &\text{ by assumption } g \in \MFDN\\
&= \chi(x)^{-1} g|_{k-1} \Matx{1 & \b y - \a t \\0 & d} &\text{since } \b \equiv x^{-1} \bmod DN
\end{align*}

We get $\b = \frac{1 + z\a}{x}$ from the fact that $x \b - z \a = 1$ (note that we can't guarantee that $z \not= 0$ but we can guarantee that $x \not= 0$) whence
$$\b y - \a t = \frac{1 + z\a}{x} y - \a t = \frac{(1 + z\a)y - \a t x}{x} = \frac{y - \a (tx - yz)}{x} = \frac{y}{x} - \frac{\a}{x} d$$
is congruent to $\frac{y}{x} \mod d$ should $\gcd(x, d) = 1$. Since the action $g|_{k-1} \Matx{1 & r\\0 & h}$ only depends on $r \bmod h$ (as $e^{2\pi i \Z} = 1$), we get
$$g|_{k-1} \gm = \chi(x)^{-1} g|_{k-1} \Matx{1 & y/x \\0 & \det \gm}.$$
\end{proof}
\end{pnote}

\begin{prop}
If $k$ is even and $g \in \MFDN$ then we have
$$g|_{k-1} V_m = \GS(\psi_m) m^{1 - k} g$$
if and only if
$$\a_\ell(g) = 0 \text{ for all } \ell \text{ such that } \psi_m(-\ell) = -1.$$
\thlabel{prop:characterization_of_Vm}
\end{prop}
\begin{proof}
The proof is verbatim as in \cite{Krieg1991} with appropriate modification to take care of the level.

\begin{pnote}
For completeness, we replicate the proof. Write $P_m = \Mabcd$. By definition, one has
\begin{align*}
g|_{k-1} V_m &= \sum_{j = 0}^{m - 1} g|_{k-1} \Matx{1 & j \\ 0 & m} \Mabcd \Matx{m & 0 \\ 0 & 1}\\
&= \sum_{j = 0}^{m - 1} g|_{k-1} \Matx{a + jc & b + jd \\ mc & md} \Matx{m & 0 \\ 0 & 1}
\end{align*}
Observe that $ND \dv mc$ since $c \equiv 0 \bmod nN$ by construction and that if $\gcd(j, m) = 1$ then $\gcd(a + jc, mc) = 1$ because $\gcd(a + jc, c) = \gcd(a, c) = 1$ and $\gcd(a + jc, m) = \gcd(j, m) = 1$ since
\begin{align}
a + jc &\equiv \begin{cases}
j \mod m\\
1 \mod n
\end{cases}
\end{align}
by definition of $P_m$. Thus, when $\gcd(m, j) = 1$, we can get a matrix in $\Gm_0(ND)$ of the form $\Matx{a + jc & *\\mc & *}$ and so by \thref{lm:action_of_completable_matrix}, one has
\begin{align*}
g|_{k-1} \Matx{a + jc & b + jd \\ mc & md} &= \chi(a + jc) g|_{k-1} \Matx{1 & -j\\0 & m}
\end{align*}
since $\chi = \chi^{-1}$ is a real character; $\det \Matx{a + jc & b + jd \\ mc & md} = m \dv ND$ and
$$\frac{b + jd}{a + jc} \equiv \frac{-1}{j} \bmod m$$
from $P_m \equiv J \mod m$. As noted in \cite{krieg1991},
\begin{align*}
\chi(a + jc) &= \psi_n(a + jc) \psi_m(a + jc)\\
&= \psi_n(1) \psi_m(j) &\text{ by } \eqref{eq:a_plus_jc_mod_mn}\\
&= \psi_m(j)
\end{align*}
Set
$$g_r(\tau) := \sum_{\ell = 0; \ell \equiv r \bmod m}^{\infty} \a_\ell(g) e^{2\pi i \ell \tau}.$$
We have
\begin{align*}
&\sum_{(j,m) = 1} g|_{k-1} \Matx{a + jc & b + jd \\ mc & md} \Matx{m & 0 \\ 0 & 1}\\
&= \sum_{(j,m) = 1} \psi_m(j) g|_{k-1} \Matx{1 & -j\\0 & m} \Matx{m & 0 \\ 0 & 1} &\text{by change of variable } j \text{ to } j^{-1}\\
&= \sum_{(j,m) = 1} \psi_m(j) g|_{k-1} \Matx{m & -j\\0 & m}\\
&= \sum_{(j,m) = 1} \psi_m(j) m^{1-k} g\parens{\frac{m\tau -j}{m}}\\
&= m^{1-k} \sum_{(j,m) = 1} \psi_m(j) g\parens{\tau - \frac{j}{m}}\\
&= m^{1-k} \sum_{(j,m) = 1} \psi_m(j) \sum_{\ell = 0}^{\infty} \a_\ell(g) \exp\parens{2\pi i \ell \parens{\tau - \frac{j}{m}}}\\
&= m^{1-k} \sum_{(j,m) = 1} \psi_m(j) \sum_{\ell = 0}^{\infty} \a_\ell(g) e^{2\pi i \ell \tau} \exp\parens{-2\pi i \ell \frac{j}{m}}\\
&= m^{1-k} \sum_{(j,m) = 1} \psi_m(j) \sum_{r = 0}^{m - 1} g_r(\tau) e^{-2\pi i r j / m}\\
&= m^{1-k} \sum_{r = 0}^{m - 1} \underbrace{\parens{\sum_{(j,m) = 1} \psi_m(j) e^{-2\pi i r j / m}}}_{\GS(\psi_m; -r)} g_r(\tau)\\
&= m^{1-k} \GS(\psi_m) \sum_{r = 0; (r, m) = 1}^{m - 1} \psi_m(-r) g_r(\tau)
\end{align*}
by basic property of Gauss sum.

Now we consider the remaining terms
$$f = \sum_{j=0; (j,m) > 1}^{m-1} g|_{k-1} \Matx{a + jc & b + jd \\ mc & md} \Matx{m & 0 \\ 0 & 1}$$
In \cite{krieg1991}, it was observed that
$$g|_{k-1} \Matx{a + jc & b + jd \\ mc & md} \Matx{m & 0 \\ 0 & 1}$$
is invariant under $\tau \mapsto \tau + \frac{1}{p}$ for any prime $p | \gcd(m, j)$. To see that, we need to show that
\begin{align*}
&g|_{k-1} \Matx{a + jc & b + jd \\ mc & md} \Matx{m & 0 \\ 0 & 1} \Matx{1 & 1/p\\0 & 1} = g|_{k-1} \Matx{a + jc & b + jd \\ mc & md} \Matx{m & 0 \\ 0 & 1}\\
\iff &g = g|_{k-1} \Matx{a + jc & b + jd \\ mc & md} \Matx{m & 0 \\ 0 & 1} \Matx{1 & 1/p\\0 & 1} \Matx{1/m & 0 \\ 0 & 1} \Matx{a + jc & b + jd \\ mc & md}^{-1}\\
\iff &g = g|_{k-1} \Matx{a + jc & b + jd \\ mc & md} \Matx{1 & m/p\\0 & 1} \Matx{d & -\frac{b + jd}{m} \\ -c & \frac{a + jc}{m}}\\
\iff &g = g|_{k-1} \Matx{a + jc & (a + jc)\frac{m}{p} + b + jd \\ mc & \frac{m^2c}{p} + md} \Matx{d & -\frac{b + jd}{m} \\ -c & \frac{a + jc}{m}}\\
\iff &g = g|_{k-1} \Matx{(a + jc) d - c \parens{(a + jc)\frac{m}{p} + b + jd} & (a + jc)\parens{-\frac{b + jd}{m}} + \parens{(a + jc)\frac{m}{p} + b + jd} \frac{a + jc}{m} \\mcd - c \parens{\frac{m^2c}{p} + md} & mc \parens{-\frac{b + jd}{m}} + \parens{\frac{m^2c}{p} + md} \frac{a + jc}{m}}\\
\iff &g = g|_{k-1} \Matx{* & \frac{(a + jc)^2}{p}\\\frac{m^2c^2}{p} & -c(b + jd) + \frac{mc(a + jc)}{p} + (a + jc)d}
\end{align*}
which is true because the matrix is in $\Gm_0(ND)$. (It is obviously integral\footnote{The only not-so-obvious one is $\frac{(a + jc)^2}{p} \in \Z$; equivalently $a^2 \equiv 0 \bmod p$ since $p | j$. But recall that $a \equiv 0 \bmod m$ and so is divisible by $p$ as well.} and of determinant 1. Then $m^2c^2/p$ is divisible by $mc^2$ for $m/p \in \Z$ and $mc^2$ is divisible by $mnN = ND$ by construction.)

Thus we can express
$$f = \sum_{\ell \geq 0; (\ell, m) > 1} \b(\ell) e^{2 \pi i \ell \tau}.$$
In other words, $f$ has no Fourier coefficients that are relative prime to $m$.

Now if $g|_{k-1} V_m = \GS(\psi_m) m^{1-k} g$ then we have
$$f + m^{1-k} \GS(\psi_m) \sum_{r = 0; (r, m) = 1}^{m - 1} \psi_m(-r) g_r(\tau) = \GS(\psi_m) m^{1-k} \sum_{r = 0}^{m - 1} g_r(\tau)$$
or equivalently
$$f - \GS(\psi_m) m^{1-k} \sum_{r = 0; (r,m) > 1}^{m - 1} g_r(\tau) = \GS(\psi_m) m^{1-k} \sum_{r = 0; (r,m) = 1}^{m - 1} (1 - \psi_m(-r)) g_r(\tau)$$
Observe that the left hand side has all $\ell$ Fourier coefficients zero for all $\gcd(\ell,m) = 1$ while non-zero Fourier coefficients of the right hand side are only for $\gcd(\ell, m) = 1$. Hence, both sides must be zero. Thus, if $\psi_m(-r) \not= 1$ i.e. $\psi_m(-r) = -1$ then we must have $g_r(\tau) = 0$ since $\GS(\psi_m) m^{1-k} \not= 0$. Hence, we found that $a_\ell(g) = 0$ for all $\ell$ such that $\psi_m(-\ell) = -1$.

Conversely, suppose that $a_\ell(g) = 0$ for all $\ell$ such that $\psi_m(-\ell) = -1$. Then
$$h = g|_{k-1} V_m - \GS(\psi_m) m^{1-k} g$$
has all $\ell$ Fourier coefficients vanish for $\gcd(\ell,m) = 1$. If we can show that $h \in \MFDN$ then we can conclude that $h = 0$ by Theorem 4.6.8 of \cite{Miyake1989}: the conductor of $\chi$ is $D$ and $\gcd\parens{m,\frac{ND}{D}} = \gcd(m, N) = 1$ so case (1) of the aforementioned theorem implies that $h = 0$. And that is because of our assumption $g \in \MFDN$ and by definition, $g|_{k-1} V_m = g|_{k-1} U_m |_{k-1} Q_m$ we already know that Hecke operator $U_m$ preserves $\MFDN$ and so is partial Fricke involution $Q_m$ as shown in lemma \ref{lm:partial_fricke_involution}.
\end{pnote}
\end{proof}

\begin{cor}
Let $k$ be even and $D = 2^e d$ with $2 \nmid d$. Then
$$\MFDNp = \braces{g \in \MFDN \suchthat \begin{array}{c}g|_{k-1} V_p = \GS(\psi_p) p^{1-k} g\\
\text{ for } p = 2^e \tand p | d \text{ odd prime}\end{array}}.$$
Furthhermore, if $g \in \MFDNp$ and $\gcd\parens{m, \frac{D}{m}} = 1$ then
$$g|_{k-1} V_m = \GS(\psi_m) m^{1-k} g.$$
\thlabel{cor:characterization_plus_forms}
\end{cor}

\begin{pnote}
\begin{proof}
By definition, $g \in \MFDNp$ if and only if $a_\ell(g) = 0$ for all $\ell$ such that $a_D(-\ell) = 0$. Recall that $a_D(-\ell) = 0$ if and only if $\chi_p(-\ell) = \psi_p(-\ell) = -1$ for some prime $p | D$. Thus, for any fix prime $p | D$ (or $p = 2^e || D$), we have $a_\ell(g) = 0$ for all $\ell$ such that $\psi_p(-\ell) = -1$ and so by our proposition, we have $g|_{k-1} V_p = \GS(\psi_p) p^{1-k} g$. This proves the inclusion $\subseteq$. The converse is by the other direction in our proposition.

The ``Furthermore'' statement follows from the proposition\footnote{Alternatively, I would imagine that one can follow Krieg's remark and use results of section 4.6 of \cite{miyake1989} to show that the operators $V_p$ and $V_{2^e}$ are mutually commutative on $\MFDN$ and somehow relate $V_{pq}$ and $V_p V_q$. Note that $\eps(\psi_p) \eps(\psi_q)$ might not be equal $\eps(\psi_{pq})$ so $V_p V_q = V_{pq}$ won't work all the time.}. First, notice that by induction, we only need to show that if $m, n | D$ are such that $m, n, \frac{D}{mn}$ are mutually coprime and that the statement is true for both $m$ and $n$ then it is also true for $mn$.

By assumption $(m,n) = 1$ so we have $\psi_{mn} = \psi_m \psi_n$ and so $\psi_{mn}(-\ell) = -1$ if and only if $\psi_m(-\ell) \not= \psi_n(-\ell)$. By assumption that the statement is true for $m$ and $n$, we know that $a_\ell(g) = 0$ if either $\psi_m(-\ell) = -1$ or $\psi_n(-\ell) = -1$. Hence,  obviously if $\psi_{mn}(-\ell) = -1$, we have $\psi_m(-\ell) \not= \psi_n(-\ell)$ and so it must be the case that either $\psi_m(-\ell) = -1$ or $\psi_n(-\ell) = -1$ whence $a_\ell(g) = 0$ as well. By the proposition, the statement is true for $mn$.
\end{proof}
\end{pnote}

Note that \thref{prop:characterization_of_Vm} and \thref{cor:characterization_plus_forms} are the generalization of Proposition and Corollary A (resp.) in Section 5 of \cite{Krieg1991}.

For any $j \in \Z/D\Z$, set $\sigma_j := \Matx{1 & j \\ 0 & D}$. Our goal is to compute $g|_{k-1} \sigma_j \sigma$ for $g \in \MFDNp$ and any $\sigma \in \Gm_0(N)$. Fix such a $\sigma = \Matx{x & y \\ z & t}$. For any $j \in \Z$, let $\mu = \mu_\sigma(j) := \gcd(x + zj, D)$ and $m = m_\sigma(j) := D_\mu$ be the $\mu$-component of $D$ i.e.
$$m = \prod_{p | \mu \text{ prime}} p^{\val_p(D)}$$
and let $n := \frac{D}{m}$. Then we have
\begin{lm}
\begin{enumerate}
\item $\gcd(x + zj, n) = 1$;
\item $\gcd(r(x + zj) - nz, m) = 1$ for any $r \in \Z$; in particular, $\gcd(m, z) = 1$;
\item $\gcd(x + zj, y + tj) = 1$;
\item $\gcd(y + tj, m) = 1$.
\end{enumerate}
\thlabel{lm:various_gcd}
\end{lm}

\begin{pnote}
Note that it could happen that $x + zj = 0$ but $m, \mu > 0$ since $\mu, m$ divides $D$.
\begin{proof}
It is obvious from the definition that $\gcd(m, n) = 1$ and that if $m$ and $\mu$ has the same set of prime divisors.
\begin{enumerate}
\item We have $\gcd(x + zj, n) \dv \gcd(x + zj, D) = \mu \dv m$ while on the other hand, we obviously have $\gcd(x + zj, n) | n$ and so $\gcd(x + zj, n) \dv \gcd(m, n) = 1$.

\item If we have a prime $p \dv \gcd(r(x + zj) - nz, m)$ then $p \dv m$ so $p \dv \mu$ by construction so $p \dv x + zj$. We also have $p \dv r(x + zj) - nz$ so that $p | nz$ which implies $p \dv z$ since $\gcd(m, n) = 1$. But then we get from $p \dv x + zj$ that $p \dv x$. So $p$ divides both $x, z$; contradicting the fact that $\gcd(x, z) = 1$ from the assumption $\sigma \in \Gm_0(N)$.

\item This is obvious since $x + zj$ and $y + jt$ are the first row of the matrix $\Matx{1 & j \\ 0 & 1}\Matx{x & y \\ z & t}$ in $\SL{2}(\Z)$.

\item Suppose that we have a prime divisor $p \dv \gcd(y + jt, m)$. Then as in (ii), $p \dv m$ so $p \dv \mu$; hence $p \dv x + zj$. But then $p \dv \gcd(x + zj, y + jt) = 1$ and we get a contradiction with (iii).
\end{enumerate}
\end{proof}
\end{pnote}

\begin{prop}
For any $g \in \MFDNp$, we have
$$g|_{k-1} \sigma_j \sigma = \sum_{s = 0}^{m-1} \frac{\GS(\psi_m; (x + zj)s - \lambda)}{m \; \psi_n(x + zj)} g|_{k-1} \Matx{1 & ns + \kappa \\ 0 & D}$$
where $\kappa = \kappa_\sigma(j) \in \Z$ is any integer such that
$$\kappa \equiv \frac{y + t j}{x + zj} \mod n$$
and
$$\lambda := \frac{y + jt - \kappa(x + zj)}{n} \in \Z.$$

In case $D$ is square-free, the formula can be simplifed to
$$g|_{k-1} \sigma_j \sigma = \frac{\GS(\psi_m; nz)}{m \; \psi_n(x + zj)} \sum_{s = 0}^{m-1} g|_{k-1} \Matx{1 & n s + \kappa \\ 0 & D}.$$
\thlabel{prop:transformation_plus_forms}
\end{prop}

\begin{pnote}
Note that if $x + zj = 0$ then $\mu = \gcd(0, D) = D$ so $m = \mu = D$ and $n = 1$. In such case, one can take $\kappa$ can be any integer and $ns + \kappa$ goes through all residue modulo $D$ as $s$ runs from 0 to $m - 1$. Also $\psi_1(0) = 1$. \end{pnote}

\begin{proof}
Let $P_m = \Mabcd$. We have
\begin{align*}
g|_{k-1} \sigma_j \sigma &= \frac{m^{k - 1}}{\GS(\psi_m)} g |_{k-1} V_m \; |_{k-1} \Matx{x + zj & y + jt \\ Dz & Dt} \text{ by \thref{cor:characterization_plus_forms}}\\
&= \frac{m^{k - 1}}{\GS(\psi_m)} \sum_{r = 0}^{m-1} g |_{k-1} \Matx{a + rc & b + rd \\ mc & md} \Matx{m & 0 \\ 0 & 1} \Matx{x + zj & y + jt \\ Dz & Dt}\\
&= \frac{1}{\GS(\psi_m)} \sum_{r = 0}^{m-1} g |_{k-1} \Matx{(a + rc)(x + zj) + (b + rd) nz & (a + rc)(y + jt) + (b + rd)nt \\ mc(x + zj) + dDz & mc(y + jt) + (b + rd) mdnt}
\end{align*}
Let us temporarily denote the above matrix by $\gm = \gm_r$ and denote by $A, B$ the entries of the first column of $\gm$ i.e.
\begin{align*}
A &= (a + rc)(x + zj) + (b + rd) nz\\
B &= mc(x + zj) + dDz
\end{align*}
Observe that $\det \gm = D$ and that the lower row of $\gm$ is divisible by $D$. Thus, the matrix obtained by dividing both entries on the lower row of $\gm$ by $D$, namely the matrix
$$\Matx{A & (a + rc)(y + jt) + (b + rd)nt \\ \frac{B}{D} & \frac{mc(y + jt) + (b + rd) mdnt}{D}}$$
is in $\SL{2}(\Z)$ whence $\gcd\parens{A, \frac{B}{D}} = 1$ and so $\gcd(A, B) = \gcd\parens{A, m n \frac{B}{D}} = \gcd(A, m) \gcd(A, n) = 1$ by \thref{lm:various_gcd} for $A \equiv \begin{cases}
r(x + zj) - nz \mod m\\
x + zj \mod n
\end{cases}$. Since $\gcd(A, B) = 1$ and $ND \dv B$, there exists a matrix $\Matx{A & *\\B & *} \in \Gm_0(ND)$ and so we have
$$g|_{k-1} \gm = \frac{\psi_n(x + zj)}{\psi_m(r(x + zj) - nz)} g|_{k-1} \Matx{1 & \frac{(a + rc)(y + jt) + (b + rd)nt}{(a + rc)(x + zj) + (b + rd) nz} \bmod D \\ 0 & D}$$
\begin{pnote}
\noindent by \thref{lm:action_of_completable_matrix}. The determinant of the matrix is clearly $D$.
\end{pnote}
Let us denote the upper right entry
$$F_r = \frac{(a + rc)(y + jt) + (b + rd)nt}{(a + rc)(x + zj) + (b + rd) nz} \bmod D \equiv \begin{cases}
\displaystyle \frac{r (y + jt) - nt}{r(x + zj) - nz} \mod m\\
\displaystyle \frac{y + jt}{x + zj} \mod n
\end{cases}$$
so that we get
\begin{align*}
g|_{k-1} \sigma_j \sigma
&= \sum_{r = 0}^{m-1} \frac{\psi_n(x + zj)}{\GS(\psi_m; r(x + zj) - nz)} g|_{k-1} \Matx{1 & F_r \\ 0 & D}
\end{align*}
As $r$ runs through the residues mod $m$, the $F_r$ also goes through the subset of residues
$$\braces{n s + \kappa \suchthat 0 \leq s \leq m - 1} \subseteq \Z/D\Z.$$
A simple change-of-variable from $r$ to $s$ and use observation that $\frac{1}{\GS(\psi_m; b)} = \frac{\GS(\psi_m; -b)}{m}$ to move the Gauss sum to the numerator, we derive the formula claimed in the proposition.

With the assumption that $D$ is square-free, we always have $\mu = m$ so $m \dv x + zj$ and the factor $\GS(\psi_m; r(x + zj) - nz) = \GS(\psi_m; -nz)$ is independent of $s$.
\end{proof}

\section{Arithmetic Criterion}
\label{sec:arithmetic_condition}

\let\ifAllProof\iffalse\let\ifFullProof\iftrue\let\ifPersonalNote\iffalse

\unless\ifdefined\IsMainDocument

\documentclass[12pt]{amsart}
\usepackage{amsmath,amssymb,iftex,cancel,xy,xcolor,verbatim}
\usepackage{theoremref}
\usepackage[normalem]{ulem}
\usepackage[margin=1in]{geometry}

\xyoption{all}

\newtheorem{thm}{Theorem}[section]
\newtheorem{lm}[thm]{Lemma}
\newtheorem{prop}[thm]{Proposition}
\newtheorem{cor}[thm]{Corollary}
\theoremstyle{definition}
\newtheorem{defn}[thm]{Definition}
\newtheorem{remark}[thm]{Remark}
\newtheorem{example}[thm]{Example}

\numberwithin{equation}{section}

\title{Hermitian Maass lift for General Level}
\date{}
\author{Lawrence (An Hoa) Vu}

\ifPersonalNote
	\ifPDFTeX
		\usepackage{mdframed}
		\newmdtheoremenv[
		    linewidth=2pt,
		    leftmargin=0pt,
		    innerleftmargin=0.4em,
		    rightmargin=0pt,
		    innerrightmargin=0.4em,
		    innertopmargin=-5pt,
		    innerbottommargin=3pt,
		    splittopskip=\topskip,
		    splitbottomskip=0.3\topskip,
		    skipabove=0.6\topsep
		]{pnote}{Personal note}
	\else
		\definecolor{OliveGreen}{rgb}{0,0.6,0}
		\newcounter{personaln}
		\newenvironment{pnote}{\color{OliveGreen}
		\stepcounter{personaln}
		\par\bigskip\noindent{\bfseries Personal note \arabic{personaln}:}
		\par\medskip}{\par\medskip}
	\fi
\else	
	\let\pnote\comment
	\let\endpnote\endcomment
\fi

\allowdisplaybreaks

\ifPDFTeX
\else
	\usepackage{yfonts}
	\renewcommand{\mathfrak}{\textfrak}
\fi

\newcommand{\LHS}{\text{LHS}}

\newcommand{\UHP}{\mathfrak{H}}
\newcommand{\C}{\mathbb{C}}
\newcommand{\R}{\mathbb{R}}
\newcommand{\Q}{\mathbb{Q}}
\newcommand{\Z}{\mathbb{Z}}

\newcommand{\term}[1]{\emph{#1}}

\newcommand{\Maass}{Maa\ss\ }

\newcommand{\Mat}[1]{M_{#1}}
\newcommand{\GGm}{\mathsf{G_m}}
\newcommand{\GL}[1]{\mathsf{GL}_{#1}}
\newcommand{\SL}[1]{\mathsf{SL}_{#1}}
\newcommand{\GU}[1]{\mathsf{GU}(#1,#1)}
\newcommand{\U}[1]{\mathsf{U}(#1,#1)}
\newcommand{\SU}[1]{\mathsf{SU}(#1,#1)}
\newcommand{\HerMat}[1]{\mathbb{S}_{#1}} \newcommand{\Res}{\text{Res}} 

\newcommand{\GS}{\mathbf{G}} \newcommand{\aDNu}{a_u} \newcommand{\aDNw}{a_w} 

\newcommand{\MF}{\mathfrak{M}}
\newcommand{\CF}{\mathfrak{S}}
\newcommand{\JF}{\mathfrak{J}}
\newcommand{\MFDp}{\MF_{k-1}^+(D, \chi)}
\newcommand{\MFDNp}{\MF_{k-1}^+(DN, \chi)}
\newcommand{\MFDN}{\MF_{k-1}(DN, \chi)}

\newcommand{\Matx}[1]{\begin{pmatrix} #1 \end{pmatrix}}

\newcommand{\ConjTran}[1]{#1^*} \newcommand{\ConjTranInv}[1]{\widehat{#1}} 

\newcommand{\epx}[1]{\mathsf{e}\left[#1\right]}
\newcommand{\sign}[1]{\mathsf{sign}\left[#1\right]}
\newcommand{\cc}[1]{\overline{#1}}

\newcommand{\OK}{\mathfrak{o}_K} \newcommand{\DK}{\mathfrak{d}_K} \newcommand{\Mabcd}{\Matx{a & b\\c & d}}
\newcommand{\Mxyzt}{\Matx{x & y\\z & t}}
\newcommand{\val}{\mathsf{val}}

\newcommand{\parens}[1]{\left( #1 \right)}
\newcommand{\braces}[1]{\left\{ #1 \right\}}
\newcommand{\suchthat}{\;\vline\;}
\newcommand{\tif}{\text{ if }}
\newcommand{\tand}{\text{ and }}
\newcommand{\twhere}{\text{ where }}
\newcommand{\totherwise}{\text{ otherwise }}
\newcommand{\tsince}{\text{ since }}
\newcommand{\teither}{\text{ either }}
\newcommand{\tforsome}{\text{ for some }}
\newcommand{\tor}{\text{ or }}

\newcommand{\Gm}{\Gamma}
\renewcommand{\a}{\alpha}
\renewcommand{\b}{\beta}
\newcommand{\gm}{\gamma}
\newcommand{\ld}{\lambda}
\newcommand{\eps}{\varepsilon}
\renewcommand{\theta}{\vartheta}
\renewcommand{\phi}{\varphi}

\newcommand{\dv}{\;|\;}
\newcommand{\ndv}{\;\nmid\;}
\newcommand{\smD}{\sqrt{-D}}
\newcommand{\sD}{\sqrt{D}}
\renewcommand{\gcd}{}
\newcommand{\jcb}[2]{(#1 \;|\; #2)}
\newcommand{\Tr}{\mathsf{Tr}}
\newcommand{\Nm}[1]{|#1|^2}  \newcommand{\Id}{\text{Id}}
\newcommand{\Gal}{\text{Gal}}

\renewcommand{\labelenumi}{(\roman{enumi})}

\begin{document}

\fi 
To simplify a lot of subsequent formulas, we shall also denote $a_u := a_D(-D \Nm{u})$ for $u \in [\DK]$. Given a plus form
$$g = \sum_{\ell = 0}^{\infty} a_\ell(g) \epx{\ell \tau} \in \MFDNp,$$
we define
\begin{align*}
g_u &:= \frac{-i \sD}{a_u} \sum_{\underset{\ell \equiv -D \Nm{u} \bmod D}{\ell = 0}}^{\infty} a_\ell(g) \epx{\frac{\ell \tau}{D}}\qquad \text{ for all } u \in [\DK]
\end{align*}
and
$$\phi_g := \sum g_u \theta_u.$$
If $\phi_g$ is a Jacobi form of level $N$ then it will be the lifting of $\chi(N) g = \pm g$, by the argument at the end of section \ref{sec:hermitian_modular_forms}.

In this section, we shall derive transformation formula for $g_u$ using result of the previous section and then deduce an \emph{arithmetic condition} for $\phi_g$ to be a Hermitian Jacobi form (\thref{prop:explicit_arithmetic_criterion_jacobi_form}).

\begin{prop}
Let $\sigma = \Mabcd \in \Gm_0(N)$. For each $j \in \Z$, we choose $\kappa = \kappa_\sigma(j) \in \Z$ such that
$$\kappa \equiv \frac{b + dj}{a + cj} \bmod n$$
as before; in addition that if $m \not= \mu$ (i.e. $m = 2\mu$ or $m = 4\mu$), then
$$\kappa \equiv \frac{(b + dj + c)/2^f}{(a+cj)/2^f} \bmod \frac{m}{\mu}$$
where $f = \val_2(a + cj) = \val_2(\mu)$ is the maximum power of two dividing $\mu$. Then
$$g_u|_{k-1} \sigma = \frac{1}{D\aDNu} \sum_{v \in [\DK]} \sum_{\underset{\gcd(D\Nm{v}, m) = \mu}{j = 0}}^{D - 1} R_\sigma(v, j) \; \GS(\psi_m; nc) \; \psi_n(a + cj) \; \epx{\Nm{u} j - \Nm{v} \kappa} g_v$$
where
\begin{align*}
R_\sigma(v, j) &= \begin{cases}
\frac{1}{2} \parens{1 + \epx{-\frac{(a + cj) D\Nm{v}}{2m}} \chi_2(5 - 2nc)} &\tif m = 4\mu,\\
1 &\tif m \not= 4\mu.
\end{cases}
\end{align*}
\thlabel{prop:transformation_g_u}
\end{prop}

It is not difficult to see that if $m \not= \mu$ then $2^f | b + dj + c$ from $ad - bc = 1$ and $2^f | a + cj$ and so $\frac{(b + dj + c)/2^f}{(a+cj)/2^f} \bmod \frac{m}{\mu}$ is well-defined.
\begin{pnote}
To see that, suppose $f \geq 1$ so $m$ is even. Recall that $\gcd(m, c) = 1$ by \thref{lm:various_gcd} (ii) so $c$ must be odd. To see that $\val_2(a + cj) = \val_2(\mu)$ if $m \not= \mu$, observe that
$$\val_2(\mu) = \min\{\val_2(D), \val_2(a + cj)\}$$
since $\mu = \gcd(D, a + cj)$. Thus, if $\val_2(\mu) \not= \val_2(a + cj)$ then $\val_2(\mu) = \val_2(D) < \val_2(a + cj)$ and so $\val_2(\mu) = \val_2(m)$ whence we reach a contradiction $m = \mu$. In fact, we have shown that $f < \val_2(D) \leq 3$ (which implies $a + cj \not= 0$). From $2^f | a + cj$, we get $j \equiv \frac{-a}{c} \equiv -ac \bmod 2^f$ where the latter is because we always have $x^2 \equiv 1 \bmod 8$ if $x$ is odd and here $f < 3$. Thus, $b + dj + c \equiv b + d(-ac) + c \equiv b + (1 - da)c \equiv b + (-bc)c \equiv b - c^2 b \equiv b - b \equiv 0 \bmod 2^f$.
\end{pnote}
Also, note that such $\kappa$ always exists for $\gcd(n, \frac{m}{\mu}) \dv \gcd(n, m) = 1$ so we can solve for $\kappa$ by Chinese Remainder Theorem. 

\begin{proof}
For $g \in \MFDNp$, observe the following
\begin{enumerate}
\item $g_u = \frac{-i D^{k-3/2}}{a_u} \sum_{j = 0}^{D-1} \epx{\Nm{u} j} g|_{k-1} \sigma_j$
\item $g = \frac{i}{\sD} \sum_{v \in [\DK]} g_v|_{k-1} \Matx{D & 0 \\ 0 & 1}$
\item $g_v|_{k-1} \Matx{D & 0 \\ 0 & 1} \sigma_j = D^{1-k} \epx{- \Nm{v} j} g_v$
\end{enumerate}
and then use \thref{prop:transformation_plus_forms}, we arrive at
\begin{align*}
g_u|_{k-1} \sigma &= \frac{-i D^{k-3/2}}{a_u} \sum_{j = 0}^{D - 1} \sum_{s = 0}^{m-1} \frac{\epx{\Nm{u} j} \GS(\psi_m; s (a + cj) - \lambda)}{m \; \psi_n(a + cj)} g|_{k-1} \sigma_{n s + \kappa} \text{ by (i) and \thref{prop:transformation_plus_forms}}\\
&= \frac{-i D^{k-3/2}}{a_u} \sum_{j, s} \frac{\epx{\Nm{u} j} \GS(\psi_m; s (a + cj) - \lambda)}{m \psi_n(a + cj)} \frac{i}{\sD} \sum_{v \in [\DK]} g_v|_{k-1} \Matx{D & 0 \\ 0 & 1}|_{k-1} \sigma_{n s + \kappa} \text{ by (ii)}\\
&= \frac{-i D^{k-3/2}}{a_u} \sum_{j, s} \frac{\epx{\Nm{u} j} \GS(\psi_m; s (a + cj) - \lambda)}{m \psi_n(a + cj)} \frac{i}{\sD} \sum_{v \in [\DK]} D^{1-k} \epx{-\Nm{v} (ns + \kappa)} g_v \text{ by (iii)}\\
&= \frac{1}{D\aDNu} \sum_{v \in [\DK]} \sum_{j, s} \frac{\GS(\psi_m; s (a + cj) - \lambda) \epx{\Nm{u} j - \Nm{v} (ns + \kappa)}}{m \; \psi_n(a + cj)} g_v\\
&= \frac{1}{D\aDNu} \sum_v \sum_{j = 1}^{D - 1} \frac{\GS(\psi_m) \epx{\Nm{u} j - \Nm{v} \kappa}}{m \; \psi_n(a + cj)} \parens{\sum_{s = 0}^{m - 1} \psi_m(s (a + cj) - \lambda) \epx{-D\Nm{v} \frac{s}{m}}} g_v
\end{align*}
With the appropriate requirement for $\kappa$ as in the statement of the proposition, a careful case-by-case analysis yields the inner most sum
\begin{align*}
\sum_{s = 0}^{m - 1} &\psi_m(s (a + cj) - \lambda) \epx{-D\Nm{v} \frac{s}{m}} \\
&= \begin{cases}
\frac{m}{2} \psi_m(nc) \parens{1 + \epx{-\frac{(a + cj) D\Nm{v}}{2m}} \chi_2(5 - 2nc)} \; \delta_{\mu, \gcd(D\Nm{v},m)} &\tif m = 4\mu,\\
m \; \psi_m(nc) \; \delta_{\mu, \gcd(D\Nm{v},m)} &\tif m \not= 4\mu.
\end{cases}
\end{align*}
where
$$\delta_{X,Y} := \begin{cases}
1 &\tif X = Y,\\
0 &\totherwise
\end{cases}$$
is the usual Kronecker delta.

\begin{pnote}
Here is the derivation of the value of the sum over $s$ above. For simplicity, we set $A = -D|v|^2$ and let
$$V := \sum_{s = 0}^{m - 1} \psi_m(s (a + cj) - \lambda) \epx{A \frac{s}{m}}$$
which we are to evaluate. 

If $m = \mu$ (which always happen when $D$ is square-free), we have $m \dv a + cj$ so the factor $\psi_m(s (a + cj) - \ld) = \psi_m(-\ld) = \psi_m(nc)$ can be pulled out of the sum and we find
$$V = \psi_m(nc) \sum_{s = 0}^{m - 1} \epx{A \frac{s}{m}} = \psi_m(nc) \; m \; \delta_{A, 0}^{\bmod m}.$$

If $m \not= \mu$, we either have $m = 2\mu$ or $m = 4\mu$. Write $D = 2^e \; D'$ with $D'$ is odd and  square-free. Then $m = 2^e m'$ where $m' = \gcd(a + cj, D')$. Using Chinese Remainder Theorem, we find that\begin{align}
V &= \parens{\sum_{s \bmod m'} \psi_{m'}(s (a + cj) - \lambda) \epx{A(2^{-e} \bmod m')\frac{s}{m'}}} \parens{\sum_{s \bmod 2^e} \psi_{2^e}(s (a + cj) - \lambda) \epx{A h\frac{s}{2^e}}}
\label{eq:V_as_product_by_CRT}
\end{align}
where $h = m'^{-1} \bmod 2^e$ is odd. Here, we note a principal strategy to reduce an exponential sum into a product by Chinese Remainder Theorem which we shall exploit over and over again: Observe that $\epx{\Z} = 1$ so $\epx{r} = \epx{s}$ if $r, s \in \Q$ such that $r \equiv s \bmod \Z$. In particular, if $M = M_1 M_2 \cdots M_k$ with $\gcd(M_i, M_j) = 1$ for all $i \not= j$ then for any $x \in \Z$, we have
\begin{align}
\epx{\frac{x}{M}} &= \epx{\sum_{i = 1}^{k} \parens{\prod_{j \not= i} M_j^{-1} \bmod M_i} \frac{x}{M_i}}
\label{eq:split_exponential}
\end{align}
for one can check that
\begin{align*}
& \; \frac{x}{M} \equiv \sum_{i = 1}^{k} \parens{\prod_{j \not= i} M_j^{-1} \bmod M_i} \frac{x}{M_i} \qquad \mod \Z\\
\iff & \; x \equiv \sum_{i = 1}^{k} \parens{\prod_{j \not= i} M_j^{-1} \bmod M_i} x \prod_{j \not= i} M_j \qquad \mod M \Z\\
\iff & \; x \equiv \sum_{i = 1}^{k} \parens{\prod_{j \not= i} M_j^{-1} \bmod M_i} x \prod_{j \not= i} M_j \qquad \mod M_i \text{ for all } i \text{, by Chinese Remainder Theorem}\\
\iff & \; x \equiv \parens{\prod_{j \not= i} M_j^{-1} \bmod M_i} x \prod_{j \not= i} M_j \qquad \mod M_i \text{ for all } i
\end{align*}
which is evident. Now \eqref{eq:split_exponential} could be exploited to convert $\epx{r}$ to a product of simpler root of unity $\epx{r_i}$ with $r_i$ having smaller denominator. For instance, to derive \eqref{eq:V_as_product_by_CRT}, we use \eqref{eq:split_exponential} with $M = m, M_1 = m', M_2 = 2^e$ to write
$$\epx{A \frac{s}{m}} = \epx{(2^{-e} \bmod m') \frac{As}{m'} +  ((m')^{-1} \bmod 2^e) \frac{As}{2^e}}$$
to recognize that the summand in the definition of $V$ can be split as a product of two factors
$$\parens{\psi_{m'}(s (a + cj) - \lambda) \epx{(2^{-e} \bmod m') \frac{As}{m'}}} \parens{\psi_{2^e}(s (a + cj) - \lambda) \epx{((m')^{-1} \bmod 2^e) \frac{As}{2^e}}}$$
which only depend on $s_1 := s \bmod m'$ and $s_2 := s \bmod 2^e$ respectively; whence we get \eqref{eq:V_as_product_by_CRT} by a change of variable and Chinese Remainder Theorem.

The first factor on the RHS of \eqref{eq:V_as_product_by_CRT} can be computed explicitly as follow
\begin{align*}
\sum_{s \bmod m'} \psi_{m'}(s (a + cj) - \lambda) \epx{A(2^{-e} \bmod m')\frac{s}{m'}} &= \psi_{m'}(-\ld) \; m' \; \delta_{A, 0}^{\bmod m'}\\
&= \psi_{m'}(nc) \; m' \; \delta_{A, 0}^{\bmod m'}
\end{align*}
just like the case $m = \mu$: That is because we always have $m' \dv a + cj$ and so
$$\psi_{m'}(s (a + cj) - \ld) = \psi_{m'}(-\ld) = \psi_{m'}(nc)$$
from $-\ld \equiv \frac{-b - dj}{n} \bmod m'$ and
$$- b - dj \equiv -\frac{ad - 1}{c} - dj \equiv \frac{-ad+1-cdj}{c} \equiv \frac{1-d(a + cj)}{c} \equiv \frac{1}{c} \bmod m'.$$

It remains to compute the second factor
$$S := \sum_{s \bmod 2^e} \psi_{2^e}(s (a + cj) - \lambda) \epx{A h\frac{s}{2^e}}$$
on the RHS of \eqref{eq:V_as_product_by_CRT}. To do that, we write $a + jc = 2^f k$ with $k$ odd and $1 \leq f \leq 2$ and $f \leq e$. Note that $\mu = 2^f m'$ and $2 \dv a + cj$ so $\lambda = \frac{b + dj - \kappa(a + cj)}{n} \equiv \frac{b + dj}{n} \bmod 2$ is an odd number. Let us temporarily denote by $\psi$ for $\psi_{2^e}$.

Then identifying $s \bmod 2^e$ with $2^{e - f} s + t$ where $0 \leq s < 2^f$ and $0 \leq t < 2^{e - f}$ and notice that
$$\psi([2^{e - f} s + t] 2^f k - \ld)
= \psi(2^f k t - \ld)$$
we get
\begin{align*}
S &= \sum_{s = 0}^{2^f - 1} \quad \sum_{t = 0}^{2^{e - f} - 1} \quad \psi (2^f k t - \ld) \; \epx{A h \; \frac{2^{e - f} s + t}{2^e}}\\
&= \sum_{s = 0}^{2^f - 1} \quad \sum_{t = 0}^{2^{e - f} - 1} \quad \psi (2^f k t - \ld) \; \epx{A h \parens{\frac{s}{2^f} + \frac{t}{2^e}}}\\
&= \sum_{t = 0}^{2^{e - f} - 1} \psi (2^f k t - \ld) \epx{A h \frac{t}{2^e}} \parens{\sum_{s = 0}^{2^f - 1} \epx{A h \frac{s}{2^f}}}\\
&= \sum_{t = 0}^{2^{e - f} - 1} \psi(t 2^f k - \ld) \epx{A h \frac{t}{2^e}} \; 2^f  \; \delta_{A, 0}^{\bmod 2^f}\\
&=  2^f  \; \delta_{A, 0}^{\bmod 2^f} \parens{\sum_{t = 0}^{2^{e - f} - 1} \psi(2^f k t - \ld) \epx{A h \frac{t}{2^e}}}
\end{align*}

Assume now that $A \equiv 0 \bmod 2^f$. By change of variable $kt$ to $t$ and let $A' := \frac{A}{2^f} h k^{-1} \bmod 2^{e - f}$, we get
\begin{align*}
S &= 2^f \sum_{t = 0}^{2^{e - f} - 1} \psi(2^f t - \ld) \epx{A' \frac{t}{2^{e - f}}}\\
&= 2^f \psi(-\ld) \sum_{t = 0}^{2^{e - f} - 1} \psi(2^f t - \ld) \psi(-\ld) \epx{A' \frac{t}{2^{e - f}}}\\
&= 2^f \psi(-\ld) \sum_{t = 0}^{2^{e - f} - 1} \psi(\ld^2 - 2^f t \ld) \epx{A' \frac{t}{2^{e - f}}}
\end{align*}

\begin{itemize}
\item If $e - f = 1$ i.e. $m = 2\mu$ then we get from
$$\epx{A' \frac{t}{2^{e - f}}} = \epx{A' t / 2} = (-1)^{A't}$$
that
$$S = 2^f \psi(-\ld) \parens{1 + (-1)^{A'} \psi(\ld^2 - 2^f \ld)}$$
Recall that we have
$$\psi(n) = \begin{cases}
\jcb{-4}{n} = (-1)^{(n-1)/2}\\
\jcb{-8}{n} = (-1)^{(n^2 - 4n + 3)/8} &(= 1 \tif n \equiv 1, 3 \bmod 8)\\
\jcb{8}{n} = (-1)^{(n^2 - 1)/8} &(= 1 \tif n \equiv 1, 7 \bmod 8)
\end{cases}$$
So if $f = 1, e = 2$, we get
\begin{align*}
S &= 2 \psi(-\ld) \parens{1 + (-1)^{A'} (-1)^{((\ld^2 - 2 \ld - 1)/2}}\\
&= 2 \psi(-\ld) \parens{1 - (-1)^{A'}} \text{ since } \ld^2 \equiv 1 \bmod 8 \tif \ld \text{ odd}\\
&= 4 \psi(-\ld) \delta_{A', 1}^{\bmod 2}
\end{align*}

If $f = 2, e = 3$ then observe that we always have
$$\ld^2 - 4 \ld \equiv 5 \bmod 8.$$
(Write $\ld = 4x + y$ with $y = \pm 1$ then $\ld^2 \equiv y^2 \bmod 8$ and $4 \ld \equiv 4y \bmod 8$. If $y = 1$, $y^2 - 4y = 1 - 4 = -3 = 5 \bmod 8$. If $y = -1$ then $y^2 - 4y = 1 + 4 = 5 \bmod 8$.)
Note that we always have $\psi(5) = -1$ in both cases of $\psi = \jcb{\pm 8}{*}$ and so we need $2 \ndv A'$ in order for $S$ to be non-zero, in which case $S = 8 \psi(-\ld)$.

So in this case, we need $\val_2(A) = f$ for $S \not= 0$. Combine with $m' \dv A$, we need $\mu \dv A$ but $2\mu \ndv A$ for $S$ to be non-zero; which can be expressed as $\gcd(A, m) = \mu$. Thus
$$V = m \psi(-\ld) \psi_{m'}(nc) \delta_{\mu, \gcd(A, m)}.$$

If we can choose $\kappa$ such that
$$-\ld \equiv \frac{c}{n} \bmod 2^e$$
as well then we can merge $\psi(-\ld) \psi_{m'}(nc) = \psi_m(nc)$ to get the description
$$V = m \; \psi_m(nc) \delta_{\mu,\gcd(A, m)}$$
that is independent of $\kappa$ and $\lambda$. Well,
\begin{align*}
-\ld \equiv \frac{c}{n} \bmod 2^e
&\iff \kappa(a + cj) - (b + dj) \equiv c \bmod 2^e\\
&\iff \kappa 2^f k - b - dj \equiv c \bmod 2^e\\
&\iff \kappa \equiv \frac{(b + dj + c) / 2^f}{k} \bmod 2^{e-f} = \frac{m}{\mu}
\end{align*}
as long as $2^f \dv b + dj + c$.
Note that $\gcd(c, m) = 1$ so $c$ must be odd for $m \not= \mu$ to happen.
Since $2^f \dv \mu \dv a + cj$, $j \equiv -\frac{a}{c} \bmod 4$ and so $b + d j \equiv b + d(-\frac{a}{c}) \equiv \frac{cb - ad}{c} \equiv \frac{-1}{c} \equiv -c \bmod 2^f$ as $c^2 \equiv 1 \bmod 8$ when $c$ is odd.

\item If $e - f = 2$ i.e. $f = 1, e = 3, m = 4\mu$ then
\begin{align*}
S &= 2^f \psi(-\ld) \parens{1 + \epx{\frac{A'}{4}} \psi(\ld^2 - 2 \ld) + \epx{\frac{A'}{2}} \psi(\ld^2 - 4\ld) + \epx{\frac{3 A'}{4}} \psi(\ld^2 - 6\ld)}\\
&= 2^f \psi(-\ld) \parens{1 + \epx{\frac{A'}{4}} \psi(5 + 2 \ld) + (-1)^{A'} \psi(5) + (-1)^{A'} \epx{\frac{A'}{4}} \psi(5 - 2\ld)}\\
&\hspace{3cm} \text{ from previous observation that } \ld^2 - 4 \ld \equiv 5 \bmod 8\\
&= 2^f \psi(-\ld) \parens{1 - (-1)^{A'} + \epx{\frac{A'}{4}} \psi(5 + 2 \ld) \parens{1 + (-1)^{A'} \psi(25 - 4\ld^2)}}\\
&= 2^f \psi(-\ld) (1 - (-1)^{A'}) \parens{1 + \epx{\frac{A'}{4}} \psi(5 + 2 \ld)}\\
&= 2^{e - 1} \psi(-\ld) \parens{1 + \epx{\frac{A'}{4}} \psi(5 + 2 \ld)} \delta_{A',1}^{\bmod 2}
\end{align*}
Here, note that $\ld^2 = 4x + 1$ for some $x$ so $25 - 4 \ld^2 = 25 - 4(4x + 1) \equiv 1 - 4 - 16x \equiv -3 \equiv 5 \bmod 8$ and $\psi(5) = -1$. As before, we can always choose $\kappa$ so that $-\ld \equiv \frac{c}{n} \bmod 8$ (due to $f = 1$). Then
$$\psi(5 + 2 \ld) = \psi\parens{5 - 2 \frac{c}{n}} = \psi(5 - 2nc).$$
[Note that $5 - 2nc \equiv 3, 7 \bmod 8$ always so unfortunately the value does depend on $n$.] We recall
\begin{align*}
A' &= \frac{A}{2^f} h k^{-1} \bmod 2^{e - f}\\
&= \frac{A}{2} (m'^{-1} \bmod 8) \parens{\frac{a+cj}{2}}^{-1} \bmod 4\\
&= \frac{A}{2} m'^{-1} \frac{a+cj}{2} \bmod 4 &\text{ since } x^{-1} = x \bmod 4\\
&= \frac{A}{2} \frac{a+cj}{\mu} \bmod 4
\end{align*}
so
$$\epx{\frac{A'}{4}} = \epx{\frac{(a + cj) A}{8\mu}} = \epx{\frac{(a + cj) A}{2m}} \in \{\pm i\}.$$
\end{itemize}

In all cases, we have found that we need $\val_2(A) = \val_2(\mu)$ and $m' \dv A$ (which combines to the equivalent condition $\gcd(A, m) = \mu$; or one can also write $\mu \dv A, \; 2 \mu \ndv A$) in order that $V \not= 0$. As for the value, we found
$$V = \begin{cases}
\frac{m}{2} \psi_m(nc) \parens{1 + \epx{\frac{(a + cj) A}{2m}} \chi_2(5 - 2nc)} \; \delta_{\gcd(A,m), \mu} &\tif m = 4\mu,\\
m \; \psi_m(nc) \; \delta_{\gcd(A,m), \mu} &\tif m \not= 4\mu
\end{cases}$$
provided that we further insist
$$\kappa \equiv \frac{(b + dj + c)/2^f}{(a+cj)/2^f} \bmod \frac{m}{\mu}$$
in case $m \not= \mu$.
\end{pnote}
\end{proof}

\begin{remark}
We remark that in case $m = 4\mu$ (which can only happen if $8 \dv D$) and $\gcd(D\Nm{v},m) = \mu$, we have $\epx{-\frac{(a + cj) D\Nm{v}}{2m}} = \epx{-\frac{\frac{a + cj}{2} \frac{D\Nm{v}}{\mu}}{4}} \in \{\pm i\}$ only depends on $j \bmod 8$. Also, since $\chi_2(5) = -1$ for both possibilities of $\chi_2 = \jcb{\pm 8}{*}$, we have
$$\chi_2(5 - 2nc) =  (-1)^{\frac{nc+1}{4}} \chi_2(-1) = \begin{cases}
-\chi_2(-1) &\tif nc \equiv 1 \bmod 4,\\
\chi_2(-1) &\tif nc \equiv -1 \bmod 4.\\
\end{cases}$$
\end{remark}

\newcommand{\sumj}[1]{\sum_{\underset{\gcd(D\Nm{#1}, m) = \mu}{j = 0}}^{D - 1}}

\begin{prop}
Let $g \in \MFDNp$. Then $\phi_g = \sum g_u \theta_u$ is a Hermitian Jacobi form of level $N$ if and only if the equation
\begin{equation}
\sum_u \frac{M_{u,v}(\sigma) a_{w}}{D\aDNu} \sumj{w}  R_\sigma(w, j) \; \GS(\psi_m; nc) \; \psi_n(a + cj) \; \epx{\Nm{u} j - \Nm{w} \kappa} = \delta_{D\Nm{w},D\Nm{v}}^{\bmod D}
\label{eq:hjf_criterion}
\end{equation}
holds for any $\sigma \in \Gm_0(N)$ and any $v, w \in [\DK]$ such that $g_w \not= 0$.
\thlabel{prop:explicit_arithmetic_criterion_jacobi_form}
\end{prop}
Here, for any integers $X, Y, M$, $$\delta_{X, Y}^{\bmod M} := \begin{cases}
1 &\tif X \equiv Y \bmod M,\\
0 &\totherwise.
\end{cases}$$

\begin{proof}
\newcommand{\sumwpw}{\sum_{\underset{D\Nm{w'} \equiv D\Nm{w} \bmod D}{w' \in [\DK]}}}
\newcommand{\sumu}{\sum_{u \in [\DK]}}

From transformation formula for $\theta_u$, we find that $\phi_g$ is Jacobi form of level $N$ if and only if
$$\sumu M_{u,v}(\sigma) g_u|_{k-1} \sigma = g_v \qquad \text{for all } v \in [\DK] \tand \sigma \in \Gm_0(N)$$
By \thref{prop:transformation_g_u}, the above equation is the same as
$$\sumu \frac{M_{u,v}(\sigma)}{D\aDNu} \sum_{w \in [\DK]} \sumj{w} R_\sigma(w, j) \GS(\psi_m; nc) \psi_n(a + cj) \epx{\Nm{u} j - \Nm{w} \kappa} g_w = g_v.$$
Now observe that $g_w = g_{w'}$ if $D\Nm{w} \equiv D\Nm{w'} \bmod D$ and conversely, if $D\Nm{w} \not\equiv D\Nm{w'} \bmod D$ (or equivalently, $\Nm{w} \not= \Nm{w'}$ in $\Q/\Z$) then $g_w$ and $g_{w'}$ have disjoint Fourier expansion so if they are non-zero, they must be linearly independent. So grouping the summation over $w$ by $D\Nm{w} \bmod D$, we find that the above equation holds if and only if
\begin{align}
\sumu \frac{M_{u,v}(\sigma)}{D\aDNu} \sumwpw \sumj{w'} R_\sigma(w', j) \GS(\psi_m; nc) \psi_n(a + cj) \epx{\Nm{u} j - \Nm{w'} \kappa} = \delta_{D\Nm{w}, D\Nm{v}}^{\bmod D}
\label{eq:arithmetic_criterion_1}
\end{align}
for all $w \in [\DK]$ such that $g_w \not= 0$. Now if $D\Nm{w'} \equiv D\Nm{w} \bmod D$ then
\begin{itemize}
\item $\epx{-\Nm{w'} \kappa} = \epx{-\Nm{w} \kappa}$ as $\kappa \in \Z$;
\item $\gcd(D\Nm{w'}, m) = \gcd(D\Nm{w}, m)$ so the sum over $j$ such that $\gcd(D\Nm{w'}, m) = \mu$ goes through the same set of $j$'s independent of $w' \in [\DK]$; and also
\item $R_\sigma(w', j) = R_\sigma(w, j)$: obviously if $m \not= 4\mu$ and in case $m = 4\mu$, the only term that involve the $w$ in the definition of $R_\sigma$ is
$$\epx{-\frac{(a + cj) D\Nm{w'}}{2m}} \qquad \text{ which equals } \qquad \epx{-\frac{(a + cj) D\Nm{w}}{2m}}$$
as long as $D\Nm{w'} \equiv D\Nm{w} \bmod D$.
\end{itemize}
Thus, we can replace all occurrences of $w'$ in the summand of \eqref{eq:arithmetic_criterion_1} with $w$ and the \eqref{eq:arithmetic_criterion_1} reduces to the equivalent
$$\sumu \frac{M_{u,v}(\sigma)}{D\aDNu} \sumwpw \underbrace{\sumj{w} R_\sigma(w, j) \GS(\psi_m; nc) \psi_n(a + cj) \epx{\Nm{u} j - \Nm{w} \kappa}}_{S} = \delta_{D\Nm{w},D\Nm{v}}^{\bmod D}$$

One can now see that the summand $S$ is free of the summation variable $w'$ and so the middle sum over $w'$ is just $a_w S$ where $a_w$ is the number of $w' \in [\DK]$ such that $D\Nm{w'} \equiv D\Nm{w} \bmod D$. Hence, we have \eqref{eq:hjf_criterion}.
\end{proof}

\begin{remark}
Observe that the equation \eqref{eq:hjf_criterion} does not have modular forms in it:
It purely depends on the arithmetic of the field $K$.
Thus, if it is true, one could expect that it is true in general for all $\sigma \in \SL{2}(\Z)$ and $v, w \in [\DK]$.
(Note that even if the equation holds for all $\SL{2}(\Z)$, this does not implies that $\phi_g$ is Jacobi form of level 1; since the computation of $g|_{k-1} \sigma_j \sigma$ we used along the way is only valid for $\sigma \in \Gamma_0(N)$.)
\end{remark}

\begin{remark}
We remark here that one technique to show \eqref{eq:hjf_criterion} is true for higher level is to use existence of lifting for level 1 given by Krieg. In section 5 of \cite{Krieg1991}, Krieg defined the Eisenstein series
$$E_{k-1}^*(\tau) = \sum_{\underset{\gcd(m,n) = 1}{D = mn}} \psi_m(-1) f_{k-1}(\tau; \psi_m, \psi_n)$$
and he showed that
$$\MFDp = \C E_{k-1}^* \oplus \CF_{k-1}^+(D, \chi)$$
for all even $k > 2$. The Fourier coefficient of $E_{k-1}^*$ is given by
$$\a_\ell(E_{k-1}^*) = a_D(\ell) \sum_{d \dv \ell} \chi(d) d^{k-2}$$
for $(\ell, D) = 1$. Thus, $(E_{k-1}^*)_w \not= 0$ for all $w \in [\DK]$ such that $\gcd(D\Nm{w}, D) = 1$. Thus, by Krieg's isomorphism for the case $N = 1$, we know that \eqref{eq:hjf_criterion} is true for all $\sigma \in \Gm$ and all $w, v$ such that $\gcd(D\Nm{w}, D) = 1$. When $D$ is prime, the remaining class $w \in [\DK]$ such that $\gcd(D\Nm{w}, D) \not= 1$ is $w = 0$ and one easily works out the equation to show surjectivity in the case of interest in \cite{Klosin2018}. We will not pursue this direction.
\end{remark}

Our next result reduces the verification of the equation in \thref{prop:explicit_arithmetic_criterion_jacobi_form} to that of representatives of each equivalence class of $\Gm_0(N)$ under the equivalence relation induced by $\Gm_0(D)$, namely $\alpha \sim \beta$ if $\alpha = \beta \gamma$ for some $\gamma \in \Gm_0(D)$. 

\begin{prop}
If the equation \eqref{eq:hjf_criterion} in \thref{prop:explicit_arithmetic_criterion_jacobi_form} holds for $\sigma \in \SL{2}(\Z)$ and every $w, v \in [\DK]$ then it also holds for $\sigma \gm$ and every $w, v \in [\DK]$ for any $\gm \in \Gm_0(D)$. In particular, \eqref{eq:hjf_criterion} holds for all $\sigma \in \Gm_0(D)$.
\thlabel{prop:reduction_to_Gamma(D)_classes}
\end{prop}
\begin{proof}
\renewcommand{\Mxyzt}{\Matx{x & y \\ Dz & t}}
Fix a matrix $\gm = \Mxyzt \in \Gm_0(D)$. From
$$\sigma \gm = \Mabcd \Mxyzt = \Matx{ax + bDz & ay + bt\\cx + dDz & cy + dt}$$
we find that $\mu_{\sigma \gm}(j) = \mu_\sigma(j)$ which implies $m_{\sigma\gm}(j) = m_\sigma(j)$ and $n_{\sigma\gm}(j) = n_\sigma(j)$. For any $w \in [\DK]$ such that $\gcd(D\Nm{w}, m) = \mu$, it is easy to check that
$$R_{\sigma\gm}(w, j) = R_{\sigma}(w, j)$$
and finally, we can take
$$\kappa_{\sigma \gm}(j) = ty + t^2 \kappa_\sigma(j).$$

\begin{pnote}
We have
\begin{align*}
\kappa_{\sigma \gm}(j) &\equiv
\frac{ay + bt + (cy + dt) j}{ax + bDz + (cx + dDz) j} &\bmod n\\
&\equiv \frac{ay + bt + (cy + dt) j}{ax + c x j}\\
&\equiv \frac{t[y(a + cj) + t(b + dj)]}{a + c j}\\
&\equiv ty + t^2 \kappa_\sigma(j) &\bmod n
\end{align*}
and in case $m \not= \mu$, we have $\frac{m}{\mu} \dv D$ and so
$$f_{\sigma \gm} = \val_2(ax + bDz + (cx + dDz) j) = \val_2(ax + cxj) = \val_2(a + cj) = f_\sigma$$
because $\val_2(ax+cxj) < \val_2(D(bz + dzj))$ and thus we also have
\begin{align*}
\kappa_{\sigma \gm}(j) &\equiv \frac{(ay + bt + (cy + dt) j + (cx + dDz))/2^f}{(ax + bDz + (cx + dDz) j)/2^f}\\
&\equiv \frac{(ay + bt + (cy + dt) j + cx)/2^f}{(ax + cxj)/2^f} &\text{ since }D/2^f \equiv 0 \bmod m/\mu\\
&\equiv \frac{t (ay + bt + (cy + dt) j + cx)/2^f}{(a + cj)/2^f} &\text{ since } tx \equiv 1 \bmod D \Rightarrow tx \equiv 1 \bmod m/\mu\\
&\equiv \frac{(tay + bt^2 + tcyj + dt^2 j + c)/2^f}{(a + cj)/2^f}\\
&\equiv ty + t^2 \kappa_\sigma(j) \bmod m/\mu
\end{align*}
This shows that we can always take
$$\kappa_{\sigma \gm}(j) = ty + t^2 \kappa_\sigma(j)$$

Now notice that for any $w$ such that $\gcd(D\Nm{w}, m) = \mu$, we also have
$$R_{\sigma\gm}(w, j) = R_{\sigma}(w, j)$$
for if $m = 4\mu$,
\begin{align*}
&\frac{1}{2} \parens{1 + \epx{-\frac{(ax + bDz + (cx + dDz)j) D\Nm{v}}{2m}} \chi_2(5 - 2n(cx + dDz))}\\
&= \frac{1}{2} \parens{1 + \epx{-\frac{x (a + cj) D\Nm{v}}{2m}} \chi_2(5 - 2ncx)}
\end{align*}
so if $x \equiv 1 \bmod 4$, it is evident and if $x \equiv -1 \bmod 4$ then $\chi_2(5 - 2ncx) = -\chi_2(5 - 2nc)$ since $\chi_2(5 - 2ncx) = \chi_2(5 + 2ncx)$.
\end{pnote}

First, assume $z > 0$. In Lemma 3.1 of \cite{Klosin2018}, it was mentioned that $M_{u,v}(\gm) = \delta_{u, tv} \epx{xy |u|^2} \chi(t)$ and so from the fact that $M$ is a group homomorphism $M(\sigma\gm) = M(\sigma)M(\gm)$, one has $M_{u,v}(\sigma \gm) = M_{u,tv}(\sigma) \epx{xy\Nm{tv}} \chi(t)$.
So the left hand side of \eqref{eq:hjf_criterion} for $\sigma \gm$ is
\begin{align*}
&\sum_u \frac{M_{u,v}(\sigma\gm) a_{w}}{D\aDNu} \sumj{w} R_{\sigma\gm}(w, j) \GS(\psi_m; n(cx + dDz)) \\ &\hspace{5cm} \times \psi_n((ax + bDz) + (cx + dDz)j) \epx{\Nm{u} j - \Nm{w} \kappa_{\sigma \gm}}\\
&= \sum_u \frac{M_{u,tv}(\sigma) \epx{xy\Nm{tv}} \chi(t) a_{w}}{D\aDNu} \sumj{w} R_\sigma(w, j) \GS(\psi_m; ncx) \psi_n(ax + cxj) \\ &\hspace{8cm} \times \epx{\Nm{u} j - \Nm{w} (ty + t^2 \kappa_\sigma)}\\
&= \sum_u \frac{M_{u,tv}(\sigma) \epx{xyt^2\Nm{v}-ty \Nm{w}} \chi(tx) a_{w}}{D\aDNu} \sumj{w} \GS(\psi_m; nc) \psi_n(a + cj) \epx{\Nm{u} j - \Nm{tw} \kappa_\sigma}\\
&= \epx{ty(\Nm{v} - \Nm{w})} \underbrace{\sum_u \frac{M_{u,tv}(\sigma) a_{tw}}{D\aDNu} \sumj{w} \GS(\psi_m; nc) \psi_n(a + cj) \epx{\Nm{u} j - \Nm{tw} \kappa_\sigma}}_{\text{this is LHS of } \eqref{eq:hjf_criterion} \text{ with } tv, tw \text{ replacing } v, w \text{ respectively}}
\end{align*}
by observing that $xt \equiv 1 \bmod D$ whence $\epx{xyt^2\Nm{v}-ty \Nm{w}} = \epx{ty(\Nm{v} - \Nm{w})}$ as $(xt-1)\Nm{v} \in \Z$. Since $\gcd(t, D) = 1$, multiplication by $t$ induces automorphism of $[\DK]$ so that we have $a_{tw} =\aDNw$, $\gcd(D\Nm{tw},m) = \gcd(D\Nm{w}, m)$.

As indicated, the last sum is the left hand side of \eqref{eq:hjf_criterion} for $\sigma$ and $tv, tw$ in place of $v, w$. So we get $\delta_{D\Nm{tw},D\Nm{tv}}^{\bmod D}$ for the answer by the hypothesis that the equation holds for $\sigma$. Note that the same as $\delta_{D\Nm{tw},D\Nm{tv}}^{\bmod D} = \delta_{D\Nm{w},D\Nm{v}}^{\bmod D}$ again by $\gcd(t, D) = 1$. Evidently, if $D\Nm{w} \equiv D\Nm{v} \bmod D$ then $\epx{ty(\Nm{v} - \Nm{w})} = 1$. So \eqref{eq:hjf_criterion} holds for $\sigma \gamma$ and $v, w$.

The case $z = 0$ can be handled similarly. Note that this case includes the matrix $-I_2$ which accounts for the case where $z < 0$. Recall that $M_{w',v}(\gm) = \sign{x} \delta_{w',xv} \epx{xy\Nm{w'}}$ from the definition so that
\begin{align*}
M_{u,v}(\sigma \gm) &= \sum_{w'} M_{u,w'}(\sigma) M_{w',v}(\gm)\\
&= \sum_{w'} M_{u,w'}(\sigma) \sign{x} \delta_{w',xv} \epx{xy\Nm{w'}}\\
&= \sign{x} M_{u,xv}(\sigma) \epx{xy\Nm{xv}}
\end{align*}
and we get
\begin{align*}
&\sum_u \frac{\sign{x} M_{u,xv}(\sigma) \epx{xy\Nm{xv}} a_{w}}{D\aDNu} \sumj{w} R_\sigma(w, j) \GS(\psi_m; ncx) \psi_n(ax + cxj) \times \\ &\hspace{8cm} \times \epx{\Nm{u} j - \Nm{w} (ty + t^2 \kappa_\sigma)}\\
&= \sign{x} \chi(x) \epx{xy\Nm{xv} - \Nm{w} ty} \sum_u \frac{M_{u,xv}(\sigma) a_{w}}{D\aDNu} \sumj{w} R_\sigma(w, j) \GS(\psi_m; nc) \psi_n(a + cj) \times \\ &\hspace{12cm} \times \epx{\Nm{u} j - \Nm{tw} \kappa_\sigma)}
\end{align*}
The inner sum is again the left hand side of \eqref{eq:hjf_criterion} with $xv$ and $tw$ in place of $v, w$ respectively so we get $\delta_{D\Nm{xv},D\Nm{tw}}^{\mod D}$ as a result by assumption. Note that $z = 0$ implies $x = t = \pm 1$ so $\Nm{xv} = \Nm{x}$ and $\Nm{tw} = \Nm{w}$. Thus, if $D\Nm{xv} \equiv D\Nm{tw} \mod D$ then if $x = t = 1$ then $\sign{x} \chi(x) \epx{xy\Nm{xv} - \Nm{w} ty} = \sign{1} \chi(1) \epx{y\Nm{v} - \Nm{w} y} = 1$ obviously. Otherwise, if $x = t = -1$ then $\sign{x} \chi(x) \epx{xy\Nm{xv} - \Nm{w} ty} = \sign{-1} \chi(-1) \epx{-y\Nm{v} + \Nm{w} y} = 1$.

For the remaining statement, one easily check that the equation holds for $\sigma = I_2$ and all $v, w \in [\DK]$.
\end{proof}

\section{Hermitian \Maass lift}
\label{sec:proof_hermitian_sk_lift}

This section is dedicated to proving the main theorem of the article

\begin{thm}
Suppose that $\gcd(D, N) = 1$. Let $g \in \MFDNp$. For any $u \in [\DK]$, define
$$g_u(\tau) := \chi(N) \frac{-i \sD}{a_D(-D\Nm{u})} \sum_{\underset{\ell \equiv -D \Nm{u} \bmod D}{\ell = 0}}^{\infty} a_\ell(g) \epx{\frac{\ell \tau}{D}}$$
and
$$\phi_g := \sum_{u \in [\DK]} g_u \theta_u \quad \text{ where } \quad \theta_u(\tau, z, w) := \sum_{a \in u + \OK} \epx{\Nm{a} \tau + \cc{a} z + a w}.$$
Then $\phi_g \in \JF_{k,1}^*(N)$ and $\phi_g \mapsto g$ under the injective map $\JF_k^*(N) \rightarrow \MFDNp$ in \thref{prop:jacobi_to_elliptic}.

In particular, we have an isomorphism $\JF_k^*(N) \cong \MFDNp$.
\thlabel{thm:jacobi_forms_plus_form_isomorphism}
\end{thm}

By the results of previous section, we only need to show the equation \eqref{eq:hjf_criterion} holds for all representatives $\sigma \in \Gamma_0(1) = \SL{2}(\Z)$ of $\Gamma_0(1)/\Gamma_0(D)$ and all $v, w \in [\DK]$. (This is stronger than the criterion in \thref{prop:explicit_arithmetic_criterion_jacobi_form}.)
From the well--known representatives for $\Gm_0(1)/\Gm_0(D)$ and the fact that the class of the identity matrix is already settled in \thref{prop:reduction_to_Gamma(D)_classes}, we reduce the verification to that of $\sigma = \Mabcd$ with $c \dv D$ and $c > 0$.
Fix such an arbitrary $\sigma$ and two arbitrary classes $v, w \in [\DK]$. Let
\begin{align*}
A_u &:= \sum_{\underset{\gcd(D\Nm{w}, m) = \mu}{j \bmod D}} \frac{a_w}{a_u} R_\sigma(w, j) \; \GS(\psi_m; nc) \; \psi_n(a + cj) \; \epx{\Nm{u} j - \Nm{w} \kappa}
\end{align*}
and
\begin{align*}
A &:= \sum_{u \in [\DK]} \frac{M_{u,v}(\sigma)}{D} A_u.
\end{align*}
be a factor of the inner summand and the left hand side (respectively) of \eqref{eq:hjf_criterion}. Under our assumption on the lower left entry $c$ of $\sigma$, we shall evaluate $A_u$ and then $A$ explicitly to establish \eqref{eq:hjf_criterion}, namely to show that $A = \delta_{D\Nm{w}, D\Nm{v}}^{\bmod D}$.

For clarity and due to the fact that the complete set of representatives of $[\DK]$ are of different forms, we consider the two cases of odd and even discriminant $D$ separately. The proof strategy is pretty much the same; except for the fact that the even discriminant case is technically much more complicated.

In this section, to simplify our formula, we recycle the previous notation $\epx{\bullet}$ as follow: For any integer $M \not= 0$ and any rational number
$r \in \Z_{(M)}$ where
$$\Z_{(M)} := \braces{\frac{a}{b} \in \Q \suchthat a, b \in \Z, \gcd(a, b) = 1 \tand \gcd(b, M) = 1},$$
the notation $\epx{\frac{r}{M}}$ will always mean the root of unity $\epx{\frac{s}{M}} = e^{2\pi i s / M}$ where $s$ is any integer such that $r \equiv s \bmod M \Z_{(M)}$. For example, if $r$ is odd, then $\epx{\frac{\frac{1}{r}}{2^e}} = \epx{\frac{s}{2^e}}$ where $s$ is a multiplicative inverse of $r$ mod $2^e$ i.e. $s$ is any integer such that $s r \equiv 1 \bmod 2^e$. This way, we do not have to write the longer notation $\epx{\frac{r^{-1} \bmod 2^e}{2^e}}$ or resort to defining a lot of new symbol for multiplicative inverse.

\subsection{The odd discriminant case}

\let\ifAllProof\iffalse\let\ifFullProof\iftrue\let\ifPersonalNote\iffalse

\unless\ifdefined\IsMainDocument

\documentclass[12pt]{amsart}
\usepackage{amsmath,amssymb,iftex,cancel,xy,xcolor,verbatim}
\usepackage{theoremref}
\usepackage[normalem]{ulem}
\usepackage[margin=1in]{geometry}

\xyoption{all}

\newtheorem{thm}{Theorem}[section]
\newtheorem{lm}[thm]{Lemma}
\newtheorem{prop}[thm]{Proposition}
\newtheorem{cor}[thm]{Corollary}
\theoremstyle{definition}
\newtheorem{defn}[thm]{Definition}
\newtheorem{remark}[thm]{Remark}
\newtheorem{example}[thm]{Example}

\numberwithin{equation}{section}

\title{Hermitian Maass lift for General Level}
\date{}
\author{Lawrence (An Hoa) Vu}

\ifPersonalNote
	\ifPDFTeX
		\usepackage{mdframed}
		\newmdtheoremenv[
		    linewidth=2pt,
		    leftmargin=0pt,
		    innerleftmargin=0.4em,
		    rightmargin=0pt,
		    innerrightmargin=0.4em,
		    innertopmargin=-5pt,
		    innerbottommargin=3pt,
		    splittopskip=\topskip,
		    splitbottomskip=0.3\topskip,
		    skipabove=0.6\topsep
		]{pnote}{Personal note}
	\else
		\definecolor{OliveGreen}{rgb}{0,0.6,0}
		\newcounter{personaln}
		\newenvironment{pnote}{\color{OliveGreen}
		\stepcounter{personaln}
		\par\bigskip\noindent{\bfseries Personal note \arabic{personaln}:}
		\par\medskip}{\par\medskip}
	\fi
\else	
	\let\pnote\comment
	\let\endpnote\endcomment
\fi

\allowdisplaybreaks

\ifPDFTeX
\else
	\usepackage{yfonts}
	\renewcommand{\mathfrak}{\textfrak}
\fi

\newcommand{\LHS}{\text{LHS}}

\newcommand{\UHP}{\mathfrak{H}}
\newcommand{\C}{\mathbb{C}}
\newcommand{\R}{\mathbb{R}}
\newcommand{\Q}{\mathbb{Q}}
\newcommand{\Z}{\mathbb{Z}}

\newcommand{\term}[1]{\emph{#1}}

\newcommand{\Maass}{Maa\ss\ }

\newcommand{\Mat}[1]{M_{#1}}
\newcommand{\GGm}{\mathsf{G_m}}
\newcommand{\GL}[1]{\mathsf{GL}_{#1}}
\newcommand{\SL}[1]{\mathsf{SL}_{#1}}
\newcommand{\GU}[1]{\mathsf{GU}(#1,#1)}
\newcommand{\U}[1]{\mathsf{U}(#1,#1)}
\newcommand{\SU}[1]{\mathsf{SU}(#1,#1)}
\newcommand{\HerMat}[1]{\mathbb{S}_{#1}} \newcommand{\Res}{\text{Res}} 

\newcommand{\GS}{\mathbf{G}} \newcommand{\aDNu}{a_u} \newcommand{\aDNw}{a_w} 

\newcommand{\MF}{\mathfrak{M}}
\newcommand{\CF}{\mathfrak{S}}
\newcommand{\JF}{\mathfrak{J}}
\newcommand{\MFDp}{\MF_{k-1}^+(D, \chi)}
\newcommand{\MFDNp}{\MF_{k-1}^+(DN, \chi)}
\newcommand{\MFDN}{\MF_{k-1}(DN, \chi)}

\newcommand{\Matx}[1]{\begin{pmatrix} #1 \end{pmatrix}}

\newcommand{\ConjTran}[1]{#1^*} \newcommand{\ConjTranInv}[1]{\widehat{#1}} 

\newcommand{\epx}[1]{\mathsf{e}\left[#1\right]}
\newcommand{\sign}[1]{\mathsf{sign}\left[#1\right]}
\newcommand{\cc}[1]{\overline{#1}}

\newcommand{\OK}{\mathfrak{o}_K} \newcommand{\DK}{\mathfrak{d}_K} \newcommand{\Mabcd}{\Matx{a & b\\c & d}}
\newcommand{\Mxyzt}{\Matx{x & y\\z & t}}
\newcommand{\val}{\mathsf{val}}

\newcommand{\parens}[1]{\left( #1 \right)}
\newcommand{\braces}[1]{\left\{ #1 \right\}}
\newcommand{\suchthat}{\;\vline\;}
\newcommand{\tif}{\text{ if }}
\newcommand{\tand}{\text{ and }}
\newcommand{\twhere}{\text{ where }}
\newcommand{\totherwise}{\text{ otherwise }}
\newcommand{\tsince}{\text{ since }}
\newcommand{\teither}{\text{ either }}
\newcommand{\tforsome}{\text{ for some }}
\newcommand{\tor}{\text{ or }}

\newcommand{\Gm}{\Gamma}
\renewcommand{\a}{\alpha}
\renewcommand{\b}{\beta}
\newcommand{\gm}{\gamma}
\newcommand{\ld}{\lambda}
\newcommand{\eps}{\varepsilon}
\renewcommand{\theta}{\vartheta}
\renewcommand{\phi}{\varphi}

\newcommand{\dv}{\;|\;}
\newcommand{\ndv}{\;\nmid\;}
\newcommand{\smD}{\sqrt{-D}}
\newcommand{\sD}{\sqrt{D}}
\renewcommand{\gcd}{}
\newcommand{\jcb}[2]{(#1 \;|\; #2)}
\newcommand{\Tr}{\mathsf{Tr}}
\newcommand{\Nm}[1]{|#1|^2}  \newcommand{\Id}{\text{Id}}
\newcommand{\Gal}{\text{Gal}}

\renewcommand{\labelenumi}{(\roman{enumi})}

\begin{document}

\fi 
Let us consider the case where the discriminant $-D$ is odd. In this case, $D \equiv 3 \bmod 4$ and we also know that $D$ is square-free. The representatives for $[\DK]$ can be taken from
$$\braces{\frac{i x}{\sD} \suchthat x \in \Z/D\Z}.$$
Here, $x \in \Z/D\Z$ refers to the choice of $x \in \Z$ from any complete sets of representatives of $\Z/D\Z$. Note that $D\Nm{u} \bmod D$ is independent of the representative for $u$.

We record the following lemma which is a slight generalization of Lemma A in \cite{Krieg1991} Section 6.

\begin{lm}
Given odd prime $p$ and $x, y, z \in \Z$, $p \nmid z$, let $\psi = (* | p)$ be the Legendre symbol for the prime $p$. Then one has
$$\sum_{j = 1}^{p - 1} \psi(j) \epx{z \frac{j x^2 + j^{-1} y^2}{p}} = \GS(\psi; z) \frac{\psi(x^2) + \psi(y^2)}{1 + \psi(y^2)} \sum_{\underset{\gm^2 \equiv y^2 \bmod p}{\gm \bmod p}} \epx{\frac{2xz\gm}{p}}$$
\thlabel{lm:gauss_salie_sum}
\end{lm}
\ifAllProof
\begin{proof}
If $p \dv x, y$ then the left hand side is just $\sum_{j = 1}^{p - 1} (j | p) = 0$ and $\psi(x^2) + \psi(y^2) = 0$ on the right hand side.

When $p \dv x$ but $p \nmid y$, we can remove $j x^2$ and by a change of variable $t = j^{-1} y^2 z$, we get Gauss sum:
\begin{align*}
\sum_{j = 1}^{p - 1} \psi(j) \epx{z \frac{j^{-1} y^2}{p}} &= \sum_{t = 1}^{p - 1} \psi(t y^{-2} z^{-1}) \epx{\frac{t}{p}}\\
&= \psi(z) \sum_{t = 1}^{p - 1} \psi(t) \epx{\frac{t}{p}}\\
&= \GS(\psi; z)
\end{align*}
One can check that the remaining factor on the RHS is 1. Likewise for the case $p \dv y$ but $p \nmid x$.

For the remaining case $p \nmid xy$, we get a Salie sum $S(zx^2; zy^2)$ where
\begin{align*}
S(a;b) &= \sum_{j = 1}^{p - 1} (j | p) \epx{\frac{a j + b j^{-1}}{p}}\\
&= \begin{cases}
0 &\text{if } (ab|p) = -1\\
(ac|p) S(c;c) &\text{if } (ab|p) = 1 \tand ab = c^2 \bmod p\\
\end{cases}\\
S(c;c) &= (c|p) i^{((p-1)/2)^2} p^{1/2} (\epx{2c} + \epx{-2c})\\
&= \GS((*|p); c) \parens{\epx{\frac{2c}{p}} + \epx{\frac{-2c}{p}}}
\end{align*}
explicitly evaluated according to \cite{kwilliam}. Thus, we get $S(zx^2; zy^2) = (zx^2 xyz|p) \GS((*|p); xyz) (\epx{2xyz/p} + \epx{-2xyz/p}) = \GS(\psi; z) \parens{\epx{\frac{2xyz}{p}} + \epx{\frac{-2xyz}{p}}}$. The remaining factor matches the right hand side.
\end{proof}
\fi

Set $D^* = \frac{D}{c}$. We proceed to evaluate the left hand side of \eqref{eq:hjf_criterion}. Under our assumption that $c | D$ and $c > 0$, it is easy to evaluate the matrix explicitly
\begin{align}
M_{u,v}(\sigma) = \frac{-i}{\sD} \epx{\frac{a x^2 - 2xy + d y^2}{Dc}} \; \GS(\psi_c; a) \; \delta_{x,dy}^{\bmod c}
\label{eq:matrix_M_for_c_divides_D}
\end{align}
if $u, v \in [\DK]$ are represented by $u = [\frac{ix}{\sD}], v = [\frac{iy}{\sD}]$ for some $x, y \in \Z$.

\ifAllProof
\begin{proof}
Note that $u, v$ are purely imaginary so that $\cc{u} = -u, \cc{v} = -v$. Also, $D\Nm{u} = x^2$ and $D\Nm{v} = y^2$. We first simplify the summand in the definition of $M_{u,v}(\sigma)$ as
\begin{align*}
&\epx{\frac{1}{c}(a\Nm{\gm+u} - (\gm+u)\cc{v} - \cc{\gm+u} v + d\Nm{v})}\\
&= \epx{\frac{1}{c}(a\Nm{\gm} + a \cc{\gm} u + a \gm \cc{u} + a \Nm{u} - \gm \cc{v} - u \cc{v} - \cc{\gm} v - \cc{u} v + d \Nm{v})}\\
&= \epx{\frac{1}{c}(a\Nm{\gm} + a \cc{\gm} u - a \gm u + a \Nm{u} + \gm v + u v - \cc{\gm} v + u v + d\Nm{v})}\\
&= \epx{\frac{1}{c}(a\Nm{\gm} - (au - v) (\gm - \cc{\gm}) + a \Nm{u} + 2 u v + d\Nm{v})}\\
&= \epx{\frac{1}{c}(a \Nm{u} + 2 u v + d\Nm{v})} \epx{\frac{a\Nm{\gm} - (au - v) (\gm - \cc{\gm})}{c}}
\end{align*}
Now let $\omega = \frac{1 + i\sD}{2}$ and identify $\OK = \Z[\omega]$, we get
\begin{align*}
&\sum_{\gm \in \OK/c\OK} \epx{\frac{a\Nm{\gm} - (a u - v) (\gm - \cc{\gm})}{c}}\\
&= \sum_{\a,\b \in \Z/c} \epx{\frac{a(\a^2 + \a \b + \frac{D + 1}{4} \b^2) - (a \frac{ix}{\sD} - \frac{iy}{\sD}) (i\b\sD)}{c}} &\text{ change of variable } \gm = \a + \b \omega\\
&= \sum_{\a,\b \in \Z/c} \epx{\frac{a(\a^2 + \a \b + \frac{D + 1}{4} \b^2) + (a x - y) \b}{c}}\\
&= \sum_{\a, \b \in \Z/c} \epx{\frac{a(\a^2 + 2 \a \b + (D + 1) \b^2) + 2 (a x - y) \b}{c}} &\text{ since } 2 \nmid c, \text{ we can make change of variable } \b \text{ to } 2\b\\
&= \sum_{\b \in \Z/c} \sum_{\a \in \Z/c} \epx{\frac{a (\a + \b)^2 + 2 (a x - y) \b}{c}} &\text{ for } c \dv D\\
&= \sum_{\a, \b \in \Z/c} \epx{\frac{a \a^2 + 2 (a x - y) \b}{c}} &\text{ change of variable } \a \text{ for } \a + \b\\
&= \parens{\sum_{\a \in \Z/c} \epx{\frac{a \a^2}{c}}} \underbrace{\parens{\sum_{\b \in \Z/c} \epx{\frac{2 (a x - y) \b}{c}}}}_{c \; \delta_{ax,y}^{\bmod c}}
\end{align*}
By Chinese Remainder Theorem,
\begin{align*}
\sum_{\a \in \Z/c} \epx{\frac{a \a^2}{c}} &= \prod_{p \dv c} \parens{\sum_{\a \in \Z/p} \epx{\frac{a ((\frac{c}{p})^{-1} \bmod p) \a^2}{p}}}\\
&= \prod_{p \dv c} \GS\parens{\psi_p; a ((\frac{c}{p})^{-1} \bmod p)}\\
&= \GS(\psi_c; a)
\end{align*}
and thus
\begin{align*}
M_{u,v}(\sigma) &= \frac{-i}{c \sD} \epx{\frac{1}{c}(a\Nm{u} + 2uv + d\Nm{v})} \sum(...)\\
&= \frac{-i}{c \sD} \epx{\frac{a D\Nm{u} + 2 Du v + d D\Nm{v}}{Dc}} \; \GS(\psi_c; a) \; c \; \delta_{ax,y}^{\bmod c}\\
&= \frac{-i}{\sD} \epx{\frac{a x^2 - 2xy + d y^2}{Dc}} \; \GS(\psi_c; a) \; \delta_{x,dy}^{\bmod c}
\end{align*}

Note that we find in particular that if $c = D$ then $M_{u,v}(\sigma) = \frac{-i}{\sD} \epx{\frac{a x^2 - 2xy + d y^2}{D^2}} \; \GS(\psi_D; a) \; \delta_{x,dy}^{\bmod D} = \delta_{u,dv} \epx{ab\Nm{u}} \chi(d)$. \end{proof}
\fi

Choose any $h \in \Z$ such that $h \equiv \begin{cases}
b &\mod c\\
c^{-1} &\mod D^*
\end{cases}$
and for any $u \in [\DK]$, we set
$$F_u := \sum_{\underset{\gm^2 \equiv D\Nm{w} \bmod D^*}{\gm \bmod D^*}} \epx{\frac{2x h^2 \gm}{D^*}} \qquad \text{ if } u = \left[\frac{i x}{\sD}\right].$$
(The right hand side is the same for any $x \in \Z$ as long as $x^2 \equiv D\Nm{u} \bmod D$. Also, recall that we fix $w \in [\DK]$.)

Then using \thref{lm:gauss_salie_sum}, we find that for any fixed $u \in [\DK]$, the inner sum over $j$ on the left hand side of \eqref{eq:hjf_criterion} equals
\begin{align}
A_u &= c\; \delta_{D\Nm{u}, D\Nm{dw}}^{\bmod c}\; F_{u, w} \GS(\psi_{D^*}) \psi_{c}(a) \epx{-\Nm{w} dh - \frac{D\Nm{u}ah^2}{D^*}}.
\label{eq:D_odd_inner_sum_over_j}
\end{align}

Take representative $v = [\frac{iy}{\sD}] \in [\OK]$. By mean of \eqref{eq:matrix_M_for_c_divides_D} and \eqref{eq:D_odd_inner_sum_over_j} and change of variable from $u = [\frac{ix}{\sD}]$ to $x \bmod D$, one has
\begin{align*}
A &= \sum_{\underset{x^2 \equiv D\Nm{dw} \bmod c}{x \bmod D}} \frac{\frac{-i}{\sD} \epx{\frac{a x^2 - 2xy + d y^2}{Dc}} \; \GS(\psi_c; a) \; \delta_{x,dy}^{\bmod c}}{D} c F_{[\frac{ix}{\sD}]} \GS(\psi_{D^*}) \psi_{c}(a) \epx{-\Nm{w}dh - \frac{x^2 a h^2}{D^*}}\\
&= \sum_{\underset{\underset{x^2 \equiv D\Nm{dw} \bmod c}{x \equiv dy \bmod c}}{x \bmod D}} \frac{1}{D^*} \epx{\frac{a x^2 - 2xy + d y^2}{Dc} -\Nm{w}dh - \frac{x^2 a h^2}{D^*}} F_{[\frac{ix}{\sD}]}
\end{align*}
Here, note that $\GS(\psi_c) \GS(\psi_{D^*}) = i \sD$ since $\psi_c(-1) \psi_{D^*}(-1) = \chi(-1) = -1$.

Evidently, if $D\Nm{v} \equiv d^2y^2 \not\equiv D\Nm{dw} \bmod c$ then the last sum is an empty sum (there is no $x$ satisfying the two conditions) and thus is zero which is the same as $\delta_{D\Nm{v}, D\Nm{w}}^{\bmod D}$.

Let us consider the case $D\Nm{v} \equiv D\Nm{dw} \bmod c$ and the sum reduces to just sum over $x \bmod D$ such that $x \equiv dy \bmod c$. Making change of summation variable to  $x \bmod D^*$ by identifying $x$ with $cx + dy$ and we continue with
\begin{align*}
A &= \sum_{x \bmod D^*} \frac{1}{D^*} \epx{\frac{a (cx+dy)^2 - 2(cx+dy)y + d y^2}{Dc} - \Nm{w}dh - \frac{(cx+dy)^2 a h^2}{D^*}} F_{[\frac{i(cx+dy)}{\sD}]}\\
&= \sum_{x \bmod D^*} \frac{1}{D^*} \epx{\frac{a (cx+dy)^2 - 2(cx+dy)y + d y^2}{Dc} -\Nm{w}dh - \frac{(cx+dy)^2 a h^2}{D^*}} \\
&\hspace{9cm} \times \parens{\sum_{\underset{\gm^2 \equiv D\Nm{w} \bmod D^*}{\gm \bmod D^*}} \epx{\frac{2(cx+dy) h^2 \gm}{D^*}}}\\
&= \sum_{\underset{\gm^2 \equiv D\Nm{w} \bmod D^*}{\gm \bmod D^*}} \epx{-\frac{D\Nm{w} dh}{D} + \frac{bdy^2}{D} + \frac{2dyh^2 \gm - a d^2y^2 h^2}{D^*}} \frac{1}{D^*} \sum_{x \bmod D^*} \epx{(2by + 2 h \gm - 2dyah) \frac{x}{D^*}}
\end{align*}
The inner sum over $x \bmod D^*$ is either $D^*$ or $0$ depending on whether $D^* \dv 2by + 2 h \gm - 2dyah$ or not. Note that $D^* \dv 2by + 2 h \gm - 2dyah$ is equivalent to $\gm \equiv y \bmod D^*$.
And so the outer sum over $\gm$ is empty i.e. $A = 0$ if $y^2 = D\Nm{v} \not\equiv D\Nm{w} \bmod D^*$. We have thus shown that $A = 0$ if $D\Nm{v} \not\equiv D\Nm{w} \bmod D$ i.e. if either $D\Nm{v} \not\equiv D\Nm{w} \bmod D^*$ or $D\Nm{v} \not\equiv D\Nm{w} \bmod c$; which matches the right hand side $\delta_{D\Nm{w}, D\Nm{v}}^{\bmod D}$ of \eqref{eq:hjf_criterion} by Chinese Remainder Theorem.

In case $D\Nm{v} \equiv D\Nm{w} \bmod D$, the sum over $\gm$ is just the summand at $\gm = y$ i.e.
\begin{align*}
A &= \epx{-\frac{D\Nm{w} dh}{D} + \frac{bdy^2}{D} + \frac{2dyh^2y - a d^2y^2 h^2}{D^*}}\\
&= \epx{\frac{-hd + bd + 2 cdh^2 - a c d^2 h^2}{D} y^2}\\
&= 1
\end{align*}
because $-hd + bd + 2 cdh^2 - a c d^2 h^2 \equiv 0 \bmod D$ by definition of $h$. It is obvious mod $c$ since $h \equiv b \bmod c$. We have $c(-hd + bd + 2 cdh^2 - a c d^2 h^2) \equiv -d + cbd + 2d - ad^2 \equiv d - (ad - bc)d \equiv 0 \bmod D^*$ and so $-hd + bd + 2 cdh^2 - a c d^2 h^2 \equiv 0 \bmod D^*$ as well.

This concludes our verification of \eqref{eq:hjf_criterion}; thus establishing the lifting when $D$ is odd.

\subsection{The even discriminant case}

\let\ifAllProof\iffalse\let\ifFullProof\iftrue\let\ifPersonalNote\iffalse

\unless\ifdefined\IsMainDocument

\documentclass[12pt]{amsart}
\usepackage{amsmath,amssymb,iftex,cancel,xy,xcolor,verbatim}
\usepackage{theoremref}
\usepackage[normalem]{ulem}
\usepackage[margin=1in]{geometry}

\xyoption{all}

\newtheorem{thm}{Theorem}[section]
\newtheorem{lm}[thm]{Lemma}
\newtheorem{prop}[thm]{Proposition}
\newtheorem{cor}[thm]{Corollary}
\theoremstyle{definition}
\newtheorem{defn}[thm]{Definition}
\newtheorem{remark}[thm]{Remark}
\newtheorem{example}[thm]{Example}

\numberwithin{equation}{section}

\title{Hermitian Maass lift for General Level}
\date{}
\author{Lawrence (An Hoa) Vu}

\ifPersonalNote
	\ifPDFTeX
		\usepackage{mdframed}
		\newmdtheoremenv[
		    linewidth=2pt,
		    leftmargin=0pt,
		    innerleftmargin=0.4em,
		    rightmargin=0pt,
		    innerrightmargin=0.4em,
		    innertopmargin=-5pt,
		    innerbottommargin=3pt,
		    splittopskip=\topskip,
		    splitbottomskip=0.3\topskip,
		    skipabove=0.6\topsep
		]{pnote}{Personal note}
	\else
		\definecolor{OliveGreen}{rgb}{0,0.6,0}
		\newcounter{personaln}
		\newenvironment{pnote}{\color{OliveGreen}
		\stepcounter{personaln}
		\par\bigskip\noindent{\bfseries Personal note \arabic{personaln}:}
		\par\medskip}{\par\medskip}
	\fi
\else	
	\let\pnote\comment
	\let\endpnote\endcomment
\fi

\allowdisplaybreaks

\ifPDFTeX
\else
	\usepackage{yfonts}
	\renewcommand{\mathfrak}{\textfrak}
\fi

\newcommand{\LHS}{\text{LHS}}

\newcommand{\UHP}{\mathfrak{H}}
\newcommand{\C}{\mathbb{C}}
\newcommand{\R}{\mathbb{R}}
\newcommand{\Q}{\mathbb{Q}}
\newcommand{\Z}{\mathbb{Z}}

\newcommand{\term}[1]{\emph{#1}}

\newcommand{\Maass}{Maa\ss\ }

\newcommand{\Mat}[1]{M_{#1}}
\newcommand{\GGm}{\mathsf{G_m}}
\newcommand{\GL}[1]{\mathsf{GL}_{#1}}
\newcommand{\SL}[1]{\mathsf{SL}_{#1}}
\newcommand{\GU}[1]{\mathsf{GU}(#1,#1)}
\newcommand{\U}[1]{\mathsf{U}(#1,#1)}
\newcommand{\SU}[1]{\mathsf{SU}(#1,#1)}
\newcommand{\HerMat}[1]{\mathbb{S}_{#1}} \newcommand{\Res}{\text{Res}} 

\newcommand{\GS}{\mathbf{G}} \newcommand{\aDNu}{a_u} \newcommand{\aDNw}{a_w} 

\newcommand{\MF}{\mathfrak{M}}
\newcommand{\CF}{\mathfrak{S}}
\newcommand{\JF}{\mathfrak{J}}
\newcommand{\MFDp}{\MF_{k-1}^+(D, \chi)}
\newcommand{\MFDNp}{\MF_{k-1}^+(DN, \chi)}
\newcommand{\MFDN}{\MF_{k-1}(DN, \chi)}

\newcommand{\Matx}[1]{\begin{pmatrix} #1 \end{pmatrix}}

\newcommand{\ConjTran}[1]{#1^*} \newcommand{\ConjTranInv}[1]{\widehat{#1}} 

\newcommand{\epx}[1]{\mathsf{e}\left[#1\right]}
\newcommand{\sign}[1]{\mathsf{sign}\left[#1\right]}
\newcommand{\cc}[1]{\overline{#1}}

\newcommand{\OK}{\mathfrak{o}_K} \newcommand{\DK}{\mathfrak{d}_K} \newcommand{\Mabcd}{\Matx{a & b\\c & d}}
\newcommand{\Mxyzt}{\Matx{x & y\\z & t}}
\newcommand{\val}{\mathsf{val}}

\newcommand{\parens}[1]{\left( #1 \right)}
\newcommand{\braces}[1]{\left\{ #1 \right\}}
\newcommand{\suchthat}{\;\vline\;}
\newcommand{\tif}{\text{ if }}
\newcommand{\tand}{\text{ and }}
\newcommand{\twhere}{\text{ where }}
\newcommand{\totherwise}{\text{ otherwise }}
\newcommand{\tsince}{\text{ since }}
\newcommand{\teither}{\text{ either }}
\newcommand{\tforsome}{\text{ for some }}
\newcommand{\tor}{\text{ or }}

\newcommand{\Gm}{\Gamma}
\renewcommand{\a}{\alpha}
\renewcommand{\b}{\beta}
\newcommand{\gm}{\gamma}
\newcommand{\ld}{\lambda}
\newcommand{\eps}{\varepsilon}
\renewcommand{\theta}{\vartheta}
\renewcommand{\phi}{\varphi}

\newcommand{\dv}{\;|\;}
\newcommand{\ndv}{\;\nmid\;}
\newcommand{\smD}{\sqrt{-D}}
\newcommand{\sD}{\sqrt{D}}
\renewcommand{\gcd}{}
\newcommand{\jcb}[2]{(#1 \;|\; #2)}
\newcommand{\Tr}{\mathsf{Tr}}
\newcommand{\Nm}[1]{|#1|^2}  \newcommand{\Id}{\text{Id}}
\newcommand{\Gal}{\text{Gal}}

\renewcommand{\labelenumi}{(\roman{enumi})}

\begin{document}

\fi 
Now we move on to the even discriminant case. In this case the representatives for $[\DK]$ can be taken from
$$\braces{
\frac{x_1}{2} + \frac{ix_2}{\sD} \suchthat
x_1 \in \Z/2\Z \tand
x_2 \in \Z/\parens{\frac{D}{2}}\Z
}$$

Decompose $D = 2^e D', c = 2^f c'$ where $e = \val_2(D), f = \val_2(c)$ are 2-adic valuation of $D$ and $c$ respectively; in other words, $2 \ndv c'$ and $2 \ndv D'$. Let $c^*$ be the $c$ component of $D$ i.e. $c^* = 2^e c'$ if $f \geq 1$ and $c^* = c = c'$ if $f = 0$ i.e. $2 \ndv c$. (Note that $D'$ is square-free so $c'$ is also square-free.) Let $D^* := \frac{D}{c^*}$, $f_1 := \val_2(\gcd(D^*, D\Nm{w}))$,
$$f_2 := e - \log_2\parens{\frac{c^*}{c'}} = \begin{cases}
0 &\tif 2 \dv c,\\
e &\tif 2 \ndv c.
\end{cases}$$
Note that $f_2$ is the power of 2 dividing $D^*$ for
$$D^* = \frac{D}{c^*} = \frac{2^e D'}{2^{e-f_2} c'} = 2^{f_2} \frac{D'}{c'}$$
so
$$f_1 = \min(f_2, \val_2(D\Nm{w})).$$
For any $u \in [\DK]$, set
\begin{align*}
E_u &:=
\begin{cases}
\GS\parens{\chi_2; D' c} \chi_2(D\Nm{u} + D \Nm{w}) &\tif 2 \ndv c\\
\displaystyle
\delta_{D\Nm{u}-D\Nm{w},c}^{\bmod 2^{e-1}} \; 2^{e-1} \; \chi_2(a) \; \epx{- \frac{D' D\Nm{w} ab}{2^e}}\\
\qquad \times \parens{1 + \epx{\frac{D' (D\Nm{u} -  D\Nm{w}(2bc + 1) - D\Nm{w} c d)}{2^e}} \chi_2(1 + ac)} &\tif 2 \dv c
\end{cases}\\
F_u &:= \sum_{\underset{\gm^2 \equiv D\Nm{w} \bmod \frac{D'}{c'}}{\gm \bmod \frac{D'}{c'}}} \epx{\frac{2 x \frac{1}{c c^*} 2^{-f_2} \gm}{\frac{D'}{c'}}} \qquad \text{ for any } x \text{ such that } x^2 \equiv D\Nm{u} \bmod \frac{D'}{c'}\\
E'_u &:= \begin{cases}
1 &\tif e - f_1 = 0,\\
(-1)^{D\Nm{u}} &\tif e - f_1 = 1,\\
(-1)^{D\Nm{u}/2} &\tif e - f_1 = 2 \tand D\Nm{u} \equiv 0 \bmod 2,\\
(-1)^{\frac{D\Nm{u} + D\Nm{w}/2}{2}} \chi_2(-1) &\tif e - f_1 = 2 \tand D\Nm{u} \equiv 1 \bmod 2
\end{cases}\\
K_u &:= \epx{-\frac{D\Nm{u} \frac{ab}{D/c'}}{c'} - \frac{D\Nm{u} \frac{a}{c c^*}}{D^*} - \frac{D\Nm{w} \frac{d}{c c^*}}{D^*}}\\
&= \epx{-\frac{D\Nm{u} \frac{ab}{D/c'}}{c'} - \frac{\frac{1}{2^{f_2}} D\Nm{u} \frac{a}{c c^*}}{D'/c'} - \frac{\frac{1}{2^{f_2}} D\Nm{w} \frac{d}{c c^*}}{D'/c'} - \frac{\frac{1}{D'/c'} D\Nm{u} \frac{a}{c c^*}}{2^{f_2}}  - \frac{\frac{1}{D'/c'} D\Nm{w} \frac{d}{c c^*}}{2^{f_2}}}
\end{align*}

Then with \thref{lm:gauss_salie_sum} and a lot of computations, we find that the inner sum over $j$ on the left hand side of \eqref{eq:hjf_criterion} is likewise given by
\begin{align*}
A_u &= B_u + (1 - \delta_{f_1, 0}) C_u
\end{align*}
where
\begin{align*}
B_u &= c' \; \delta_{D\Nm{u}, D\Nm{dw}}^{\bmod c'} F_u K_u E_u \frac{1 + \chi_2(D\Nm{w})}{1 + \chi_2(D\Nm{u})} \;
\GS(\psi_{\frac{D^*}{2^{f_2}}}; 2^{f_2} c^* c) \psi_{c'}(a)
\\
C_u &= c \; \delta_{D\Nm{u},D\Nm{dw}}^{\bmod c} F_u K_u E'_u \frac{1}{1 + \chi_2(D\Nm{u})} \GS(\psi_{D^*}) \psi_{c}(a)
\end{align*}
In the formula above, the $B_u$ and $(1 - \delta_{f_1, 0}) C_u$ are the partition of the sum over $j$ in the definition of $A_u$ into those such that $2 \ndv \gcd(D\Nm{w}, m)$ and $2 \dv \gcd(D\Nm{w}, m)$ respectively. Certainly, there exists $j$ such that  $2 \dv \gcd(D\Nm{w}, m)$
only when $2 \dv \gcd(D\Nm{w}, D^*)$ i.e. $f_1 \not= 0$; the fact we indicated by the factor $(1 - \delta_{f_1, 0})$ in the formula.

Note that $1 - \delta_{0,f_1} \not= 0$ only when $c$ is odd and $2 \dv D\Nm{w}$; in which case, $f_2 = e$,  $\GS(\psi_{\frac{D^*}{2^{f_2}}}; 2^{f_2} c^* c) = \GS(\psi_{\frac{D^*}{2^e}}; 2^e)$, $\GS(\psi_{D^*}) = \GS(\psi_{\frac{D^*}{2^e}}; 2^e) \GS(\chi_2; \frac{D^*}{2^e})$ and notice that $\frac{D^*}{2^e} = \frac{D}{c^* 2^e} = \frac{D'}{c'}$ so if $1 - \delta_{0,f_1} \not= 0$, we have a simpler formula
\begin{align*}
A_u &= c' \; \delta_{D\Nm{u}, D\Nm{dw}}^{\bmod c'} F_u K_u \frac{1 + \chi_2(D\Nm{w})}{1 + \chi_2(D\Nm{u})} \;
\GS(\psi_{\frac{D^*}{2^{f_2}}}; 2^{f_2} c^* c) \psi_{c'}(a) [E_u + (1 - \delta_{f_1, 0}) E_u' \GS(\chi_2; D'/c')].
\end{align*}

As for the matrix, let $u = [\frac{x_1}{2} + \frac{i x_2}{\sD}]$ and $v = [\frac{y_1}{2} + \frac{i y_2}{\sD}]$ be representatives then we have
\begin{align*}
M_{u,v}(\sigma) &= \epx{\frac{\frac{1}{D/c'} b dy_2^2}{c'}} \frac{ \; \delta_{x_2, dy_2}^{\bmod c'} \; \GS(\psi_{c'}; a 2^f)}{i \sD}
\epx{\frac{\frac{1}{2^{e+f} c'^2}(ax_2^2 - 2x_2y_2 + d y_2^2)}{D'/c'}} M_{u,v}^*(\sigma)
\end{align*}
where the factor
\begin{align*}
M_{u,v}^*(\sigma) &:=2^{f-2} \delta_{x_2, d y_2}^{\bmod 2^{f - 1}}
\epx{
    \frac{\frac{1}{c'}(ax_1^2 - 2x_1y_1 + dy_1^2)}{2^{f+2}}
   + \frac{\frac{1}{D'c'} (ax_2^2 - 2x_2y_2 + d y_2^2)}{2^{e + f}}
}
\end{align*}
Here, we take a convention that $\delta_{X,Y}^{\bmod 1/2} = 1$ for all $X, Y \in \Z$.

Since $D = 2^e \; c' \; (\frac{D'}{c'})$ and the three factors are relatively prime,
$$\delta_{D\Nm{v}, D\Nm{w}}^{\bmod D} = \delta_{D\Nm{v}, D\Nm{w}}^{\bmod 2^e} \times \delta_{D\Nm{v}, D\Nm{w}}^{\bmod c'} \times \delta_{D\Nm{v}, D\Nm{w}}^{\bmod D'/c'}.$$
Similar to the odd discriminant case, we show that $A = \delta_{D\Nm{v}, D\Nm{w}}^{\bmod D}$ by showing that $A = 1$ only when $D\Nm{v} \equiv D\Nm{w} \bmod c'$, $D\Nm{v} \equiv D\Nm{w} \bmod \frac{D'}{c'}$ and $D\Nm{v} \equiv D\Nm{w} \bmod 2^e$; and $A = 0$ otherwise.

In all cases, we see that in order that $\frac{M_{u,v}(\sigma)}{D} A_u \not= 0$, $u$ must satisfies the two congruences $D\Nm{u} \equiv D\Nm{dw} \bmod c'$ and $x_2 \equiv dy_2 \bmod c'$. The second condition implies $D\Nm{u} \equiv D\Nm{dv} \bmod c'$. Thus, if $D\Nm{v} \not\equiv D\Nm{w} \bmod c'$, $A = 0$. So let us now assume that $D\Nm{v} \equiv D\Nm{w} \bmod c'$ and the sum reduces to over $u$ such that $x_2 \equiv dy_2 \bmod c'$ (for then $D\Nm{u} \equiv D\Nm{dw} \bmod c'$ holds automatically). Identifying $u$ with a pair $(x_1, x_2) \in \Z/2\Z \times \Z/(\frac{D}{2})\Z$, we can rewrite the sum
\begin{align*}
A &= \frac{\GS(\psi_{c'}; 2^f) \GS(\psi_{\frac{D'}{c'}}; 2^{e+f})}{i 2^e \sD} A^+ A^- \end{align*}
where
\newcommand{\sumUModTwo}{\sum_{\underset{x_2 \bmod 2^{e-1}}{x_1 \bmod 2}}}
\begin{align*}
A^+ &= \frac{1}{D'/c'} \sum_{\underset{x_2 \equiv dy_2 \bmod c'}{x_2 \bmod D'}}
F_u \epx{\frac{\frac{1}{2^{e+f} c'^2}(ax_2^2 - 2x_2y_2 + d y_2^2)}{D'/c'} - \frac{\frac{1}{2^{f_2}} D\Nm{u} \frac{a}{c c^*}}{D'/c'} - \frac{\frac{1}{2^{f_2}} D\Nm{w} \frac{d}{c c^*}}{D'/c'}}\\
A^- &= \sumUModTwo \epx{- \frac{\frac{1}{D'/c'} D\Nm{u} \frac{a}{c c^*}}{2^{f_2}}  - \frac{\frac{1}{D'/c'} D\Nm{w} \frac{d}{c c^*}}{2^{f_2}}} \frac{1 + \chi_2(D\Nm{w})}{1 + \chi_2(D\Nm{u})} \times \\ &\hspace{5cm} \times [E_u + (1 - \delta_{f_1, 0}) E_u' \GS(\chi_2; D'/c')] M_{u,v}^*(\sigma)
\end{align*}
We shall show that
\begin{align}
A^+ = \delta_{D\Nm{v}, D\Nm{w}}^{\bmod D'/c'} \label{eq:value_of_A+}
\end{align}
and
\begin{align}
\frac{\GS(\psi_{c'}; 2^f) \GS(\psi_{\frac{D'}{c'}}; 2^{e+f})}{i 2^e \sD} A^- = \delta_{D\Nm{v}, D\Nm{w}}^{\bmod 2^e} \label{eq:value_of_A-}
\end{align}
which completes the proof that $A = \delta_{D\Nm{v}, D\Nm{w}}^{D}$. But first, using property of Gauss sum and quadratic reciprocity, one can simplify \eqref{eq:value_of_A-} to
\begin{align}
A^- = \delta_{D\Nm{v}, D\Nm{w}}^{\bmod 2^e} (-1)^{\frac{(c'-1)(D'/c' - 1)}{4}} \psi_{D'}(2)^f \psi_{c'}(2^e) 2^e \GS(\chi_2; D')
\label{eq:value_of_A-_simplified}
\end{align}

We compute $A^+$ to establish \eqref{eq:value_of_A+}:
\begin{align*}
A^+ &= \frac{1}{D'/c'} \sum_{\underset{x_2 \equiv dy_2 \bmod c'}{x_2 \bmod D'}}
\sum_{\underset{\gm^2 \equiv D\Nm{w} \bmod \frac{D'}{c'}}{\gm \bmod \frac{D'}{c'}}} \epx{\frac{2 x_2 \frac{1}{c c^*} 2^{-f_2} \gm}{D'/c'}} \times \\ &\hspace{5cm} \times \epx{\frac{\frac{1}{2^{e+f}c'^2}(\cancel{ax_2^2} - 2x_2y_2 + d y_2^2)}{D'/c'} - \cancel{\frac{\frac{1}{2^{f_2}} x_2^2 \frac{a}{c c^*}}{D'/c'}} - \frac{\frac{1}{2^{f_2}} D\Nm{w} \frac{d}{c c^*}}{D'/c'}}\\
&\hspace{8cm}\tsince 2^{f_2} c c^* = 2^{f_2} 2^f c' 2^{e - f_2} c' = 2^e c'^2\\
&= \frac{1}{D'/c'} \epx{\frac{\frac{1}{2^{e+f}c'^2}d y_2^2}{D'/c'} - \frac{\frac{1}{2^{e+ f} c'^2} d D\Nm{w}}{D'/c'}} \sum_{\underset{\gm^2 \equiv D\Nm{w} \bmod \frac{D'}{c'}}{\gm \bmod \frac{D'}{c'}}} \sum_{\underset{x_2 \equiv dy_2 \bmod c'}{x_2 \bmod D'}}
 \epx{\frac{2\frac{1}{2^{e+f}c'^2}(\gm - y_2) x_2}{D'/c'}}\\
&= \frac{1}{D'/c'} \epx{\frac{\frac{1}{2^{e+f}c'^2}d y_2^2}{D'/c'} - \frac{\frac{1}{2^{e+ f} c'^2} d D\Nm{w}}{D'/c'}} \sum_{\underset{\gm^2 \equiv D\Nm{w} \bmod \frac{D'}{c'}}{\gm \bmod \frac{D'}{c'}}} \frac{D'}{c'} \delta_{\gm, y_2}^{\bmod D'/c'}\\
&= \frac{1}{D'/c'} \frac{D'}{c'} \delta_{y_2^2, D\Nm{w}}^{\bmod D'/c'} \epx{\frac{\frac{1}{2^{e+f}c'^2}d y_2^2}{D'/c'} - \frac{\frac{1}{2^{e+ f} c'^2} d D\Nm{w}}{D'/c'}}\\
&= \delta_{D\Nm{v}, D\Nm{w}}^{\bmod D'/c'}
\end{align*}
Thus, we showed that $A = 0$ if $D\Nm{v} \not\equiv D\Nm{w} \bmod D'/c'$.

To prove \eqref{eq:value_of_A-_simplified}, observe that it is in principle a finite check since we have finitely many (but a lot of) possibilities for $e, f$ and $a, b, c, d, \frac{D'}{c'}, D\Nm{w}, D\Nm{v} \mod 2^e$.
Note that $D\Nm{u} \equiv 2^{e-2} D' x_1^2 + x_2^2 \bmod 2^e$ only depend on $x_1 \bmod 2$ and $x_2 \bmod 2^{e-1}$. Also, we must have $\chi_2(D\Nm{u}) \not= -1$ so $\chi_2(D\Nm{u}) = 0$ if $2 \dv D\Nm{u}$ and is 1 otherwise. In other words, $\chi_2(D\Nm{u}) = \delta_{D\Nm{u},1}^{\bmod 2}$. In particular, if $e = 2$ then $D\Nm{u} \equiv 1 \bmod 4$ if $2 \ndv D\Nm{u}$. For illustration, we establish \eqref{eq:value_of_A-_simplified} in the case $e = 2$. The case $e = 3$ is similarly solved by case-by-case analysis.

From now on, we assume $e = 2$ and ``left hand side'' (LHS) and ``right hand side'' (RHS) refers to that of equation \eqref{eq:value_of_A-_simplified}. As the matrix $\sigma$ is fixed, we drop it in the notation $M_{u,v}^*(\sigma)$.

Note that when $e = 2$, we must have $D' \equiv 1 \bmod 4$ so that $D = 2^e D' = 4 D'$ is a fundamental discriminant.

\subsubsection{The case $f = 2$ and $e = 2$}

Note that in this case, $c$ is even so $a, d$ are both odd. Then
\begin{align}
A^- = \sumUModTwo \frac{1 + \chi_2(D\Nm{w})}{1 + \chi_2(D\Nm{u})} E_u M_{u,v}^*(\sigma).
\label{eq:A_minus_when_2_divides_c}
\end{align}
since $f_2 = 0$ and $1 - \delta_{f_1,0} = 0$. (Note that the above holds as long as $2 \dv c$ i.e. $f \geq 1$.)

By definition, $M_{u,v}^*(\sigma) \not= 0$ only when $x_1 \equiv dy_1 \bmod 2$ and $x_2 \equiv dy_2 \bmod 2$ in which case $\chi_2(D\Nm{u}) = \chi_2(d^2 D\Nm{v}) = \chi_2(D\Nm{v})$ and \eqref{eq:A_minus_when_2_divides_c} reduces to
\newcommand{\sumUModTwoCondition}{\sum_{\underset{x_1 = dy_1, x_2 \equiv dy_2 \bmod 2}{x_2 \bmod 2^{e-1}}}}
\begin{align}
A^- &= \sumUModTwoCondition \frac{1 + \chi_2(D\Nm{w})}{1 + \chi_2(D\Nm{v})} E_u M_{u,v}^*(\sigma) \notag\\
&= \frac{1 + \chi_2(D\Nm{w})}{1 + \chi_2(D\Nm{v})} \sumUModTwoCondition E_u M_{u,v}^*(\sigma)
\label{eq:A_minus_when_2_divides_c_reduced}
\end{align}
Likewise, observe that $E_u \not= 0$ only when $D\Nm{u} - D\Nm{w} \equiv c \bmod 2^{e-1}$ so we must have $D\Nm{v} \equiv D\Nm{w} \bmod 2^{e-1}$ in order that $A^- \not= 0$.
It also follows from the definition of
\begin{align*}
E_u &= \delta_{D\Nm{u}-D\Nm{w},c}^{\bmod 2^{e-1}} \; 2^{e-1} \; \chi_2(a) \; \epx{- \frac{D' D\Nm{w} ab}{2^e}} \times \\& \hspace{3cm} \times \parens{1 + \epx{\frac{D' (D\Nm{u} -  D\Nm{w}(2bc + 1) - D\Nm{w} c d)}{2^e}} \chi_2(1 + ac)}
\end{align*}
that we further need $D\Nm{v} \equiv D\Nm{w} \bmod 2^e$ in order that the last factor $1 + \epx{...} \chi_2(1 + ac) \not= 0$. Hence, $A^- = 0$ unless $D\Nm{v} \equiv D\Nm{w} \bmod 2^e$.

Now let us assume $D\Nm{v} \equiv D\Nm{w} \bmod 2^e$. We find that
$$E_u = 2^e \chi_2(a) \epx{- \frac{D' D\Nm{w} ab}{2^e}}$$
is independent of $u$ (if its parameter satisfies the summation condition $x_1 \equiv dy_1 \bmod 2$ and $x_2 \equiv dy_2 \bmod 2$) and so in case $e = f = 2$, one has\begin{align*}
A^- &= 2^e \chi_2(a) \epx{- \frac{D' D\Nm{w} ab}{2^e}} \frac{1 + \chi_2(D\Nm{w})}{1 + \chi_2(D\Nm{v})} \sum_{\underset{x_1 = dy_1, x_2 \equiv dy_2 \bmod 2}{x_2 \bmod 2^{e-1}}}  M_{u,v}^*(\sigma)\\
&= 2^e \chi_2(a) \epx{- \frac{D' D\Nm{v} ab}{2^e}} \epx{
    \frac{\frac{1}{c'}(a(dy_1)^2 - 2(dy_1)y_1 + dy_1^2)}{16}
   + \frac{\frac{1}{D'c'} (a(dy_2)^2 - 2(dy_2)y_2 + d y_2^2)}{16}
} \times \\ &\hspace{5cm} \times \parens{1 + \epx{\frac{a c'D'}{4}}} \parens{1 + \epx{\frac{ac'}{4}}}\\
&= 2^e \chi_2(a) \epx{- \frac{D' D\Nm{v} ab}{2^e}} \epx{
    \frac{bdy_1^2}{4}
   + \frac{\frac{1}{D'} b dy_2^2}{4}
} \parens{1 + \epx{\frac{a c'}{4}}}^2 \\
&= 2^e \chi_2(a) \parens{1 + \epx{\frac{a c'}{4}}}^2 \\
&= 2^e \chi_2(a) \parens{1 + 2\epx{\frac{a c'}{4}} + \underbrace{\epx{\frac{a c'}{2}}}_{= (-1)^{ac'} = -1}}\\
&= 2^e 2 \chi_2(a) \epx{\frac{a c'}{4}} = 2 i \chi_2(a) \chi_2(ac')\\
&= 2^{e + 1} i \chi_2(c').
\end{align*}
In this case, the right hand side of \eqref{eq:value_of_A-_simplified} is $2^e (-1)^{(c'-1)/2} 2 i = 2^e 2 i \chi_2(c')$. That is because $\GS(\psi_2, D') = 2 i$ and since $D' \equiv 1 \bmod 4$, we find that $\frac{D'}{c'} \equiv c' \bmod 4$ and so
$$(-1)^{(c'-1)(D'/c' - 1)/4} = (-1)^{(c'-1)^2/4} = (-1)^{(c'-1)/2}.$$
So we are done.

\subsubsection{The case $f = 1$ and $e = 2$}

For the case $f = 1$, we have
\begin{align*}
M_{u,v}^* &= 2 \delta_{2^{e-2}D', x_2 - d y_2}^{\bmod 2} \delta_{1,x_1 - dy_1}^{\bmod 2} \;
\epx{
    \frac{\frac{1}{c'}(ax_1^2 - 2x_1y_1 + dy_1^2)}{8}
   + \frac{\frac{1}{D'c'} (ax_2^2 - 2x_2y_2 + d y_2^2)}{2^{e + 1}}
}
\end{align*}
So $M_{u,v}^* \not= 0$ only for $x_2 = dy_2 + D' \equiv dy_2 + 1 \bmod 2$ and $x_1 \equiv dy_1 + 1 \bmod 2$ (then $D\Nm{u} \equiv (dy_1+1)^2 + (dy_2+1)^2 \equiv D\Nm{v} + 2d(y_1+y_2) + 2 \bmod 4$ so $D\Nm{u} \equiv D\Nm{v} \bmod 2$ and we still have $\chi_2(D\Nm{u}) = \chi_2(D\Nm{v})$. We likewise find that with such conditions on $x_1, x_2$, the factor $E_u \not= 0$ only when $D\Nm{v} \equiv D\Nm{w} \bmod 4$: first, we need $D\Nm{u} \equiv D\Nm{w} \bmod 2$ so $D\Nm{v} \equiv D\Nm{w} \bmod 2$ and given that, the factor
\begin{align*}
&1 + \epx{\frac{D' (D\Nm{u} -  D\Nm{w}(2bc + 1) - D\Nm{w} c d)}{2^e}} \chi_2(1 + ac)\\
&= 1 + \epx{\frac{D\Nm{v} + 2d(y_1+y_2) + 2 -  D\Nm{w} - 2 D\Nm{w} c' d}{4}} (-1)^{ac'} \tsince D' \equiv 1 \bmod 4 \tand c = 2 c'\\
&= 1 - \epx{\frac{D\Nm{v} -  D\Nm{w}}{4} + \frac{d(y_1 + y_2) + 1 - D\Nm{w}c'd}{2}}\\
&= 1 + \epx{\frac{D\Nm{v} -  D\Nm{w}}{4}}
\end{align*}
in the definition of $E_u$ is zero unless $D\Nm{v} \equiv D\Nm{w} \bmod 4$. When $D\Nm{v} \equiv D\Nm{w} \bmod 4$, the value of the above factor is 2. So if $D\Nm{v} \equiv D\Nm{w} \bmod 4$ then the summation in \eqref{eq:A_minus_when_2_divides_c_reduced} is just a singleton where $x_1 = dy_1 + 1$ and $x_2 = dy_2 + 1$:
\begin{align*}
A^- &= 2^e \; \chi_2(a) \; \epx{- \frac{D' D\Nm{w} ab}{2^e}} 2 \epx{
    \frac{\frac{1}{c'}(a(dy_1 + 1)^2 - 2(dy_1 + 1)y_1 + dy_1^2)}{8}
} \times \\ & \hspace{2cm} \times \epx{
   \frac{\frac{1}{D'c'} (a(dy_2 + 1)^2 - 2(dy_2 + 1)y_2 + d y_2^2)}{2^{e + 1}}
}\\
&= 2^{e + 1} \; \chi_2(a) \; \epx{- \frac{D\Nm{v} ab}{4} +
    \frac{\frac{1}{c'}(2(ad - 1) y_1 + (ad - 1)dy_1^2 + a)}{8}} \times \\ & \hspace{2cm} \times \epx{\frac{\frac{D'}{c'} (2(ad - 1) y_2 + (ad - 1)dy_2^2 + a)}{8}}\\
&= 2^{e + 1} \; \chi_2(a) \; \epx{- \frac{D\Nm{v} ab}{4} +
    \frac{\frac{1}{c'}(2 \cdot 2 b c'  \cdot  y_1 + 2 b c' dy_1^2 + a) + \frac{D'}{c'} (2  \cdot  2 b c'  \cdot y_2 + 2 b c' dy_2^2 + a)}{8}
}\\
&= 2^{e + 1} \; \chi_2(a) \; \epx{- \frac{D\Nm{v} ab}{4} +
    \frac{4b y_1 + 2 b dy_1^2 + D' (4 b y_2 + 2 b dy_2^2) + a c' + a \frac{D'}{c'}}{8}
}\\
&= 2^{e + 1} \; \chi_2(a) \; \epx{- \frac{D\Nm{v} ab}{4} +
    \frac{b y_1 + D' b y_2}{2}
    + \frac{b dy_1^2 + D' b dy_2^2}{4}
    + \frac{a c' (D' + 1)}{8}
}\\
&= 2^{e + 1} \; \chi_2(a) \; \epx{\frac{b((d - a) D\Nm{v} + 2 (y_1 + y_2))}{4} + \frac{a c' (D' + 1)}{8}}\\
&= 2^{e + 1} \; \chi_2(a) \; \epx{\frac{c'a (D' + 1)}{8}}
\end{align*}
Note here to derive the last equation above, we reason that $a \equiv d \equiv 1 \bmod 2$ so we either have $a - d \equiv 2 \bmod 4$ whence $\frac{d - a}{2} D\Nm{v} + y_1 + y_2 \equiv D\Nm{v} + y_1 + y_2 \equiv y_1^2 + y_2^2 + y_1 + y_2 \equiv 0 \bmod 2$ and so $b((d - a) D\Nm{v} + 2 (y_1 + y_2))\equiv 0 \bmod 4$; or $a - d \equiv 0 \bmod 4$ whence $ad \equiv 1 \bmod 4$ and so we must have $2 \dv b$ which also leads to $b((d - a) D\Nm{v} + 2 (y_1 + y_2))\equiv 0 \bmod 4$ since the other factor is clearly divisible by 2 and so
$$\frac{b((d - a) D\Nm{v} + 2 (y_1 + y_2))}{4} \in \Z$$
always.

Once again, we match the RHS perfectly as in the previous case for $e = f = 2$. The difference is the involvement of the $\psi_{D'}(2)$ since the exponent $f$ is odd:
\begin{align*}
&(-1)^{\frac{(c'-1)(D'/c' - 1)}{4}} \psi_{D'}(2)^f \psi_{c'}(2^e) 2^e \GS(\chi_2; D') = (-1)^{\frac{(c'-1)(D'/c' - 1)}{4}} \psi_{D'}(2) 2^{e + 1} i
\end{align*}
and we observe that
\begin{align*}
\epx{\frac{c'a (D' + 1)}{8}} &= \epx{\frac{c'a \frac{D' + 1}{2}}{4}}\\
&= \begin{cases}
\epx{\frac{c'a}{4}} &\tif \frac{D' + 1}{2} \equiv 1 \bmod 4 \iff D' \equiv 1 \bmod 8\\
\epx{-\frac{c'a}{4}} &\tif \frac{D' + 1}{2} \equiv -1 \bmod 4 \iff D' \equiv 5 \bmod 8
\end{cases}\\
&= \psi_{D'}(2) \epx{\frac{c'a}{4}} \tsince \epx{\frac{c'a}{4}} \in \{\pm i\}
\end{align*}
and so \eqref{eq:value_of_A-_simplified} reduces to
$$\chi_2(a) \epx{\frac{c'a}{4}} = (-1)^{\frac{(c'-1)(D'/c' - 1)}{4}} i = \chi_2(c') i$$
which was established in the previous case.

\subsubsection{The case $f = 0$ and $e = 2$}

If $f = 0$ then the equation \eqref{eq:value_of_A-_simplified} reads
\begin{align*}
&\sumUModTwo \epx{- \frac{\frac{1}{D'/c'} D\Nm{u} \frac{a}{c c^*}}{2^e}} \frac{1 + \chi_2(D\Nm{w})}{1 + \chi_2(D\Nm{u})} \times \\ &\hspace{3cm} \times [\GS\parens{\chi_2; D' c} \chi_2(D\Nm{u} + D \Nm{w}) + (1 - \delta_{f_1, 0}) E_u' \GS(\chi_2; D'/c')] M_{u,v}^*\\
&\quad = \delta_{D\Nm{v}, D\Nm{w}}^{\bmod 2^e} (-1)^{\frac{(c'-1)(D'/c' - 1)}{4}} \psi_{c'}(2^e) 2^e \GS(\chi_2; D') \epx{\frac{\frac{1}{D'/c'} D\Nm{w} \frac{d}{c c^*}}{2^e}}
\end{align*}
We can cancel the Gauss sum on both sides to get an equivalent equation
\begin{align}
&\sumUModTwo \epx{- \frac{\frac{1}{D'/c'} (\cancel{2^{e-2} D' x_1^2} + \cancel{x_2^2}) a}{2^e} - \frac{\frac{1}{D'/c'} D\Nm{w} d}{2^e}} \frac{1 + \chi_2(D\Nm{w})}{1 + \chi_2(D\Nm{u})} \times \notag\\
& \hspace{2cm} \times [\chi_2(D\Nm{u} + D \Nm{w}) + (1 - \delta_{f_1, 0}) E_u']  \times \notag\\
& \hspace{2cm} \times \epx{
    \frac{\frac{1}{c'}(\cancel{ax_1^2} - 2x_1y_1 + dy_1^2)}{4}
   + \frac{\frac{1}{D'c'} (\cancel{ax_2^2} - 2x_2y_2 + d y_2^2)}{2^e}
} \notag\\
=  & \delta_{D\Nm{v}, D\Nm{w}}^{\bmod 2^e} (-1)^{\frac{(c'-1)(D'/c' - 1)}{4}} \psi_{c'}(2^e) 2^e \chi_2(c)
\label{eq:value_of_A-_simplified_for_f=0_e=2}
\end{align}
whose left hand side can be further simplified to
\begin{align*}
&(1 + \chi_2(D\Nm{w})) \epx{-\frac{\frac{d}{D'/c'} (D\Nm{v} - D\Nm{w})}{2^e}}  \times \\ &\hspace{2cm} \times \sumUModTwo \frac{\chi_2(D\Nm{u} + D \Nm{w}) + (1 - \delta_{f_1, 0}) E_u'}{1 + \chi_2(D\Nm{u})} \epx{
   -\frac{\frac{1}{c'} x_1y_1}{2}
   -\frac{\frac{1}{D'c'} x_2y_2}{2^{e-1}}
}
\end{align*}
since
$$\epx{
    \frac{\frac{1}{c'}(dy_1^2)}{4}
   + \frac{\frac{1}{D'c'}(d y_2^2)}{2^e}
} = \epx{\frac{\frac{1}{c'} (dy_1^2) 2^{e-2} + \frac{1}{D'c'}(d y_2^2)}{2^e}
} = \epx{\frac{\frac{d}{D'/c'}  [2^{e-2} D' y_1^2 + y_2^2]}{2^e}
} = \epx{\frac{\frac{d}{D'/c'} D\Nm{v}}{2^e}}.$$

If $2 \ndv D\Nm{w}$ then $\chi_2(D\Nm{u} + D \Nm{w}) + (1 - \delta_{f_1, 0}) E_u' = \chi_2(D\Nm{u} + D \Nm{w})$ is non-zero only when $D\Nm{u} \equiv 0 \bmod 2$; in other words, we need $x_1 \equiv x_2 \bmod 2$. In this case, we always have $\chi_2(D\Nm{w}) = 1$ and $\chi_2(D\Nm{u}) = 0$. So the left hand side of \eqref{eq:value_of_A-_simplified_for_f=0_e=2} is just
$$2 \epx{\frac{\frac{1}{D'/c'} (D\Nm{v} - D\Nm{w}) d}{2^e}} \sum_{\underset{\underset{2 \dv D\Nm{u}}{x_2 \bmod 2^{e-1}}}{x_1 \bmod 2}} \chi_2(D\Nm{u} + D\Nm{w})
\epx{
    \frac{\frac{1}{c'}(- 2x_1y_1)}{4}
   + \frac{\frac{1}{D'c'} (-2x_2y_2)}{2^e}
}$$
If $e = 2$ then note that $D' \equiv 1 \bmod 4$ and $D\Nm{w} \equiv 1 \bmod 4$ and so
\begin{align*}
&\sum_{\underset{\underset{2 \dv D\Nm{u}}{x_2 \bmod 2^{e-1}}}{x_1 \bmod 2}} \chi_2(D\Nm{u} + D\Nm{w})
\epx{
    \frac{\frac{1}{c'}(- 2x_1y_1)}{4}
   + \frac{\frac{1}{D'c'} (-2x_2y_2)}{2^e}
}\\
=& \sum_{\underset{x_2 = x_1}{x_1 \bmod 2}} \chi_2(2x_1^2 + D\Nm{w})
\epx{
    \frac{\frac{1}{c'}(-2x_1y_1)}{4}
   + \frac{\frac{1}{c'}(-2x_2y_2)}{4}
}\\
=&  \sum_{\underset{x_2 = x_1}{x_1 \bmod 2}}
(-1)^{\frac{2x_1^2 + D\Nm{w}-1}{2}}
\epx{
    \frac{\frac{-1}{c'}(y_1+y_2)x_1}{2}
}\\
=&  \sum_{\underset{x_2 = x_1}{x_1 \bmod 2}}
(-1)^{\frac{2x_1^2}{2}}
(-1)^{(y_1+y_2)x_1}\\
=& 1 - (-1)^{y_1 + y_2}
\end{align*}
so $A^- = 0$ unless $y_1 + y_2 \equiv 1 \bmod 2$, in which case $D\Nm{v} \equiv 1 \bmod 2$ and so $D\Nm{v} \equiv D\Nm{w} \equiv 1 \bmod 4$. When that happen the left hand side is just 4. And the remaining factor on the right hand side $(-1)^{(c'-1)(D'/c' - 1)/4} \psi_{c'}(2^e) \chi_2(c) = 1$. That is because we must have $c' \equiv D'/c' \bmod 4$ here and so $(-1)^{(c'-1)(D'/c' - 1)/4} = (-1)^{(c'-1)/2} = \chi_2(c)$.

Now let us consider $2 \dv D\Nm{w}$. We either have $e - f_1 = 0$ i.e. $D\Nm{w} \equiv 0 \bmod 4$ whence $E'_u = 1$ or $e - f_1 = 1$ i.e. $D\Nm{w} \equiv 2 \bmod 4$ whence $E'_u = (-1)^{D\Nm{u}}$. And so
\begin{align*}
\LHS_{\eqref{eq:value_of_A-_simplified_for_f=0_e=2}} &= \epx{-\frac{\frac{d}{D'/c'} (D\Nm{v} - D\Nm{w})}{2^e}} \sum_{x_1, x_2 \bmod 2} \frac{\chi_2(D\Nm{u} + D \Nm{w}) + E_u'}{1 + \chi_2(D\Nm{u})}
\epx{
   -\frac{\frac{1}{c'} x_1y_1}{2}
   -\frac{\frac{1}{D'c'} x_2y_2}{2}
}\\
&= \epx{...} \sum_{x_1, x_2 \bmod 2}  \frac{\chi_2(D\Nm{u} + D \Nm{w}) + E_u'}{1 + \chi_2(D\Nm{u})} (-1)^{x_1y_1 + x_2y_2} \tsince D' \equiv 1 \bmod 4\\
&= \epx{...} \left\{ \sum_{\underset{x_2 = x_1}{x_1 \bmod 2}} \frac{\chi_2(D\Nm{u} + D \Nm{w}) + E_u'}{1 + \chi_2(D\Nm{u})} (-1)^{x_1y_1 + x_2y_2} + \right. \\ &\hspace{4cm} \left. + \sum_{\underset{x_2 = 1 - x_1}{x_1 \bmod 2}} \frac{\chi_2(D\Nm{u} + D \Nm{w}) + E_u'}{1 + \chi_2(D\Nm{u})} (-1)^{x_1y_1 + x_2y_2} \right\}\\
&= \epx{...} \parens{\sum_{\underset{x_2 = x_1}{x_1 \bmod 2}} E_u' (-1)^{x_1(y_1 + y_2)} + \sum_{\underset{x_2 = 1 - x_1}{x_1 \bmod 2}} \frac{(-1)^{(D\Nm{u} + D \Nm{w} - 1)/2} + E_u'}{2} (-1)^{x_1(y_1 - y_2) + y_2}}\\
&= \epx{...} \parens{\sum_{\underset{x_2 = x_1}{x_1 \bmod 2}} E_u' (-1)^{x_1(y_1 + y_2)} + \frac{(-1)^{y_2}}{2} \sum_{\underset{x_2 = 1 - x_1}{x_1 \bmod 2}} [(-1)^{D \Nm{w}/2} + E_u'] (-1)^{x_1(y_1 + y_2)}}\\
&= \epx{...} \begin{cases}
\parens{\sum_{\underset{x_2 = x_1}{x_1 \bmod 2}} (-1)^{x_1(y_1 + y_2)} + (-1)^{y_2} \sum_{\underset{x_2 = 1 - x_1}{x_1 \bmod 2}} (-1)^{x_1(y_1 + y_2)}} \\ \hspace{8cm} \tif D\Nm{w} \equiv 0 \bmod 4,\\
\parens{\sum_{\underset{x_2 = x_1}{x_1 \bmod 2}} (-1)^{x_1(y_1 + y_2)} + \frac{(-1)^{y_2}}{2} \sum_{\underset{x_2 = 1 - x_1}{x_1 \bmod 2}} (- 1 - 1) (-1)^{x_1(y_1 + y_2)}} \\ \hspace{8cm} \tif D\Nm{w} \equiv 2 \bmod 4\\
\end{cases}\\
&= \epx{...} \begin{cases}
(1 + (-1)^{y_2}) \parens{\sum_{x_1 \bmod 2} (-1)^{x_1(y_1 + y_2)}} &\tif D\Nm{w} \equiv 0 \bmod 4,\\
(1 - (-1)^{y_2}) \parens{\sum_{x_1 \bmod 2} (-1)^{x_1(y_1 + y_2)}} &\tif D\Nm{w} \equiv 2 \bmod 4\\
\end{cases}\\
&= \epx{...}
(1 + (-1)^{D\Nm{w}/2} (-1)^{y_2}) \parens{1 + (-1)^{y_1 + y_2}}
\end{align*}
is non-zero only when $y_1 + y_2 \equiv 0 \bmod 2$ and $y_2 + \frac{D\Nm{w}}{2} \equiv 0 \bmod 2$. The first implies $y_1 \equiv y_2 \bmod 2$; together with the second condition, we get $D\Nm{v} \equiv D\Nm{w} \bmod 4$. So we find that LHS of \eqref{eq:value_of_A-_simplified_for_f=0_e=2} is 0 if $D\Nm{v} \not\equiv D\Nm{w} \bmod 4$ and is 4 if $D\Nm{v} \equiv D\Nm{w} \bmod 4$ in case $2 \dv D\Nm{w}$ as well.

\subsection{The map from elliptic forms to plus forms}
\label{sec:to_plus_forms}

\let\ifAllProof\iffalse\let\ifFullProof\iftrue\let\ifPersonalNote\iffalse

\unless\ifdefined\IsMainDocument

\documentclass[12pt]{amsart}
\usepackage{amsmath,amssymb,iftex,cancel,xy,xcolor,verbatim}
\usepackage{theoremref}
\usepackage[normalem]{ulem}
\usepackage[margin=1in]{geometry}

\xyoption{all}

\newtheorem{thm}{Theorem}[section]
\newtheorem{lm}[thm]{Lemma}
\newtheorem{prop}[thm]{Proposition}
\newtheorem{cor}[thm]{Corollary}
\theoremstyle{definition}
\newtheorem{defn}[thm]{Definition}
\newtheorem{remark}[thm]{Remark}
\newtheorem{example}[thm]{Example}

\numberwithin{equation}{section}

\title{Hermitian Maass lift for General Level}
\date{}
\author{Lawrence (An Hoa) Vu}

\ifPersonalNote
	\ifPDFTeX
		\usepackage{mdframed}
		\newmdtheoremenv[
		    linewidth=2pt,
		    leftmargin=0pt,
		    innerleftmargin=0.4em,
		    rightmargin=0pt,
		    innerrightmargin=0.4em,
		    innertopmargin=-5pt,
		    innerbottommargin=3pt,
		    splittopskip=\topskip,
		    splitbottomskip=0.3\topskip,
		    skipabove=0.6\topsep
		]{pnote}{Personal note}
	\else
		\definecolor{OliveGreen}{rgb}{0,0.6,0}
		\newcounter{personaln}
		\newenvironment{pnote}{\color{OliveGreen}
		\stepcounter{personaln}
		\par\bigskip\noindent{\bfseries Personal note \arabic{personaln}:}
		\par\medskip}{\par\medskip}
	\fi
\else	
	\let\pnote\comment
	\let\endpnote\endcomment
\fi

\allowdisplaybreaks

\ifPDFTeX
\else
	\usepackage{yfonts}
	\renewcommand{\mathfrak}{\textfrak}
\fi

\newcommand{\LHS}{\text{LHS}}

\newcommand{\UHP}{\mathfrak{H}}
\newcommand{\C}{\mathbb{C}}
\newcommand{\R}{\mathbb{R}}
\newcommand{\Q}{\mathbb{Q}}
\newcommand{\Z}{\mathbb{Z}}

\newcommand{\term}[1]{\emph{#1}}

\newcommand{\Maass}{Maa\ss\ }

\newcommand{\Mat}[1]{M_{#1}}
\newcommand{\GGm}{\mathsf{G_m}}
\newcommand{\GL}[1]{\mathsf{GL}_{#1}}
\newcommand{\SL}[1]{\mathsf{SL}_{#1}}
\newcommand{\GU}[1]{\mathsf{GU}(#1,#1)}
\newcommand{\U}[1]{\mathsf{U}(#1,#1)}
\newcommand{\SU}[1]{\mathsf{SU}(#1,#1)}
\newcommand{\HerMat}[1]{\mathbb{S}_{#1}} \newcommand{\Res}{\text{Res}} 

\newcommand{\GS}{\mathbf{G}} \newcommand{\aDNu}{a_u} \newcommand{\aDNw}{a_w} 

\newcommand{\MF}{\mathfrak{M}}
\newcommand{\CF}{\mathfrak{S}}
\newcommand{\JF}{\mathfrak{J}}
\newcommand{\MFDp}{\MF_{k-1}^+(D, \chi)}
\newcommand{\MFDNp}{\MF_{k-1}^+(DN, \chi)}
\newcommand{\MFDN}{\MF_{k-1}(DN, \chi)}

\newcommand{\Matx}[1]{\begin{pmatrix} #1 \end{pmatrix}}

\newcommand{\ConjTran}[1]{#1^*} \newcommand{\ConjTranInv}[1]{\widehat{#1}} 

\newcommand{\epx}[1]{\mathsf{e}\left[#1\right]}
\newcommand{\sign}[1]{\mathsf{sign}\left[#1\right]}
\newcommand{\cc}[1]{\overline{#1}}

\newcommand{\OK}{\mathfrak{o}_K} \newcommand{\DK}{\mathfrak{d}_K} \newcommand{\Mabcd}{\Matx{a & b\\c & d}}
\newcommand{\Mxyzt}{\Matx{x & y\\z & t}}
\newcommand{\val}{\mathsf{val}}

\newcommand{\parens}[1]{\left( #1 \right)}
\newcommand{\braces}[1]{\left\{ #1 \right\}}
\newcommand{\suchthat}{\;\vline\;}
\newcommand{\tif}{\text{ if }}
\newcommand{\tand}{\text{ and }}
\newcommand{\twhere}{\text{ where }}
\newcommand{\totherwise}{\text{ otherwise }}
\newcommand{\tsince}{\text{ since }}
\newcommand{\teither}{\text{ either }}
\newcommand{\tforsome}{\text{ for some }}
\newcommand{\tor}{\text{ or }}

\newcommand{\Gm}{\Gamma}
\renewcommand{\a}{\alpha}
\renewcommand{\b}{\beta}
\newcommand{\gm}{\gamma}
\newcommand{\ld}{\lambda}
\newcommand{\eps}{\varepsilon}
\renewcommand{\theta}{\vartheta}
\renewcommand{\phi}{\varphi}

\newcommand{\dv}{\;|\;}
\newcommand{\ndv}{\;\nmid\;}
\newcommand{\smD}{\sqrt{-D}}
\newcommand{\sD}{\sqrt{D}}
\renewcommand{\gcd}{}
\newcommand{\jcb}[2]{(#1 \;|\; #2)}
\newcommand{\Tr}{\mathsf{Tr}}
\newcommand{\Nm}[1]{|#1|^2}  \newcommand{\Id}{\text{Id}}
\newcommand{\Gal}{\text{Gal}}

\renewcommand{\labelenumi}{(\roman{enumi})}

\begin{document}

\fi 
\newcommand{\CFDN}{\CF_{k-1}(DN, \chi)}
\newcommand{\CFDNp}{\CF_{k-1}^+(DN, \chi)}

With Berger-Klosin's result, we have completed most of the Saito-Kurokawa lift. All that remains is to construct a surjective linear map
$$\CFDN \rightarrow \CFDNp.$$

For this linear map, we modify Ikeda's construction in \cite{Ikeda2008}. We recall the notation of \cite{Ikeda2008}. For any positive integer $M$ and prime $q$, we denote
$$M_q = q^{\val_q(M)}$$
where $\val_q$ is the $q$-adic valuation and for a Dirichlet character $\psi$ modulo $M$, we denote $\psi_q$ to be the character mod $M_q$ defined by $\psi_q(n) = \psi_q(n')$ for $\gcd(n, q) = 1$ where $n'$ is any integer such that
$$n' \equiv \begin{cases}
n &\mod M_q\\
1 &\mod M/M_q
\end{cases}$$
and $\psi_q(n) = 0$ in case $\gcd(n, q) > 1$. Let $\psi'_q$ be the character $\psi/\psi_q$. (When $\psi = \chi_K$, this notation agrees with what we used previously in section \ref{sec:hermitian_modular_forms}.)
We remark that our weight $k - 1$ corresponds to the weight $2k + 1$ in Ikeda's paper.

As in \cite{Ikeda2008}, we observe that for a primitive form $f \in \CF_{k-1}(Dm, \chi)$ and every subset $Q \subset Q_D := \{\text{prime divisors of } D\}$, there exists a primitive form $f_Q$ whose Fourier coefficients at prime $p$ is given by$$a_{f_Q}(p) = \begin{cases}
\chi_Q(p) a_f(p) &\text{ if } p \not\in Q,\\
\chi_Q'(p) \cc{a_f(p)} &\text{otherwise.}
\end{cases}$$
where
$$\chi_Q := \prod_{q \in Q} \chi_q \qquad \tand \qquad \chi'_Q := \prod_{q \in Q_D \backslash Q} \chi_q = \frac{\chi}{\chi_Q}.$$

For any primitive form $f \in \CF_{k-1}(Dm, \chi)$ where $m \dv N$ and for any $\ell$ such that $\ell m \dv N$, we define
$$f^*[\ell] := \sum_{Q \subset Q_D} \chi_Q(-\ell) f_Q$$
Let us recycle the notation $\sigma_\ell := \Matx{\ell & \\ & 1}$. Then any general form $f \in \CFDN$ can be expressed uniquely as
$$f = \sum_{i \in I} \alpha_i \; (f_i|_{k-1} \sigma_{\ell_i})$$
where $I$ is some finite indexing set, $\alpha_i \in \C^\times$, $f_i \in \CF_{k-1}(D m_i, \chi)$ are primitive forms (under appropriate independence assumption) and $m_i, \ell_i$ are natural numbers such that $m_i \ell_i \dv N$ and we set
$$f^* := \sum \alpha_i \; (f_i^*[\ell_i] \;|\; \sigma_{\ell_i}).$$

It is easy to verify that $f^* \in \CFDNp$. This is because Lemma 15.4 in \cite{Ikeda2008} (which we restate below) holds verbatim  and Corollary 15.5 follows with a minor twist:

\newcommand{\chiu}{\underline{\chi}}

\begin{lm}[Lemma 15.4 for higher level]
Let $f \in \CF_{k-1}(D m, \chi)$ be a primitive form with $\gcd(m, D) = 1$. For any $M \in \Z$, we split the prime factorization of $M$ into primes not dividing $D$, primes dividing $D$ that are not in $Q$, and primes dividing $D$ in $Q$: $$M' = \prod_{p \nmid D} p^{\val_p(M)} \qquad M_Q' = \prod_{p \dv D, p \not\in Q} p^{\val_p(M)} \qquad M_Q = \prod_{p \dv D, p \in Q} p^{\val_p(M)}$$
as in the notation of Ikeda. Then the $M$-th Fourier coefficient of $f_Q$ are related to that of $f$ by
$$a_{f_Q}(M) = a_f(M' M'_Q) \cc{a_f(M_Q)} \prod_{q \in Q} \chiu_q(M)$$
where $\chiu = \bigotimes' \chiu_q$ is the idele character corresponding to the quadratic character $\chi$.
\thlabel{lm:Fourier_coefficient_f_Q}
\end{lm}

\begin{cor}
Let $f$ be as in previous lemma. Then
\begin{align*}
a_{f^*[\ell]}(M) &= a_f(M') \prod_{q | D}(a_f(M_q) + \chi_q(\ell)  \chiu_q(-M) \cc{a_f(M_q)})\\
&= a_f(M') a_D(\ell M) \prod_{q \dv (D, M)} (a_f(M_q) + \chi_q(-1) \chi_q(\ell) \cc{a_f(M_q)} \chiu_q(M))
\end{align*}
\end{cor}

With this computation, it is evident that the form $f^*[\ell] | \sigma_\ell \in \CF_{k-1}^+(Dm\ell, \chi)$ as $a_D(\ell^2 M) = a_D(M)$ as long as $\gcd(D, \ell) = 1$.

\begin{prop}[Ikeda \cite{Ikeda2008} Proposition 15.17 for higher level]
The map $\CFDN \rightarrow \CFDNp$ where $f \mapsto f^*$ is surjective.
\end{prop}
\begin{proof}
Suppose that $g \in \CFDNp$ which we can express as
$$g = \sum_{i \in I} \alpha_i \; g_i |_{k-1} \sigma_{\ell_i}$$
with similar meaning for $I, \alpha_i$; for example, $g_i$'s are primitive forms in $\CF_{k-1}(Dm_i, \chi)$ and $m_i \ell_i \dv N$. Following Ikeda, for each subset $Q \subset Q_D$, we consider the form
$$g' = \sum \alpha_i \chi_Q(-\ell_i) (g_i)_Q | \sigma_{\ell_i}$$
and observe that $g - g'$ has its $n$-th Fourier coefficient vanishing for all $(n, D_Q) = 1$ (i.e. $q \nmid n$ for all $q \in Q$; equivalently, $n_Q = 1$). To see that, one has
\footnote{Here, we take a convention that $a_g(r) = 0$ if $r \not\in \Z$.}
\begin{align*}
a_{g - g'}(n) &= \sum \alpha_i a_{g_i}(n/\ell_i) - \sum \alpha_i \chi_Q(-\ell_i) a_{(g_i)_Q}(n/\ell_i)\\
&= \sum_{\underset{l_i | n}{i}} \alpha_i [a_{g_i}(n/\ell_i) - \chi_Q(-\ell_i) a_{g_i}((n/\ell_i)' (n/\ell_i)'_Q) \cc{a_{g_i}((n/\ell_i)_Q)} \prod_{q \in Q} \chiu_q(n/\ell_i)] \\ & \hspace{8cm} \text{ by \thref{lm:Fourier_coefficient_f_Q}}\\
&= \sum_{\underset{l_i | n}{i}} \alpha_i [a_{g_i}(n/\ell_i) - \chi_Q(-\ell_i) a_{g_i}(\underbrace{(n/\ell_i)' (n/\ell_i)'_Q}_{n/\ell_i}) \cc{a_{g_i}(\underbrace{(n/\ell_i)_Q}_{1})} \prod_{q \in Q} \chi_q(n/\ell_i)] \\ & \hspace{8cm} \text{ by Lemma 15.1 of \cite{Ikeda2008}}\\
&= \sum_{\underset{l_i | n}{i}} \alpha_i [a_{g_i}(n/\ell_i) - \chi_Q(-n) a_{g_i}(n/\ell_i)]\\
&= \sum_{\underset{l_i | n}{i}} \alpha_i [1 - \chi_Q(-n)] a_{g_i}(n/\ell_i)\\
&= [1 - \chi_Q(-n)] \left(\sum \alpha_i  a_{g_i}(n/\ell_i)\right)\\
&= [1 - \chi_Q(-n)] a_g(n)
\end{align*}
Note the fact that $(n/\ell_i)_Q = 1$ due to $(n, Q) = n_Q = 1$ whence $(n/\ell_i)' (n/\ell_i)'_Q = n/\ell_i$. Now by assumption $a_g(n) = 0$ whenever $a_D(n) = 0$, we find that $a_{g - g'}(n) = 0$ if $a_D(n) = 0$. If $(n, Q) = 1$ and $a_D(n) \not= 0$ then we must have $\chi_q(-n) \not= -1$ for all $q | D$; in other words, either $\chi_q(-n) = 0$ or $\chi_q(-n) = 1$ for all $q | D$. We can't conclude that $\chi_q(-n) = 1$ for all $q | D$ but this should be the case for all $q \in Q$ by assumption $(n, Q) = 1$. And so the factor $1 - \chi_Q(-n) = 0$ and we still have $a_{g - g'}(n) = 0$.

We have shown that $g - g' \in \CFDN$ has vanishing $n$-th Fourier coefficient for all $(n, D_Q) = 1$. By Miyake \cite{Miyake1989} Theorem 4.6.8, we find that $g - g' = 0$. Thus, we proved that
$$g = \sum \alpha_i \chi_Q(-\ell_i) (g_i)_Q | \sigma_{\ell_i}$$
for every $Q \subset Q_D$. Summing both sides over all $Q$, we get
\begin{align*}
2^{|Q_D|} g &= \sum_Q \sum \alpha_i \chi_Q(-\ell_i) (g_i)_Q | \sigma_{\ell_i}\\
&= \sum \alpha_i g_i^*[\ell_i] | \sigma_{\ell_i}
\end{align*}
so $g$ is in the image of the map.
\end{proof}

\begin{remark}
We remark that when $D$ is prime and $N = 1$, we get Krieg's remark mentioned in the introduction. For then, $Q = \{D\}$, $\CFDN$ has a basis of primitive form and we have $f_\emptyset = f$ and $f_Q = f^\rho$ for any primitive form $f$. And we find that $f^* = f_\emptyset - f_Q = f - f^\rho$.
\end{remark}

\section{Hecke equivariant}
\label{sec:hecke_equivariant}

\let\ifAllProof\iffalse\let\ifFullProof\iftrue\let\ifPersonalNote\iffalse

\unless\ifdefined\IsMainDocument

\documentclass[12pt]{amsart}
\usepackage{amsmath,amssymb,iftex,cancel,xy,xcolor,verbatim}
\usepackage{theoremref}
\usepackage[normalem]{ulem}
\usepackage[margin=1in]{geometry}

\xyoption{all}

\newtheorem{thm}{Theorem}[section]
\newtheorem{lm}[thm]{Lemma}
\newtheorem{prop}[thm]{Proposition}
\newtheorem{cor}[thm]{Corollary}
\theoremstyle{definition}
\newtheorem{defn}[thm]{Definition}
\newtheorem{remark}[thm]{Remark}
\newtheorem{example}[thm]{Example}

\numberwithin{equation}{section}

\title{Hermitian Maass lift for General Level}
\date{}
\author{Lawrence (An Hoa) Vu}

\ifPersonalNote
	\ifPDFTeX
		\usepackage{mdframed}
		\newmdtheoremenv[
		    linewidth=2pt,
		    leftmargin=0pt,
		    innerleftmargin=0.4em,
		    rightmargin=0pt,
		    innerrightmargin=0.4em,
		    innertopmargin=-5pt,
		    innerbottommargin=3pt,
		    splittopskip=\topskip,
		    splitbottomskip=0.3\topskip,
		    skipabove=0.6\topsep
		]{pnote}{Personal note}
	\else
		\definecolor{OliveGreen}{rgb}{0,0.6,0}
		\newcounter{personaln}
		\newenvironment{pnote}{\color{OliveGreen}
		\stepcounter{personaln}
		\par\bigskip\noindent{\bfseries Personal note \arabic{personaln}:}
		\par\medskip}{\par\medskip}
	\fi
\else	
	\let\pnote\comment
	\let\endpnote\endcomment
\fi

\allowdisplaybreaks

\ifPDFTeX
\else
	\usepackage{yfonts}
	\renewcommand{\mathfrak}{\textfrak}
\fi

\newcommand{\LHS}{\text{LHS}}

\newcommand{\UHP}{\mathfrak{H}}
\newcommand{\C}{\mathbb{C}}
\newcommand{\R}{\mathbb{R}}
\newcommand{\Q}{\mathbb{Q}}
\newcommand{\Z}{\mathbb{Z}}

\newcommand{\term}[1]{\emph{#1}}

\newcommand{\Maass}{Maa\ss\ }

\newcommand{\Mat}[1]{M_{#1}}
\newcommand{\GGm}{\mathsf{G_m}}
\newcommand{\GL}[1]{\mathsf{GL}_{#1}}
\newcommand{\SL}[1]{\mathsf{SL}_{#1}}
\newcommand{\GU}[1]{\mathsf{GU}(#1,#1)}
\newcommand{\U}[1]{\mathsf{U}(#1,#1)}
\newcommand{\SU}[1]{\mathsf{SU}(#1,#1)}
\newcommand{\HerMat}[1]{\mathbb{S}_{#1}} \newcommand{\Res}{\text{Res}} 

\newcommand{\GS}{\mathbf{G}} \newcommand{\aDNu}{a_u} \newcommand{\aDNw}{a_w} 

\newcommand{\MF}{\mathfrak{M}}
\newcommand{\CF}{\mathfrak{S}}
\newcommand{\JF}{\mathfrak{J}}
\newcommand{\MFDp}{\MF_{k-1}^+(D, \chi)}
\newcommand{\MFDNp}{\MF_{k-1}^+(DN, \chi)}
\newcommand{\MFDN}{\MF_{k-1}(DN, \chi)}

\newcommand{\Matx}[1]{\begin{pmatrix} #1 \end{pmatrix}}

\newcommand{\ConjTran}[1]{#1^*} \newcommand{\ConjTranInv}[1]{\widehat{#1}} 

\newcommand{\epx}[1]{\mathsf{e}\left[#1\right]}
\newcommand{\sign}[1]{\mathsf{sign}\left[#1\right]}
\newcommand{\cc}[1]{\overline{#1}}

\newcommand{\OK}{\mathfrak{o}_K} \newcommand{\DK}{\mathfrak{d}_K} \newcommand{\Mabcd}{\Matx{a & b\\c & d}}
\newcommand{\Mxyzt}{\Matx{x & y\\z & t}}
\newcommand{\val}{\mathsf{val}}

\newcommand{\parens}[1]{\left( #1 \right)}
\newcommand{\braces}[1]{\left\{ #1 \right\}}
\newcommand{\suchthat}{\;\vline\;}
\newcommand{\tif}{\text{ if }}
\newcommand{\tand}{\text{ and }}
\newcommand{\twhere}{\text{ where }}
\newcommand{\totherwise}{\text{ otherwise }}
\newcommand{\tsince}{\text{ since }}
\newcommand{\teither}{\text{ either }}
\newcommand{\tforsome}{\text{ for some }}
\newcommand{\tor}{\text{ or }}

\newcommand{\Gm}{\Gamma}
\renewcommand{\a}{\alpha}
\renewcommand{\b}{\beta}
\newcommand{\gm}{\gamma}
\newcommand{\ld}{\lambda}
\newcommand{\eps}{\varepsilon}
\renewcommand{\theta}{\vartheta}
\renewcommand{\phi}{\varphi}

\newcommand{\dv}{\;|\;}
\newcommand{\ndv}{\;\nmid\;}
\newcommand{\smD}{\sqrt{-D}}
\newcommand{\sD}{\sqrt{D}}
\renewcommand{\gcd}{}
\newcommand{\jcb}[2]{(#1 \;|\; #2)}
\newcommand{\Tr}{\mathsf{Tr}}
\newcommand{\Nm}[1]{|#1|^2}  \newcommand{\Id}{\text{Id}}
\newcommand{\Gal}{\text{Gal}}

\renewcommand{\labelenumi}{(\roman{enumi})}

\begin{document}

\fi 
Now we establish Hecke equivariance of our Maass lift for certain Hecke operators.

First, we characterize the \Maass space, similar to Andrianov's characterization of \Maass space in \cite{Andrianov1979}. The following lemma is generalization of the Lemma in \cite{Krieg1991}, Section 7 for higher level \Maass space. Intuitively, it expresses the fact that the Fourier coefficients of the hermitian modular forms in the \Maass space only depends on $\epsilon(T)$ and $\det(T)$, which is evident in \thref{defn:Maass_space}.

\begin{lm}
Suppose that $F \in \MF_{k,2}(N)$ with Fourier expansion
$$F(Z) = \sum_{T \in S_2(\Q)} c_F(T) \; \epx{\Tr(T Z)}.$$ Then $F \in \MF_k^*(N)$ if and only if there exists a function $\beta : \Z^+ \times \Z_{\geq 0} \rightarrow \C$ such that
\begin{enumerate}
\item for all $T \in S_2(\Q), T \geq 0, T \not= 0$:
$$c_F(T) = \beta\parens{\epsilon(T), D \frac{\det(T)}{\epsilon(T)^2}}$$
\item for all $d \in \Z_{\geq 0}$, $q \in \Z^+$ and all primes $p \ndv N q$:
$$(1 - p^{k-1} W) \sum_{v = 0}^{\infty} \beta(p^v q, d) W^v = \sum_{v = 0}^{\infty} \beta(q, p^{2v} d) W^v$$
holds as a formal power series in $W$; in other words,
$\beta(p^v q, d) - p^{k-1} \beta(p^{v - 1} q, d) = \beta(q, d p^{2v})$
for all $v \geq 0$.
\item $\beta(u, v) = \beta(1, v u^2)$ for all $u \dv N^\infty$ i.e. if every prime divisor of $u$ divides $N$.
\end{enumerate}
\thlabel{lm:alt_characterization_Maass_space}
\end{lm}

\begin{pnote}
\begin{proof}

Let $F$ be a form in the \Maass space and let $\alpha_F$ be the function attached to $F$ as in \thref{defn:Maass_space}. To satisfy (i), we define
$$\beta(u, v) := \sum_{\underset{d \dv u, (d, N) = 1}{d \in \Z^+}} d^{k-1} \alpha_F\parens{v (\frac{u}{d})^2}$$
and it is easy to check that (ii) and (iii) holds.

For the other direction, suppose that $\beta$ exists and we construct the appropriate function $\alpha_F$. Notice that if $T \in S_2(\Q)$ is such that $\epsilon(T) = 1$ then we expect $c_F(T) = \beta(1, D_K \det(T)) = \alpha_F(D_K \det(T))$ for such $\alpha_F$ which leads to the natural definition
$$\alpha_F(n) := \beta(1, n)$$
for all $n \in \Z^+$. We verify that with such $\alpha_F$, $F$ is a form in the \Maass space; equivalently,
$$\beta\parens{\epsilon(T), D \frac{\det(T)}{\epsilon(T)^2}} = \sum_{\underset{d \dv \epsilon(T), (d, N) = 1}{d \in \Z^+}} d^{k-1} \beta\parens{1, D \frac{\det(T)}{d^2}}$$
for all $T$. We prove a more general statement
\begin{align}
\beta(u, v) = \sum_{\underset{d \dv u, (d, N) = 1}{d \in \Z^+}} d^{k-1} \beta\parens{1, v(\frac{u}{d})^2}
\label{eq:beta_function_identity}
\end{align}
by induction on the prime decomposition of $u$. If $u \dv N^\infty$ then the equation becomes $\beta(u, v) = \beta(1, v u^2)$ which is true by (iii). If a prime $p \ndv N$ then we want to show that \eqref{eq:beta_function_identity} holds for the pair $(p^w u, v)$ should it hold for the pair $(p^h u, v)$ with $p \ndv u$ and $h < w$. By (ii), we have
\begin{align*}
\beta(p^w u, v) &= p^{k-1} \beta(p^{w-1} u, v) + \beta(u, v p^{2w})\\
&= p^{k-1} \sum_{\underset{d \dv p^{w-1} u, (d, N) = 1}{d \in \Z^+}} d^{k-1} \beta\parens{1, v(\frac{p^{w-1} u}{d})^2} + \sum_{\underset{d \dv u, (d, N) = 1}{d \in \Z^+}} d^{k-1} \beta\parens{1, v p^{2w} (\frac{u}{d})^2} \\&\hspace{4cm} \text{ by induction hypothesis}\\
&= \sum_{\underset{d \dv p^{w-1} u, (d, N) = 1}{d \in \Z^+}} (pd)^{k-1} \beta\parens{1, v(\frac{p^w u}{dp})^2} + \sum_{\underset{d \dv u, (d, N) = 1}{d \in \Z^+}} d^{k-1} \beta\parens{1, v (\frac{p^w u}{d})^2}\\
&= \sum_{\underset{d \dv p^w u, p \dv d, (d, N) = 1}{d \in \Z^+}} d^{k-1} \beta\parens{1, v(\frac{p^w u}{d})^2} + \sum_{\underset{d \dv p^w u, p \ndv d, (d, N) = 1}{d \in \Z^+}} d^{k-1} \beta\parens{1, v (\frac{p^w u}{d})^2}\\
&= \sum_{\underset{d \dv p^w u, (d, N) = 1}{d \in \Z^+}} d^{k-1} \beta\parens{1, v(\frac{p^w u}{d})^2}
\end{align*}
which is equation \eqref{eq:beta_function_identity} for the pair $(p^w u, v)$.
\end{proof}
\end{pnote}

The main utility of the lemma is the fact that the function $\beta$ relates directly to the Fourier coefficients of the form.

Now we recall the Hecke operator for Hermitian modular forms. The argument in Lemma 3.1 and Lemma 3.2 of \cite{Andrianov1995} can be used to show that the commeasurator of $\Gamma_{0,n}(N)$ in $\GU{n}(\Q)^+$ is the whole group $\GU{n}(\Q)^+$. (This was also claimed in \cite{Krieg1991}, Section 1.) Modifying Andrianov's definition for Siegel modular forms, we define the following sub-semigroup of $\GU{n}(\Q)^+$
$$\Delta_{0,n}(N) := \Gamma_{0,n}(N) \braces{g \in \GU{n}(\Q)^+ \cap \GL{2n}(K_{(N)}) \suchthat g \equiv \Matx{I_n & 0 \\ 0 & \mu(g) I_n} \bmod N} \Gamma_{0,n}(N)$$
where $K_{(N)} = \bigcap_{\mathfrak{p} \dv N} (\OK)_{\mathfrak{p}} \subset K$ denotes the subring of our imaginary quadratic field $K$ consisting of $N$-integral elements.

Then $(\Delta_{0,n}(N), \Gamma_{0,n}(N))$ is a Shimura pair and we have the associated $\Z$-ring of Hecke operators as in \cite{Miyake1989}, Chapter 2 or \cite{Shimura1971}, Chapter 3.
\begin{pnote}
I expect this to be the correct analogue for Hermitian modular forms. That is, this Hecke algebra should be commutative and give us the correct $L$-function theory as in Gritsenko.
\end{pnote}
For a rational prime $p$, let us consider the Hecke operator
$$T_p = T_p(N) := \Gamma_{0,2}(N) \Matx{I_2 & 0 \\ 0 & p I_2} \Gamma_{0,2}(N).$$

In the proof of Theorem in Section 7 of \cite{Krieg1991}, Krieg showed that when $N = 1$ and $p$ is inert in $K$, one has the decomposition
\begin{align}
\Gamma_{0,2}(1) \Matx{I_2 & 0 \\ 0 & p I_2} \Gamma_{0,2}(1) &= \bigsqcup_{(D, B) \in R} \Gm_{0,2}(1) \Matx{p \ConjTranInv{D} & B \\ 0 & D}
\label{eq:CosetDecompositionLevel1}
\end{align}
where $\ConjTranInv{D} := (\cc{D^t})^{-1} = (D^*)^{-1}$ and $R$ is the finite set of following pair of matrices $(D, B)$ where
\begin{itemize}
\item $D = I_2, B = 0$;
\item $D = p I_2, B = \Matx{\gamma & b \\ \cc{b} & \delta}$ where $b \in \OK/p\OK$ and $\gm, \delta = 1, 2, ..., p$;
\item $D = \Matx{p & 0 \\ 0 & 1}, B = \Matx{\gamma & 0 \\ 0 & 0}$ where $\gm = 1, 2, ..., p$; and
\item $D = \Matx{1 & d \\ 0 & p}, B = \Matx{0 & 0 \\ 0 & \gamma}$ where $d \in \OK/p\OK$ and $\gm = 1, 2, ..., p$.
\end{itemize}
We extend his result to higher level.

\begin{pnote}
Let $(\Gamma, \Delta)$ be a Shimura pair. To find a collection of $\beta_i \in \Gm$ such that we have the decomposition
$$\Gm \alpha \Gm = \bigsqcup_i \Gm \alpha \beta_i,$$
observe that
$\Gm \alpha \beta = \Gm \alpha \beta' \iff \alpha \beta' \beta^{-1} \alpha^{-1} \in \Gm \iff \beta' \beta^{-1} \in \alpha^{-1} \Gm \alpha \cap \Gm \iff \beta' \in (\alpha^{-1} \Gm \alpha \cap \Gm) \beta \iff \beta'$ and $\beta$ are in the same class of \emph{right} cosets of $(\alpha^{-1} \Gm \alpha \cap \Gm) \backslash \Gm$ (which is finite by definition of a Shimura pair i.e. $\alpha \in \Delta$). Thus, computing the decomposition of a double coset $\Gm \alpha \Gm$ reduces to computing the right coset representatives of $(\alpha^{-1} \Gm \alpha \cap \Gm) \backslash \Gm$. (This is Proposition 3.1 in \cite{Shimura1971}.) Note that one only care about coset $\Gm \alpha \beta$, not the actual $\beta$.

\begin{example}
In the case for elliptic modular form where $\Gm = \Gm_0(N)$, the classical operator
$$T(p) = \Gm \Matx{1 & 0 \\ 0 & p} \Gm$$
for $p \ndv N$ can be computed via the above method. Here, $\alpha = \Matx{1 & 0 \\ 0 & p}$ and from
$$\alpha^{-1} \Matx{a & b \\ c & d} \alpha = \Matx{1 & 0 \\ 0 & 1/p} \Matx{a & b \\ c & d} \Matx{1 & 0 \\ 0 & p} = \Matx{a & pb \\ c/p & d} \in \Gm \iff Np | c,$$
so one easily sees that $K = \Gm \cap \alpha^{-1} \Gm \alpha$ is precisely $\alpha^{-1} \Gm_0(Np) \alpha$ by our assumption $p \ndv N$; more concretely,
$$K = \braces{\Matx{a & pb \\ Nc & d} \in \SL{2}(\Z)}$$

Now observe that $\Matx{a & b \\ Nc & d}, \Matx{u & v \\ Nt & w} \in \Gm$ gives the same coset representative in $K \backslash \Gm$ if and only if the product
$$\Matx{a & b \\ Nc & d} \Matx{u & v \\ Nt & w}^{-1} = \Matx{a & b \\ Nc & d} \Matx{w & -v \\ -Nt & u} =  \Matx{* & -av+bu\\N * & *}$$
has the upper right entry $-av + bu$ divisible by $p$. Thus, if $p \ndv a$ then $\Matx{a & b \\ Nc & d} \sim \Matx{1 & a^{-1} b \\ 0 & 1}$; otherwise, if $p \dv a$ then $\Matx{a & b \\ Nc & d} \sim \Matx{p & \lambda \\ N & \xi}$ for any $\lambda, \xi \in \Z$ such that $p\xi - N \lambda = 1$ since in that case $-av + bu \equiv 0 \bmod p \iff bu \equiv 0 \bmod p \iff u \equiv 0 \bmod p$. Thus, we find that $K \backslash \Gm$ has $p + 1$ representatives given by the matrix $\Matx{p & \lambda \\ N & \xi}$ and $p$ matrices of the form $\Matx{1 & a \\ 0 & 1}$ where $a \in \Z/p\Z$. So we come to the decomposition
\begin{align*}
\Gm \Matx{1 & 0 \\ 0 & p} \Gm &= \Gm \alpha \Matx{p & \lambda \\ N & \xi} \sqcup \bigsqcup_{a = 0}^{p - 1} \Gm \alpha \Matx{1 & a \\ 0 & 1}\\
&= \Gm \Matx{p & \lambda \\ pN & p\xi} \sqcup \bigsqcup_{a = 0}^{p - 1} \Gm \Matx{1 & a \\ 0 & p}\\
&= \Gm \Matx{p & 0 \\0 & 1} \sqcup \bigsqcup_{a = 0}^{p - 1} \Gm \Matx{1 & a \\ 0 & p}\\
\end{align*}
for we can replace $\Matx{p & \lambda \\ pN & p\xi}$ by
$$\Matx{p \xi & -\lambda \\-N & 1} \Matx{p & \lambda \\ pN & p\xi} = \Matx{p & 0 \\0 & 1}$$
And this is what one learns in \cite{Miyake1989}. Note that the representatives are the same for all level but this fact requires some computation.

\end{example}

Let us now move on to the Hermitian modular form case where $\Gm = \Gm_{0,n}(N)$ and $\alpha = \Matx{I_n & 0 \\ 0 & p I_n}$. Just like the example above, one has
\begin{align*}
\alpha^{-1} \Matx{A & B \\ C & D} \alpha &= \Matx{I_n & 0 \\ 0 & \frac{1}{p} I_n} \Matx{A & B \\ C & D} \Matx{I_n & 0 \\ 0 & p I_n}\\
&= \Matx{A & p B \\ \frac{1}{p} C & D}
\end{align*}
and should one recall that
$$\GU{n} = \braces{\Matx{A & B \\ C & D} \suchthat A^* C = C^* A, B^* D = D^* B, D^* A - B^* C = \mu I_n}$$
explicitly then it is easy to see that $\Matx{A & p B \\ \frac{1}{p} C & D} \in \U{n}(\Q)$ whenever $\Matx{A & B \\ C & D} \in \U{n}(\Q)$ and that
$$K := \Gm \cap \alpha^{-1} \Gm \alpha = \alpha^{-1} \Gm_{0,n}(Np) \alpha$$
if $p \ndv N$. So finding the decomposition of the Hecke operator $T_p$ reduces to finding right coset representatives of $\alpha^{-1} \Gm_{0,n}(Np) \alpha \backslash \Gm_{0,n}(N)$.
One can probably employ a similar method as in the elliptic case: The matrices $\Matx{A & B \\ NC & D}, \Matx{U & V \\ NT & W} \in \Gm$ are in the same coset if and only if\footnote{
Here, I exploit the definition of $g \in \U{n}$ means $g^* J g = J$ and since $J^2 = -I$, we find that $-J g^* J g = -J^2 = I$ which shows that $g^{-1} = -J g^* J$. Hence, we have $\Matx{U & V \\ NT & W}^{-1} = \Matx{& I_n \\ -I_n & } \Matx{U^* & N T^* \\ V^* & W^*} \Matx{ & -I_n \\ I_n & } = \Matx{V^* & W^* \\ -U^* & -N T^*} \Matx{ & -I_n \\ I_n & } = \Matx{W^* & -V^* \\ -N T^* & U^*}$. It is easy to check that $\Matx{W^* & -V^* \\ -N T^* & U^*} \Matx{U & V \\ NT & W} = I_{2n}$. Note that the inverse is pretty much like inverse of $2 \times 2$ matrices; except in block and with the star.
}
\begin{align*}
\Matx{A & B \\ NC & D} \Matx{U & V \\ NT & W}^{-1} &= \Matx{A & B \\ NC & D} \Matx{W^* & -V^* \\ -N T^* & U^*}\\
&= \Matx{? & - AV^* + B U^* \\ N ? & ?}
\end{align*}
has its upper right block $- A V^* + B U^*$ divisible by $p$. Note that the latter condition \emph{does not depend on the level} so if $\rho, \rho' \in \Gm_{0, n}(N)$ are in different classes in $\alpha^{-1} \Gm_{0,n}(p) \alpha \backslash \Gm_{0,n}(1)$ then they are also in different classes in $\alpha^{-1} \Gm_{0,n}(Np) \alpha \backslash \Gm_{0,n}(N)$.
Based on this, we can probably find canonical representatives similar to the elliptic case. For instance, if $A$ is invertible mod $p$ then we can take $U = W = I_n$ and $V^*$ to be any lift of $A^{-1} B \bmod p \in M_n(\OK/p\OK)$ to $M_n(\OK)$; note that $\OK/p\OK$ is the field of $p^2$ elements by assumption that $p$ is inert in $K$. In general, we distinguish the classes based on the rank of $A$ as matrix $M_n(\OK/p\OK)$. Suppose that $A \bmod p$ is of rank $0 \leq r \leq n$. Then we can find a matrix $P \in \GL{n}(\OK/p\OK)$ such that $P^{-1} A P \equiv \Matx{I_r & 0 \\ 0 & 0} \bmod p$. Lift $P$ to $\GL{n}(\OK)$. Then we find that
\begin{align*}
A V^* \equiv B U^* \bmod p &\iff P^{-1} A P P^{-1} V^* \equiv B U^* \bmod p\\
&\iff \Matx{I_r & 0 \\ 0 & 0} P^{-1} V^* \equiv B U^* \bmod p\\
&\iff \text{the left top } r \times r \text{ block of } P^{-1} V^* \equiv B U^* \bmod p\\
&\qquad \text{ and the bottom right block of } B U^* \equiv 0 \bmod p
\end{align*}

\end{pnote}

\begin{lm}
Suppose that $p \ndv N$ is inert in $K$ and let $\xi, \lambda \in \Z$ be such that $p \xi - N \lambda = 1$. Then the set
$$R_N := \left\{
\begin{array}{r}
\Matx{
\xi p I_2 & \lambda I_2\\
N I_2 & I_2}, \Matx{
1 & 0 & \gamma & b\\
0 & 1 & \cc{b} & \delta\\
N & 0 & 1 + N \gamma & N b\\
0 & N & N\cc{b} & 1 + N\delta
}, \Matx{
1 & 0 & \gamma & 0\\
0 & \xi p & 0 & \lambda\\
0 & 0 & 1 & 0\\
0 & N & 0 & 1
}, \Matx{
\xi p & 0 & \lambda & \lambda d\\
-\cc{d} & 1 & 0 & \gamma\\
N & 0 & 1 & d\\
0 & 0 & 0 & 1
}\\
\\
\text{ where } \gamma, \delta = 1, 2, ..., p \text{ and } b, d \in \OK/p\OK
\end{array}
\right\}$$
is a complete set of representatives for
$$\underbrace{\Gm_{0,2}(N) \cap \alpha^{-1} \Gm_{0,2}(N) \alpha}_{\alpha^{-1} \Gm_{0,2}(Np) \alpha} \backslash \Gm_{0,2}(N) \qquad \text{ where } \qquad \alpha = \Matx{I_2 & 0 \\ 0 & p I_2}.$$
\thlabel{lm:right_coset_representatives}
\end{lm}

\begin{proof}
It follows from \eqref{eq:CosetDecompositionLevel1} and Shimura \cite{Shimura1971}, Proposition 3.1 that $R_N$ is a complete set of representatives for $\alpha^{-1} \Gm_{0,2}(p) \alpha \backslash \Gm_{0,2}(1)$.

\begin{pnote}
Here is the details of how I find $R_N$. Krieg's result  cited above gives
\begin{align*}
\Gm_{0,2}(1) \; \alpha \; \Gm_{0,2}(1) = \bigsqcup_{(D, B) \in S} \Gm_{0,2}(1) \; \Matx{p \ConjTranInv{D} & B \\ 0 & D}
\end{align*}
where the set $S$ consists of the following pairs of matrices:
\begin{itemize}
\item $D = I, B = 0$;
\item $D = pI, B = \Matx{\gamma & b \\ \cc{b} & \delta}$ where $b \in \OK/p\OK$ and $\gm, \delta = 1, 2, ..., p$;
\item $D = \Matx{p & 0 \\ 0 & 1}, B = \Matx{\gamma & 0 \\ 0 & 0}$ where $\gm = 1, 2, ..., p$; and
\item $D = \Matx{1 & d \\ 0 & p}, B = \Matx{0 & 0 \\ 0 & \gamma}$ where $d \in \OK/p\OK$ and $\gm = 1, 2, ..., p$.
\end{itemize}
Recall our notation that $\ConjTran{D} = \cc{D}^t$ is the conjugate transposed of $D$ and let us denote by $\ConjTranInv{D} := (\ConjTran{D})^{-1} = \ConjTran{(D^{-1})}$.

From Krieg's computation, we get the representatives for $\alpha^{-1} \Gm_{0,2}(p) \alpha \backslash \Gm_{0,2}(1)$ by taking appropriate element $\Matx{U & V \\ T & W} \in \Gm_{0,2}(1)$ for each pair $(D, B)$ such that
$$\rho(B, D, U, V, W, T) := \alpha^{-1} \Matx{U & V \\ T & W} \Matx{p \ConjTranInv{D} & B \\ 0 & D} \in \Gm_{0,2}(1)$$
and the matrices $\rho(B, D, U, V, W, T)$ will be our representatives. One has
\begin{align*}
\rho(B, D, U, V, W, T) &= \Matx{I & \\ & \frac{1}{p} I} \Matx{U & V \\ T & W} \Matx{p \ConjTranInv{D} & B \\ 0 & D}\\
&= \Matx{I & \\ & \frac{1}{p} I} \Matx{U (p \ConjTranInv{D}) & UB + V D \\ T (p \ConjTranInv{D}) & TB + WD}\\
&= \Matx{U (p\ConjTranInv{D}) & UB + V D \\ T \ConjTranInv{D} & \frac{1}{p}(TB + WD)}
\end{align*}
and observes that
\begin{enumerate}
\item $\rho(...) \in \GU{2}(\Q)$ since it is the product of three matrices in $\GU{2}(\Q)$ and in fact, $\rho(...) \in \U{2}(\Q)$ since the similitude is a group homomorphism $\GU{2}(\Q) \rightarrow \Q^\times$ and it is clear that the three matrices $\alpha^{-1}$, $\Matx{U & V \\ T & W}$ and $\Matx{p \ConjTranInv{D} & B \\ 0 & D}$ have similitude factors $\frac{1}{p}, 1$ and $p$ respectively. Thus, as long as $\rho(...)$ is integral, it is in $\U{2}(\Z)$.

\item in all possibilities of $(D, B) \in S$, the matrices $p \ConjTranInv{D}$ and $B, D$ are integral; so the top row of $\rho$ (i.e. the matrices $U (p \ConjTranInv{D})$ and $U B + V D$) is integral.

\item $\ConjTranInv{D}$ only has power of $p$ in the denominator; so as long as $T \equiv 0 \bmod N$ i.e. $\Matx{U & V \\ T & W} \in \Gm_{0,2}(N)$ then $\rho(...) \in \Gm_{0,2}(N)$ if it is integral by our assumption that $p \ndv N$.
\end{enumerate}

Thus, for each cases of $(D, B) \in S$:
\begin{itemize}
\item $D = I, B = 0 \Rightarrow $ only need $\frac{1}{p} W \in M_2(\OK)$ so we take $T = N I, W = p I$, $U = \xi I$ and $V = \lambda I$ where $\lambda, \xi \in \Z$ such that $1 = p \xi - N \lambda$. The representative we get is
$$\Matx{\xi p I & \lambda I \\ N I & I}$$

\item $D = pI, B = \Matx{\gamma & b \\ \cc{b} & \delta} \Rightarrow \ConjTranInv{D} = \frac{1}{p} I$ and we need $\frac{1}{p} T$ and $\frac{1}{p} T B + W$ being integrals; which reduces to simply $T \equiv 0 \bmod p$. In this case, we can simply take $T = N p I$ and fill up the matrix with $U = W = I$ and $V = 0$. The representatives are
$$\Matx{I & B \\ N I & N B + I} = \Matx{
1 & 0 & \gamma & b\\
0 & 1 & \cc{b} & \delta\\
N & 0 & 1 + N \gamma & N b\\
0 & N & N\cc{b} & 1 + N\delta
}$$

\item $D = \Matx{p & 0 \\ 0 & 1}, B = \Matx{\gamma & 0 \\ 0 & 0} \Rightarrow \ConjTranInv{D} = \Matx{\frac{1}{p} & 0 \\ 0 & 1}$ so we need the first column of $T = (t_{ij})$ to be divisible by $p$ and
\begin{align*}
\frac{1}{p}(TB + WD) &= \frac{1}{p} \Matx{t_{11} & t_{12} \\ t_{21} & t_{22}} \Matx{\gamma & 0 \\ 0 & 0} + W \Matx{1 & 0 \\ 0 & \frac{1}{p}}\\
&= \Matx{\frac{1}{p} \gamma t_{11} & 0 \\ \frac{1}{p} \gamma t_{21} & 0} + \Matx{w_{11} & \frac{1}{p} w_{12} \\ w_{21} & \frac{1}{p} w_{22}}\\
&= \Matx{\frac{1}{p} \gamma t_{11} + w_{11} & \frac{1}{p} w_{12} \\ \frac{1}{p} \gamma t_{21} + w_{21} & \frac{1}{p} w_{22}}
\end{align*}
to be integral. So all we need is that $t_{11}, t_{21}, w_{12}, w_{22}$ are divisible by $p$. The matrix
$$\Matx{
1 & 0 & 0 & 0\\
0 & \xi & 0 & \lambda\\
0 & 0 & 1 & 0\\
0 & N & 0 & p
}$$
fits the bill. (Recall that $1 = p \xi - N \lambda$.) The corresponding representatives are
\begin{align*}
\rho &= \Matx{U (p\ConjTranInv{D}) & UB + V D \\ T \ConjTranInv{D} & \frac{1}{p}(TB + WD)}\\
&= \Matx{\Matx{1 & 0 \\ 0 & \xi} \Matx{1 & 0 \\ 0 & p} & \Matx{1 & 0 \\ 0 & \xi} \Matx{\gamma & 0 \\ 0 & 0} + \Matx{0 & 0 \\ 0 & \lambda} \Matx{p & 0 \\ 0 & 1} \\ \Matx{0 & 0 \\ 0 & N} \Matx{\frac{1}{p} & 0 \\ 0 & 1} & \frac{1}{p}(\Matx{0 & 0 \\ 0 & N} \Matx{\gamma & 0 \\ 0 & 0} + \Matx{1 & 0 \\ 0 & p} \Matx{p & 0 \\ 0 & 1})}\\
&= \Matx{
\Matx{1 & 0 \\ 0 & \xi p} & \Matx{\gamma & 0 \\ 0 & \lambda}\\
\Matx{0 & 0 \\ 0 & N} & I}\\
&= \Matx{
1 & 0 & \gamma & 0\\
0 & \xi p & 0 & \lambda\\
0 & 0 & 1 & 0\\
0 & N & 0 & 1
}
\end{align*}

\item $D = \Matx{1 & d \\ 0 & p}, B = \Matx{0 & 0 \\ 0 & \gamma} \Rightarrow \ConjTranInv{D} = \frac{1}{p} \Matx{p & 0 \\ -\cc{d} & 1}$ and we need \begin{align*}
T \frac{1}{p} \Matx{p & 0 \\ -\cc{d} & 1} &= \Matx{t_{11} & t_{12} \\ t_{21} & t_{22}} \Matx{1 & 0 \\ -\cc{d}/p & 1/p}\\
&= \Matx{t_{11} - \frac{t_{12} \cc{d}}{p} & \frac{t_{12}}{p} \\ t_{21} - \frac{t_{22} \cc{d}}{p} & \frac{t_{22}}{p}}
\end{align*}
as well as
\begin{align*}
\frac{1}{p} (T\Matx{0 & 0 \\ 0 & \gamma} + W\Matx{1 & d \\ 0 & p}) &= \frac{1}{p} \Matx{0 & \gamma t_{12} \\ 0 & \gamma t_{22}} + \frac{1}{p} \Matx{w_{11} & w_{11}d + w_{12}p \\ w_{21} & w_{21}d + w_{22}p}\\
&= \frac{1}{p} \Matx{w_{11} & \gamma t_{12} + w_{11}d + w_{12}p \\ w_{21} & \gamma t_{22} + w_{21}d + w_{22}p}
\end{align*}
to be integral; which is evidently equivalent to $t_{12}, t_{22}, w_{11}, w_{21}$ is divisible by $p$. The matrix
$$\Matx{
\xi & 0 & \lambda & 0\\
0 & 1 & 0 & 0\\
N & 0 & p & 0\\
0 & 0 & 0 & 1
}$$
should work and we get the remaining set of representatives
\begin{align*}
\rho &= \Matx{U (p\ConjTranInv{D}) & UB + V D \\ T \ConjTranInv{D} & \frac{1}{p}(TB + WD)}\\
&= \Matx{\Matx{\xi & 0 \\ 0 & 1} \Matx{p & 0 \\ -\cc{d} & 1} & \Matx{\xi & 0 \\ 0 & 1} \Matx{0 & 0 \\ 0 & \gamma} + \Matx{\lambda & 0 \\ 0 & 0} \Matx{1 & d \\ 0 & p} \\ \Matx{N & 0 \\ 0 & 0} \frac{1}{p} \Matx{p & 0 \\ -\cc{d} & 1} & \frac{1}{p}(\Matx{N & 0 \\ 0 & 0} \Matx{0 & 0 \\ 0 & \gamma} + \Matx{p & 0 \\ 0 & 1} \Matx{1 & d \\ 0 & p})}\\
&= \Matx{\Matx{\xi p & 0 \\ -\cc{d} & 1} & \Matx{\lambda & \lambda d \\ 0 & \gamma} \\ \Matx{N & 0 \\ 0 & 0} & \Matx{1 & d \\ 0 & 1}}\\
&= \Matx{
\xi p & 0 & \lambda & \lambda d\\
-\cc{d} & 1 & 0 & \gamma\\
N & 0 & 1 & d\\
0 & 0 & 0 & 1
}
\end{align*}
\end{itemize}
\end{pnote}

These representatives for $\alpha^{-1} \Gm_{0,2}(p) \alpha \backslash \Gm_{0,2}(1)$ must be thus distinct classes in $\alpha^{-1} \Gm_{0,2}(Np) \alpha \backslash \Gm_{0,2}(N)$ as well since the condition for two matrices to be in the same class is independent of the level.

It remains to see that there are no more classes. In other words, any matrix $\Matx{A & B \\ NC & D} \in \Gm_{0,2}(N)$ lies in one of the right $\alpha^{-1} \Gm_{0,2}(Np) \alpha$-coset of some matrix in $R_N$: Let $\rho \in \Gm_{0,2}(N)$ be arbitrary and let $\rho' \in R_N$ be the element such that $\rho$ lies in the right coset $(\alpha^{-1} \Gm_{0,2}(p) \alpha) \; \rho'$. (Such an element exists because $R_N$ is a complete set of representatives for $(\alpha^{-1} \Gm_{0,2}(p) \alpha) \backslash \Gm_{0,2}(1)$ that we extracted from Krieg.) Then there exists a matrix $\rho'' \in \Gm_{0,2}(p)$ such that $\rho = \alpha^{-1} \rho'' \alpha \rho'$. We then have $\alpha^{-1} \rho'' \alpha = \rho (\rho')^{-1} \in \Gm_{0,2}(N)$ so that $\rho'' \in \alpha \Gm_{0,2}(N) \alpha^{-1}$. From a simple observation that
\begin{align*}
\alpha \Matx{A & B \\ NC & D} \alpha^{-1} &= \Matx{I_n & 0 \\ 0 & p I_n} \Matx{A & B \\ NC & D} \Matx{I_n & 0 \\ 0 & \frac{1}{p} I_n} = \Matx{A & \frac{1}{p} B \\ pNC & D}
\end{align*}
is in $\Gm_{0,2}(p)$ if and only if $B \equiv 0 \bmod p$ and we find $\alpha \Gm_{0,2}(N) \alpha^{-1} \cap \Gm_{0,2}(p) = \Gm_{0,2}(Np)$ under the assumption that $p \ndv N$. This proves $\rho'' \in \Gm_{0,2}(Np)$; in other words, $\rho$ is in the right $\alpha^{-1} \Gm_{0,2}(Np) \alpha$-coset of $\rho'$.

Thus, our set $R_N$ is the complete set of representatives for $\alpha^{-1} \Gm_{0,2}(Np) \alpha \backslash \Gm_{0,2}(N)$.
\end{proof}

\begin{cor}
Suppose that $p \ndv N$ is inert in $K$. We have the following decomposition
$$\Gm_{0,2}(N) \; \alpha \; \Gm_{0,2}(N) = \bigsqcup_{(D, B) \in R} \Gm_{0,2}(N) \Matx{p \ConjTranInv{D} & B \\ 0 & D}$$
\thlabel{cor:coset_decomposition_of_Tp_operator}
\end{cor}

\begin{pnote}
It follows from the previous lemma that
$$\Gm_{0,2}(N) \; \alpha \; \Gm_{0,2}(N) = \bigsqcup_{\beta \in R_N} \Gm_{0,2}(N) \alpha \beta$$
The elements $\alpha \beta$ as $\beta$ ranges through $R_N$ are unfortunately not ``upper triangular'' (i.e. having the lower left block zero) so the transformation $F |_k \; \alpha \beta$ is not easy. So we want to choose an ``upper triangular'' element in each coset $\Gm_{0,2}(N) \alpha \beta$.

\begin{proof}
In the proof of the \thref{lm:right_coset_representatives}, each representatives in $R_N$ is obtained as
$$\beta = \rho(B, D, U, V, T, W) = \alpha^{-1} \Matx{U & V\\T & W} \Matx{p \ConjTranInv{D} & B \\ 0 & D}$$
with some matrix $\Matx{U & V\\T & W} \in \Gm_{0,2}(N)$. It follows that
$$\alpha \beta = \Matx{U & V\\T & W} \Matx{p \ConjTranInv{D} & B \\ 0 & D}$$
with $(D, B)$ as described by Krieg.

[The way Krieg got the representatives is to claim (without citation) that we can find representative in ``upper triangular'' form and furthermore $D$ is upper triangular. I don't know he actually find the representatives but my guess is to use the analysis I did in previous note.]
\end{proof}
\end{pnote}

\begin{thm}
Suppose that $p$ is an inert prime in $K$ and $p \ndv N$. Then the \Maass space $\MF_k^*(N)$ is stable under $T_p$. In other words, if $F \in \MF_k^*(N)$ then $G = F |_k \; T_p$ is also in the \Maass space.
\end{thm}
\begin{proof}
Let $\beta_F$ be the function associated to $F$ by \thref{lm:alt_characterization_Maass_space}. Following \cite{Krieg1991}, Section 7, we define\footnote{There appears to be a mistake in Krieg's paper \cite{Krieg1991}, it should be $p^{4-2k}$, not $p^{6-2k}$.}
\begin{align*}
\beta_G(p^v q, r) &:= \beta_F(p^{v-1} q, r) + p^{4-2k} \beta_F(p^{v+1}q, r) \\
&\quad + \begin{cases}
p^{1-k} \beta_F(p^{v+1}q, p^{-2}r)
+ p^{3-k} \beta_F(p^{v-1} q, p^2 r) &\tif p^2 \dv r\\
p^{1-k} \beta_F(p^v q, r)
+ p^{3-k} \beta_F(p^{v-1} q, p^2 r) &\tif p \dv r, p^2 \ndv r\\
p^{1-k} (p + 1) \beta_F(p^v q, r)
+ p^{1-k} (p^2 - p) \beta_F(p^{v-1} q, p^2 r) &\tif p \ndv r\\
\end{cases}
\end{align*}
where $p \ndv q$. Then Krieg \cite{Krieg1991} already showed that $\beta_G$ satisfies (i) and (ii) in \thref{lm:alt_characterization_Maass_space} for $G$. The last property is evident from the definition of $\beta_G$ and the corresponding property (ii) and (iii) of $\beta_F$.

\begin{pnote}
To see that: As $p^v q \dv N^\infty$ if and only if $v = 0$ and $q \dv N^\infty$, we find that if $p^v q \dv N^\infty$ then
\begin{align*}
\beta_G(p^v q, r) &= p^{4-2k} \beta_F(p q, r) + \begin{cases}
p^{1-k} \beta_F(p q, p^{-2} r) &\tif p^2 \dv r\\
p^{1-k} \beta_F(q, r) &\tif p \dv r, p^2 \ndv r\\
p^{1-k} (p + 1) \beta_F(q, r) &\tif p \ndv r
\end{cases}\\
&= p^{4-2k} (p^{k-1} \beta_F(q, r) + \beta_F(q, rp^2)) + \begin{cases}
p^{1-k} (p^{k-1} \beta_F(q, p^{-2} r) + \beta_F(q, r)) &\tif p^2 \dv r\\
p^{1-k} \beta_F(q, r) &\tif p \dv r, p^2 \ndv r\\
p^{1-k} (p + 1) \beta_F(q, r) &\tif p \ndv r
\end{cases}\\
&= p^{3-k} \beta_F(1, r q^2) + p^{4-2k} \beta_F(1, rq^2p^2) + \begin{cases}
\beta_F(1, p^{-2} rq^2) + p^{1-k} \beta_F(1, rq^2) &\tif p^2 \dv r\\
p^{1-k} \beta_F(1, rq^2) &\tif p \dv r, p^2 \ndv r\\
p^{1-k} (p + 1) \beta_F(1, rq^2) &\tif p \ndv r
\end{cases}
\end{align*}
equals
\begin{align*}
\beta_G(1, rq^2) &= p^{4-2k} \beta_F(p, rq^2) + \begin{cases}
p^{1-k} \beta_F(p, p^{-2}rq^2) &\tif p^2 \dv rq^2\\
p^{1-k} \beta_F(1, rq^2) &\tif p \dv r, p^2 \ndv rq^2\\
p^{1-k} (p + 1) \beta_F(1, rq^2) &\tif p \ndv rq^2\\
\end{cases}\\
&= p^{4-2k} (p^{k-1} \beta_F(1, rq^2) + \beta_F(1, rq^2p^2)) + \begin{cases}
p^{1-k} (\beta_F(p, p^{-2}rq^2)) &\tif p^2 \dv r\\
p^{1-k} \beta_F(1, rq^2) &\tif p \dv r, p^2 \ndv r\\
p^{1-k} (p + 1) \beta_F(1, rq^2) &\tif p \ndv r\\
\end{cases}
\end{align*}
\end{pnote}

\begin{pnote}
How did Krieg work out the formula $\beta_G$ above? One works out the Fourier coefficients of $G$ from that of $F$. Let
$$F(Z) = \sum_{T \in S_2} c_F(T) \epx{\Tr\;TZ}$$
be the Fourier expansion of $F$. From \thref{cor:coset_decomposition_of_Tp_operator}, let $S$ denote the set of pairs $(D, B)$, we know that\footnote{I am following Krieg here and not put the typical normalization in $F|_k \sigma$ when $\det(\sigma) \not= 1$.}
\begin{align*}
G(Z) &= \sum_{(D, B) \in S} (F|_k \Matx{p \ConjTranInv{D} & B \\ 0 & D}) (Z)\\
&= \sum_{(D, B) \in S} \det(D)^{-k} F(p \ConjTranInv{D} Z D^{-1} + B D^{-1})\\
&= \sum_{(D, B) \in S} \det(D)^{-k} \sum_{T \in S_2} c_F(T) \epx{\Tr \; T (p \ConjTranInv{D} Z D^{-1} + B D^{-1})}\\
&= \sum_{T \in S_2} c_F(T) \sum_{(D, B) \in S} \det(D)^{-k}  \epx{\Tr \; T (p \ConjTranInv{D} Z D^{-1})} \epx{\Tr \; T B D^{-1}}\\
&= \sum_{T \in S_2} c_F(T) \sum_{(D, B) \in S} \det(D)^{-k} \epx{\Tr \; T B D^{-1}} \epx{\Tr \; p D^{-1} T \ConjTranInv{D} Z}
\end{align*}
by observing that $\Tr(XY) = \Tr(YX)$ for any two matrices $X, Y$.
And so we find that
\begin{align*}
c_G(T) &= \sum_{\underset{p D^{-1} T' \ConjTranInv{D} = T}{(D,B) \in S; T' \in S_2}} c_F(T') \det(D)^{-k} \epx{\Tr \; T' B D^{-1}}\\
&= \sum_{\underset{T' = \frac{1}{p} D T \ConjTran{D} \in S_2}{(D,B) \in S}} c_F(T') \det(D)^{-k} \epx{\Tr \; T' B D^{-1}}\\
&= \sum_{(D,B) \in S} c_F(p^{-1} D T \ConjTran{D}) \det(D)^{-k} \epx{\Tr \; p^{-1} D T \ConjTran{D} B D^{-1}}\\
&= \sum_{(D,B) \in S} c_F(p^{-1} D T \ConjTran{D}) \det(D)^{-k} \epx{\Tr \; p^{-1} T \ConjTran{D} B}
\end{align*}
where we make a convention that $c_F(T') = 0$ whenever $T' \not\in S_2$. The last formula can be continued
\begin{align*}
c_G(T) &= \underbrace{c_F\parens{\frac{T}{p}}}_{D=I, B=0}
+ \sum_{\gamma, \delta = 1}^p \sum_{b \in \OK/p\OK} c_F(pT) p^{-2k} \epx{\Tr \; \cancel{p^{-1}} T \cancel{(pI)} \Matx{\gamma & b\\\cc{b} & \delta}}\\
&\qquad + \sum_{\gamma = 1}^pc_F\parens{\underbrace{p^{-1} \Matx{p & 0\\0 & 1} T \Matx{p & 0\\0 & 1}}_{=: T^p}} p^{-k} \epx{\frac{1}{p} \Tr \; T \Matx{p & 0\\0 & 1} \Matx{\gamma & 0\\0 & 0}}\\
&\qquad + \sum_{\gamma = 1}^p \sum_{d \in \OK/p\OK} c_F\parens{\underbrace{p^{-1} \Matx{1 & d\\0 & p} T \Matx{1 & 0\\\cc{d} & p}}_{=: T^p_d}} p^{-k} \epx{\frac{1}{p} \Tr \; T \Matx{1 & 0\\\cc{d} & p} \Matx{0 & 0\\0 & \gamma}}\\
&= c_F\parens{\frac{T}{p}}
+  p^{-2k} c_F(pT) \sum_{\gamma, \delta = 1}^p \sum_{b \in \OK/p\OK} \underbrace{\epx{\Tr \; T \Matx{\gamma & b\\\cc{b} & \delta}}}_{1} + p^{-k} c_F(T^p) \sum_{\gamma = 1}^p \underbrace{\epx{\Tr \; T \Matx{\gamma & 0\\0 & 0}}}_{1}\\
&\qquad + p^{-k} \sum_{d \in \OK/p\OK} c_F(T^p_d) \sum_{\gamma = 1}^p \underbrace{\epx{\Tr \; T \Matx{0 & 0\\0 & \gamma}}}_{1}\\
&= c_F\parens{\frac{T}{p}} + p^{4-2k} c_F(pT) + p^{1-k} c_F(T^p) + p^{1-k} \sum_{d \in \OK/p\OK} c_F(T^p_d)
\end{align*}
Define
$$\epsilon_p(T) := \val_p(\epsilon(T))$$
Then we have
$$\epsilon(T) = \prod_{p \text{ prime}} p^{\epsilon_p(T)}.$$
And also\footnote{For split prime, we need to use the correct localization.}
\begin{align*}
\epsilon_p(T) &= \max\{v \in \Z_{\geq 0} \suchthat p^{-v} T \in S_2\}\\
&= \min\{\val_p(\ell), \val_p(m), \val_p(t) \}
\end{align*}
if we write
$$T = \Matx{\ell & t \\ \cc{t} & m} \qquad \text{ where } \ell, m\in \Z \tand t \in \DK.$$

For such $T$, we have
$$T^p := \Matx{p \ell & t \\ \cc{t} & \frac{m}{p}}$$
and
$$T_d^p := \Matx{\frac{\ell + d \cc{t} + \cc{d} t + |d|^2 m}{p} & t + dm \\ \cc{t} + m \cc{d} & p m}.$$
We set $\ell' = \frac{\ell}{\epsilon(T)}$, $m' = \frac{m}{\epsilon(T)}$ and $t' = \frac{t}{\epsilon(T)}$ and
$$r := r(T) = D \frac{\det(T)}{\epsilon(T)^2} = D \det\parens{\frac{T}{\epsilon(T)}}.$$
Note that by definition of $\epsilon(T)$, the prime $p$ cannot simultaneously divide all $\ell', m', t'$.

The problem with the above description is that there seems to be a dependency on the actual matrix $T$ whereas Fourier coefficients of forms in the \Maass space should only depends on $\epsilon(T)$ and $\det(T)$, not on the actual entries of $T$: For instance, we have
$$\epsilon(T^p) = \begin{cases}
p\epsilon(T) &\tif p \dv t' \tand p^2 \dv m';\\
\epsilon(T) &\tif p \dv m' \tand p \ndv t' \tor p^2 \ndv m';\\
\frac{\epsilon(T)}{p} &\tif p \ndv m'\\
\end{cases}$$
which is clearly depending on the entries of $T^p$ and cannot be distinguished from $\det(T)$ alone. One also observes that $\epsilon(T^p_d)$ is also either $p\epsilon(T), \epsilon(T)$ or $\frac{\epsilon(T)}{p}$. And $\det(T^p) = \det(T^p_d) = \det(T)$.

We consider three cases corresponding to the formula above:
\begin{itemize}
\item $p \dv t' \tand p^2 \dv m'$ then $p \ndv \ell'$ and
$\frac{\ell + d \cc{t} + \cc{d} t + |d|^2 m}{p}$ has the smallest $p$-adic valuation amongst the entries of $T^p_d$ which is $\val_p(\ell) - 1$. In other words,
$$\epsilon(T^p_d) = \frac{\epsilon(T)}{p}$$
for all $d$. Hence,
\begin{align}
c_G(T) &= \beta_F\parens{\frac{\epsilon(T)}{p}, r} + p^{4-2k} \beta_F\parens{p\epsilon(T), r} + p^{1-k} \beta_F\parens{p\epsilon(T), \frac{r}{p^2}}
+ p^{1-k} p^2 \beta_F\parens{\frac{\epsilon(T)}{p}, p^2 r}
\label{eq:FourierCoefficientFirstCase}
\end{align}
\noindent Note that in this case, we have $p^2 \dv r$.

\item $p \dv m'$ but either $p \ndv t'$ or $p^2 \ndv m'$:
\begin{itemize}
\item \item $p\epsilon(T) \dv t$ then $p^2\epsilon(T) \ndv m$ and $\val_p(\ell) = \epsilon_p(T)$ and it follows that $\val_p\parens{\frac{\ell + d \cc{t} + \cc{d} t + |d|^2 m}{p}} = \epsilon_p(T) - 1$ always. The Fourier coefficient should be
\begin{align}
c_G(T) &= \beta_F\parens{\frac{\epsilon(T)}{p}, r} + p^{4-2k} \beta_F\parens{p\epsilon(T), r} \notag\\
&\quad + p^{1-k} \beta_F\parens{\epsilon(T), r}
+ \underbrace{p^{1-k} p^2}_{p^{3-k}} \beta_F\parens{\frac{\epsilon(T)}{p}, p^2 r}
\label{eq:FourierCoefficientSecondCase}
\end{align}
Note that $p \dv r$ and $p^2 \ndv r$ in this case.

\item If $p \ndv t'$: then $\val_p(t) = \epsilon_p(T) = \val_p(t + dm)$ and $\val_p\parens{\frac{\ell + d \cc{t} + \cc{d} t + |d|^2 m}{p}} = \epsilon_p(T)$ or $\epsilon_p(T) - 1$ depending on whether $p$ divides $\ell' + d \cc{t'} + \cc{d} t' + |d|^2 m'$ or not.
\begin{align*}
p \dv \ell' + d \cc{t'} + \cc{d} t' + |d|^2 m' &\iff
d \text{ is solution to } t' X^p + \cc{t'} X + \ell' = 0
\end{align*}
in the field of $p^2$ elements. Note that conjugation is just Frobenius. There should be precisely $p$ solutions but we can't use the degree of the polynomial, write $X = a + b \omega$ where $a, b \in \Z/p\Z$ and $\omega \in \OK$ is such that $1, \omega$ is an integral basis for $\OK$ and realize that for each $b \in \Z/p\Z$, we have a unique possibility for $a$ as long as $2\Tr_{K/\Q}(t')$ is invertible mod $p$, which is the same as $p \ndv t'$ by assumption that $p$ is inert.
So in the summand in the sum over $d \in \OK/p\OK$, for $p$ values of $d$, we get $\epsilon(T^p_d) = \epsilon(T)$ and for the rest ($p^2 - p$ of them), we get $\epsilon(T^p_d) = \frac{\epsilon(T)}{p}$. Thus, the Fourier coefficient is given by
\begin{align}
c_G(T) &= \beta_F\parens{\frac{\epsilon(T)}{p}, r} + p^{4-2k} \beta_F\parens{p\epsilon(T), r} + p^{1-k} \beta_F\parens{\epsilon(T), r}
+ p^{1-k} (p^2 - p) \beta_F\parens{\frac{\epsilon(T)}{p}, p^2 r} \notag\\
&\qquad + p^{1-k} p \beta_F\parens{\epsilon(T), r} \notag\\
&= \beta_F\parens{\frac{\epsilon(T)}{p}, r} + p^{4-2k} \beta_F\parens{p\epsilon(T), r} + p^{1-k} (p^2 - p) \beta_F\parens{\frac{\epsilon(T)}{p}, p^2 r} + p^{1-k} (p + 1) \beta_F\parens{\epsilon(T), r} \label{eq:FourierCoefficientThirdCase}
\end{align}
Note that in this case $p \ndv r$.
\end{itemize}

\item $p \ndv m'$: Again, we want to identify the $\epsilon(T^p_d)$ for $d \in \OK/p\OK$. In this case, $\val_p(m) = \epsilon_p(T)$ and
\begin{align*}
\epsilon_p(T^p_d) &= \epsilon_p(T) + \min\{\val_p(\ell' + d \cc{t'} + \cc{d} t' + |d|^2 m') - 1, \val_p(t' + m' d)\}\\
&= \epsilon_p(T) + \min\{\val_p(m'\ell' + d m' \cc{t'} + m' \cc{d} t' + |d|^2 m'^2) - 1, \val_p(t' + m' d)\} \tsince p \ndv m'\\
&= \epsilon_p(T) + \min\{\val_p(m'\ell' - |t'|^2 + |t' + m' d|^2) - 1, \val_p(t' + m' d)\}\\
&= \epsilon_p(T) + \begin{cases}
1 &\tif p \dv t' + m'd \tand p^2 \dv \ell' + d \cc{t'} + \cc{d} t' + |d|^2 m'\\
0 &\tif p \dv \ell' + d \cc{t'} + \cc{d} t' + |d|^2 m' \tand p \ndv t' + m'd \tor p^2 \ndv \ell' + d \cc{t'} + \cc{d} t' + |d|^2 m'\\
-1 &\tif p \ndv \ell' + d \cc{t'} + \cc{d} t' + |d|^2 m'
\end{cases}
\end{align*}
Observe that the first case can only occur for a single $d = \frac{-t'}{m'}$ and can only happen when $p^2 \dv r$ and that in that case, $\epsilon_p(T^p_d) = \epsilon_p(T) - 1$ for the remaining $d \not= \frac{-t'}{m'}$. Thus, if $p^2 \dv r$, we also get \eqref{eq:FourierCoefficientThirdCase}.

If $p^2 \ndv r$ but $p \dv r$ then we have $\val_p(m'\ell' - |t'|^2 + |t' + m' d|^2) = 0$ only when $p \ndv t' + m' d$ which happens for $p^2 - 1$ values of $d$. Thus, we get \eqref{eq:FourierCoefficientSecondCase}.

Finally, if $p \ndv r$ then we have \eqref{eq:FourierCoefficientThirdCase} by similar counting argument.
\end{itemize}

Combining all the cases and substituting $\epsilon(T) = p^v q$ with $p \ndv q$ and $D \frac{\det(T)}{\epsilon(T)^2}$ by $r$, we define
\begin{align*}
\beta_G(p^v q, r) &:= \beta_F(p^{v-1} q, r) + p^{4-2k} \beta_F(p^{v+1}q, r) \\
&\quad + \begin{cases}
p^{1-k} \beta_F\parens{p^{v+1}q, p^{-2}r}
+ p^{3-k} \beta_F(p^{v-1} q, p^2 r) &\tif p^2 \dv r\\
p^{1-k} \beta_F\parens{p^v q, r}
+ p^{3-k} \beta_F\parens{p^{v-1} q, p^2 r} &\tif p \dv r, p^2 \ndv r\\
p^{1-k} (p + 1) \beta_F\parens{p^v q, r}
+ p^{1-k} (p^2 - p) \beta_F\parens{p^{v-1} q, p^2 r} &\tif p \ndv r\\
\end{cases}
\end{align*}
from the three formulas \eqref{eq:FourierCoefficientFirstCase}, \eqref{eq:FourierCoefficientSecondCase}, and \eqref{eq:FourierCoefficientThirdCase}. (There appears to be a mistake in Krieg's \cite{Krieg1991} paper where $p^{6-2k}$ is used in place of $p^{4-2k}$.)

\end{pnote}
\end{proof}

\begin{remark}
The \Maass space is not stable under all Hecke operators but its \emph{adelic} analogue should be stable under all Hecke operators; as illustrated by Klosin \cite{Klosin2015}, Section 5. \end{remark}

 {}
\bibliographystyle{acm}

\end{document}